\documentclass[12pt]{amsart}
\usepackage{amsmath,amssymb,latexsym,cancel,rotating}
\usepackage{graphicx,amssymb,mathrsfs,amsmath,color,fancyhdr,amsthm}
\usepackage[all]{xy}
\usepackage{verbatim} %整段注释的宏包
\usepackage[all]{xy}
\usepackage{graphicx}
\usepackage{rotating} %用于旋转对象
\usepackage{tikz-cd}

\newtheorem{theorem}{Theorem}[section]
\newtheorem{lemma}[theorem]{Lemma}
\newtheorem{corollary}[theorem]{Corollary}
\newtheorem{definition}[theorem]{Definition}
\newtheorem{proposition}[theorem]{Proposition}

\newtheorem{remark}[theorem]{Remark}

  \def\leq{\leqslant}  \def\geq{\geqslant}

\begin{document}

\title[Notes on Chevalley Groups and Root Category II]
{Notes on Chevalley Groups and Root Category \uppercase\expandafter{\romannumeral2}: Compact Lie Groups and Representations}
\thanks{The authors were partially supported by National Natural Science Foundation of China [Grant No. 12471030].}

\author[Buyan Li]{Buyan Li}
%{$^\dag$}
\address{Department of Mathematical Sciences, Tsinghua University, Beijing 100084, P. R. China}
\email{liby21@mails.tsinghua.edu.cn}
%\thanks{$^\dag$ Corresponding author}

\author[Jie Xiao]{Jie Xiao}
%{$^\dag$}
\address{School of Mathematical Sciences, Beijing Normal University, Beijing 100875, P. R. China}
\email{jxiao@bnu.edu.cn}
%\thanks{$^\dag$ Corresponding author}


\date{\today}


\bibliographystyle{abbrv}

\maketitle

\begin{abstract}
    This paper is a continuation of \cite{notes1}. 
    Using the root categories, we define the compact real forms of the complex semisimple Lie algebras, and maximal compact subgroups of the Chevalley groups over $\mathbb{C}$.
    In \cite{zform}, Lusztig used the modified quantum group $\dot{\mathbf{U}}$ and its canonical basis to obtain the reductive group and its coordinate ring $\mathbf{O}_A$, in particular the tensor product decomposition of $\mathbf{O}_A$.
    By combining these two kinds of structures, we explore in this paper how the classical theory of the compact Lie groups, such as Peter-Weyl theorem and Plancherel theorem, can be recovered completely.
\end{abstract}

\setcounter{tocdepth}{2}\tableofcontents

\section{Introduction}
Based on the Gabriel theorem for representation theory of quivers, Ringel introduced the Hall algebra, providing a tool to realize quantum groups and Lie algebras, see for instance \cite{Hallapproach}. 
Using the 2-periodic derived category of Dynkin quiver, which is called a root category, Peng and Xiao \cite{1997ROOT} gave a global construction of the complex semisimple Lie algebra $\mathfrak{g}$, and the structure constants are integers given by evalutions of Hall polynomials.
In \cite{notes1}, we constructed Chevalley groups of the root category $\mathcal{R}$.
This paper is a continuation of the previous constructions, in which we focus on the compact real form $\mathfrak{r}$ of the complex semisimple Lie algebra $\mathfrak{g}$ and maximal compact subgroup $\mathbf{K}$ of the Chevalley group $\mathbf{G}(\mathcal{R},\mathbb{C})$.
One of the advantages of using the root category to understand Lie group and its representations is that many things can be write down explicitly by notations of quiver representations and their categories, such as the Killing form of the Lie algebra, or the Cartan involution.

On the other hand, Lusztig gave the canonical basis of the positive part of the quantum group and its highest weight modules $\Lambda_{\lambda}$. Based on these, he constructed the modified quantum group $\dot{\mathbf{U}}$ and its canonical basis $\dot{\mathbf{B}}$. In \cite{zform}, Lusztig obtained the coordinate ring $\mathbf{O}_A$ of a reductive group $\mathbf{G}_A$, for any commutative ring $A$ with 1, and in particular, there is a $\mathbb{Q}$-vector space isomorphism $\mathbf{O}_{\mathbb{Q}}\cong \oplus_{\lambda\in X^+}\Lambda_{\lambda,\mathbb{Q}}\otimes \Lambda_{\lambda,\mathbb{Q}}^{\diamond}$. 
Note that using the perverse sheaves over quiver representations, \cite{FLX} and \cite{FL} gave a geometric realization of the integrable highest weight modules and their tensor products. 

For finite type cases, Lusztig introduced in \cite{LusztigB} an algebra method to construct the canonical basis $\mathbf{B}$ of the positive part of the quantum group as a set of bar-invariant elements, depending on the PBW basis which is provided by the Hall algebra of representations of quivers.
Similar parameterization also holds for the canonical basis $\dot{\mathbf{B}}$ of $\dot{\mathbf{U}}$.
In \cite{ZMH}, Xiao and Zhao pointed out that this  parameterization actually happens inside the root category $\mathcal{R}$, independent of the choice of the hereditary subcategory of $\mathcal{R}$ corresponding to the representation category of a chosen orientation of the quiver.
This observation suggested that there may be a closer connection between $\dot{\mathbf{U}}$ and $\mathcal{R}$.

Comparing these, we construct a maximal compact subgroup $(\mathbf{G}_{\mathbb{C}}^{\rho_1'})^{\circ}$ of Lusztig's reductive group $\mathbf{G}_{\mathbb{C}}$, and show that $(\mathbf{G}_{\mathbb{C}}^{\rho_1'})^{\circ}$ is the universal covering group of the maximal compact subgroup $\mathbf{K}$ obtained from the root category. 
Then the theory of compact Lie group and its representations can be deduced from root category and Lusztig's constructions. 
In particular, we consider the completion of $\mathbf{O}_{\mathbb{C}}$ to get a Hilbert space, and obtained the Peter-Weyl theorem and the Plancherel theorem for both $(\mathbf{G}_{\mathbb{C}}^{\rho_1'})^{\circ}$ and $\mathbf{K}$.

We remark here that Lusztig had pointed out in his book \cite[Chapter 29]{L1} the connection between $\dot{\mathbf{U}}$ and Peter-Weyl theorem.
The objective of this paper is just to make it clear that the classical theory for representation theory of compact Lie groups can be deduced from Lusztig theory for $\dot{\mathbf{U}}$ and $\mathbf{O}_A$.

The structure of the paper is as follows.
In Section 2, we review the construction of Chevalley groups from root category, and present the basic knowledge about compact Lie groups and Lusztig's modified quantum group $\dot{\mathbf{U}}$ and reductive group $\mathbf{G}_A$, which will be used later.
In Section 3, we first define a $\mathbb{Z}$-form $\mathfrak{u}$ of a real Lie algebra $\mathfrak{r}$ from the root category, and show that $\mathfrak{r}$ is a compact real form of $\mathfrak{g}$.
Then we calculate the exponentials of adjoint actions of the basis elements of $\mathfrak{u}$, and show that these generate a closed subgroup of $\operatorname{GL}_n(\mathbb{R})$, and obtain a maximal compact subgroup $\mathbf{K}$ of $\mathbf{G}(\mathcal{R},\mathbb{C})$.
In Section 4, we construct a group isomorphism $\rho_1'$ of $\mathbf{G}_{\mathbb{C}}$, and show that the identity component $(\mathbf{G}_{\mathbb{C}}^{\rho_1'})^{\circ}$ of $\rho_1'$-fixed points $\mathbf{G}_{\mathbb{C}}^{\rho_1'}$ is a maximal compact subgroup of $\mathbf{G}_{\mathbb{C}}$.
Then we obtain the Peter-Weyl theorem and Plancherel theorem of $(\mathbf{G}_{\mathbb{C}}^{\rho_1'})^{\circ}$.
By showing that $(\mathbf{G}_{\mathbb{C}}^{\rho_1'})^{\circ}$ is the universal covering group of $\mathbf{K}$, we also obtain these theorems for $\mathbf{K}$.

\subsection*{Conventions}
For a real or complex finite-dimensional vector space $V$, we denote its dual space by $V^*$. 
For a commutative ring $A$ with 1, and any $A$-module $V$, we set $V^{\diamond}=\operatorname{Hom}_A(V,A)$.
For a set $S$, we denote by $|S|$ the number of elements in $S$.
For any integer $n\geq 1$, $n!=1\cdot 2 \cdots n$, and $0!=1$.
Let $i=\sqrt{-1}$, if it doesn't appear in the subscript of the notations.

\section{Preliminaries}

\subsection{Root Category and Chevalley Groups}

In this subsection, we recall the construction of Lie algebras \cite{1997ROOT} and Chevalley groups \cite{notes1} arising from the root category.

Let $A$ be a finite dimensional associative representation-finite hereditary algebra over some base field $k$, and $\operatorname{mod}A$ be the category of finite dimensional $A$-modules. Let $D^b(A)$ be the bounded derived category of $\operatorname{mod}A$, which is a triangulated category with translantion functor $T$. The root category $\mathcal{R}$ of $A$ is the orbit category $D^b(A)/T^2$. 

If $k$ is a finite field, Ringel defined Hall polynomials $\varphi_{MN}^L(X)\in \mathbb{Z}[X]$ in \cite{Ringel1990},\cite{RingelLie} for $\operatorname{mod}A$, where $M,N,L$ are finite dimensional $A$-modules, such that for any conservative field extension $E$ of $k$, we have 
\begin{align*}
    \varphi_{MN}^L(|E|) = F_{M^E,N^E}^{L^E},
\end{align*} 
for some filtration number $F_{M^E,N^E}^{L^E}$.
The evalutions of the Hall polynomials at 1 are used to define the structure constants of some Hall algebra.

Peng and Xiao \cite{1997ROOT} extended the definition of Hall polynomials to the framework of $D^b(A)$ and $\mathcal{R}$.

\begin{proposition}
    \cite[Prop 3.1,3.2, Lemma 3.1]{1997ROOT}
    $D^b(A)$ (or $\mathcal{R}$) has the polynomials $\varphi_{MN}^L$ and $\varphi_{NM}^L$ for any objects $M,N,L\in D^b(A)$ (or $M,N,L\in \mathcal{R}$) with $M$ indecomposable.
    Moreover, for objects $M,N,L\in \mathcal{R}$ with $M,N$ indecomposable, if there exists a hereditary subcategory containing $M,N,L$, then the Hall polynomial $\varphi_{MN}^L$ defined for $\mathcal{R}$ coincides with that defined by Ringel.  
\end{proposition}

Let $\operatorname{ind}\mathcal{R}$ be the set of all the indecomposable objects of $\mathcal{R}$, and $\operatorname{Iso}\mathcal{R}$ be the set of isomorphism classes of objects of $\mathcal{R}$. For any object $M\in \mathcal{R}$, we denote its isomorphism class by $[M]$.

Let $\mathcal{K}$ be the Grothendieck group of $\mathcal{R}$, i.e. an abelian group defined by generators $H_{[M]}$, for all $M\in \mathcal{R}$, and subject to relations $H_{[X]}+H_{[Z]}=H_{[Y]}$, if there exists a triangle $X\rightarrow Y\rightarrow Z \rightarrow TX$ in $\mathcal{R}$.
By abuse of notations, we denote $H_M$ instead of $H_{[M]}$.

Define the symmetric Euler form of $\mathcal{R}$ by $(-|-):\mathcal{K}\times \mathcal{K} \rightarrow \mathbb{Z}$. For $X,Y\in \mathcal{R}$,
\begin{align*}
    (H_X|H_Y)=&\operatorname{dim}_k\operatorname{Hom}_{\mathcal{R}}(X,Y)-\operatorname{dim}_k\operatorname{Hom}_{\mathcal{R}}(X,TY) \\
        &+\operatorname{dim}_k\operatorname{Hom}_{\mathcal{R}}(Y,X)-\operatorname{dim}_k\operatorname{Hom}_{\mathcal{R}}(Y,TX).
\end{align*}
Define $d(X)=\operatorname{dim}_k \operatorname{End}_{\mathcal{R}}X$ for any $X\in \operatorname{ind}\mathcal{R}$.

Define $H'_X=\frac{H_X}{d(X)}$. Let $\mathcal{K}'$ be the quotient of a free abelian group generated by $H'_X$, for all $[X]\in \operatorname{Iso}\mathcal{R}$, and subject to relations $d(X)H'_X+d(Z)H'_Z=d(Y)H'_Y$, if there exists a triangle $X\rightarrow Y\rightarrow Z \rightarrow TX$ in $\mathcal{R}$.

Let $\mathcal{N}$ be a free abelian group with a basis $\{u_{[X]}\}_{X\in \operatorname{ind}\mathcal{R}}$. For simplicity, we write $u_X$ instead of $u_{[X]}$.

\begin{definition}
    Let $\mathfrak{g}_{\mathbb{Z}}=\mathcal{K}'\oplus \mathcal{N}$. 
    Define a bilinear operation $[-,-]$ on $\mathfrak{g}_{\mathbb{Z}}$ as follows:

    (1) For any $X,Y \in \operatorname{ind} \mathcal{R}$,
    \begin{equation*}
    [u_X,u_Y]=
    \begin{cases}
        \sum\limits_{[L],L \in \operatorname{ind} \mathcal{R}} \gamma^L_{XY} u_L  \qquad &\text{if} \quad Y\ncong TX, \\
        H_X'  &\text{otherwise},
    \end{cases}
    \end{equation*}
    where integers $\gamma_{XY}^L=\varphi_{XY}^L(1)-\varphi_{YX}^L(1)$ is given by evalutions of the Hall polynomials.

    (2) For any $X,Y \in \operatorname{ind} \mathcal{R}$,
    \begin{align*}
    [H_X',u_Y]=-\frac{(H_X|H_Y)}{d(X)}u_Y ,\\
    [u_Y,H_X']=-[H_X',u_Y].
    \end{align*}

    (3) For any $X,Y \in \mathcal{R}$,
    \begin{align*}
    [H_X',H_Y']=0.
    \end{align*}

    (4) Extend $[-,-]$ bilinearly to the whole $\mathfrak{g}_{\mathbb{Z}}$.
\end{definition}

We denote the integer $\frac{(H_X|H_Y)}{d(X)}$ by $A_{XY}$.

\begin{theorem}
    \cite[Thm 4.1,4.2]{1997ROOT}
    The vector space $\mathfrak{g}_{\mathbb{Z}}$ together with $[-,-]$ is a Lie algebra, and it is a $\mathbb{Z}$-form of the simple Lie algebra corresponding to the root system of $\mathcal{R}$.
\end{theorem}

Since the structure constants of $\mathfrak{g}_{\mathbb{Z}}$ are integers, one can construct Chevalley group (see the definition for Chevalley group also in Carter \cite{carter}, Steinberg \cite{Steinberg}, Geck \cite{geck}).
For any $X\in \operatorname{ind}\mathcal{R}$ and $t\in \mathbb{C}$, define $E_{[X]}(t)=exp(t \operatorname{ad}u_X)$.
Since the coefficients of the action of $E_{[X]}(t)$ on a basis of $\mathfrak{g}_{\mathbb{Z}}$ are all integers, we can extend the definition of $E_{[X]}(t)$ to any field $\mathbb{K}$.
Let $\mathfrak{g}_{\mathbb{K}}=\mathfrak{g}_{\mathbb{Z}}\otimes \mathbb{K}$, and $\mathbf{G}=\mathbf{G}(\mathcal{R},\mathbb{K})$ be the subgroup of $\operatorname{Aut}(\mathfrak{g}_{\mathbb{K}})$ generated by $E_{[X]}(t)$, for all $X\in \operatorname{ind}\mathcal{R}$ and $t\in \mathbb{K}$.
This group is called the Chevalley group of the root category $\mathcal{R}$ over field $\mathbb{K}$.

For $X,Y\in \operatorname{ind}\mathcal{R}$ such that $X\ncong Y$ and $X \ncong TY$, integers $i,j\geq 0$, we inductively define isomorphism classes $[L_{X,Y,i,j}]$ (and their representatives $L_{X,Y,i,j}$) as follows.
Let $[L_{X,Y,1,0}]=[X]$ and $[L_{X,Y,0,1}]=[Y]$. 
If $[L_{X,Y,i,j}]$ is defined, and $L_{X,Y,i,j}$ and $X$ have indecomposable extension, then let $[L_{X,Y,i+1,j}]$ be the isomorphism class of this extension. 
Otherwise, we say $[L_{X,Y,i+1,j}]$ doesn't exist.
Similarly, if  $L_{X,Y,i,j}$ and $Y$ have indecomposable extension, then let $[L_{X,Y,i,j+1}]$ be the isomorphism class of this extension.
Otherwise, we say $[L_{X,Y,i,j+1}]$ doesn't exist.
Let $p_{XY}$ be the largest integer $r$ such that $[L_{TX,Y,r,1}]$ exists, and let $q_{XY}$ be the largest integer $s$ such that $[L_{X,Y,s,1}]$ exists.

In particular, use these notations, we have 
    \begin{align*}
    E_{[X]}(t)u_Y = & u_Y+t \gamma^{L_{X,Y,1,1}}_{XY} u_{L_{X,Y,1,1}} + \frac{t^2}{2!} \gamma^{L_{X,Y,1,1}}_{XY} \gamma^{L_{X,Y,2,1}}_{X,L_{X,Y,1,1}} u_{L_{X,Y,2,1}} \\
    &+ \frac{t^3}{3!} \gamma^{L_{X,Y,1,1}}_{XY} \gamma^{L_{X,Y,2,1}}_{X,L_{X,Y,1,1}} \gamma^{L_{X,Y,3,1}}_{X,L_{X,Y,2,1}} u_{L_{X,Y,3,1}}+\cdots\\
    =& \sum\limits_{i=0}^{q_{XY}} C_{X,Y,i,1} t^i u_{L_{X,Y,i,1}},
    \end{align*}
where we denote the coefficients by $C_{X,Y,i,1}$ for $0\leq i\leq q_{XY}$.

Define automorphism $h_{[X]}(t)$ of $\mathfrak{g}_{\mathbb{K}}$ for each $X\in \operatorname{ind}\mathcal{R}$ and $t\in \mathbb{K}^{\times}$ by the action
\begin{align*}
    &h_{[X]}(t)u_{Y} = t^{A_{XY}} u_{Y}, \\
    &h_{[X]}(t)H_{S}' = H_{S}'.
\end{align*}

For $X\in \operatorname{ind}\mathcal{R}$, we use BGP reflection functor to define automorphism $n_{[X]}$ of $\mathfrak{g}_{\mathbb{K}}$ (see \cite[Section 4.1.2]{notes1} for details). 
For any $t\in \mathbb{K}^{\times}$ and $X\in \operatorname{ind}\mathcal{R}$, define $n_{[X]}(t)=h_{[X]}(t)n_{[X]}$. Then we have $n_{[X]}(t)=E_{[X]}(t) E_{[TX]}(t^{-1}) E_{[X]}(t)$, and thus $n_{[X]}(t),h_{[X]}(t)\in \mathbf{G}$.

Let $\mathbf{H}$ be the subgroup of $\mathbf{G}$ generated by $h_{[M]}(t)$, for all $M \in \operatorname{ind}\mathcal{R}, t \in \mathbb{K}^{\times}$, and let $\mathbf{N}$ be the subgroup of $\mathbf{G}$ generated by $\mathbf{H}$ and $n_{[M]},\forall M \in \operatorname{ind}\mathcal{R}$.
We have $\mathbf{H}$ is a normal subgroup of $\mathbf{N}$, and $\mathbf{N}/\mathbf{H}\cong \mathbf{W}$, where $\mathbf{W}$ is the Weyl group associated to the root system $\Phi$ of $\mathcal{R}$. 

Define a symbol $\omega_M(N)$ for each pair $M,N \in \operatorname{ind}\mathcal{R}$.
If $M \cong N$ or $M \cong TN$, let $\omega_M(N) = TN$.
Otherwise, let
\begin{equation*}
    \omega_M(N) = 
    \begin{cases}
        L_{TM,N,p_{MN}-q_{MN},1}  &  \text{if } p_{MN}-q_{MN} >0 \\
        L_{M,N,q_{MN}-p_{MN},1}  &  \text{if } p_{MN}-q_{MN} \leqslant 0
    \end{cases}
    .
\end{equation*}

Then we summarize six conjugate relations as follows, which are useful in the calculations.

\begin{lemma}
    \cite[Lemma 4.7-4.12]{notes1}

    For any $ X,Y \in \operatorname{ind}\mathcal{R},  t \in \mathbb{K}^{\times}, s \in \mathbb{K}$,

    (1) $ n_{[X]}(t) E_{[Y]}(s) n_{[X]}(t)^{-1} = E_{[\omega_X(Y)]}(\eta_{XY}t^{-A_{XY}}s)$, 
    
    where $\eta_{XY}$ is a product of some $\gamma$'s and some rational number, and it can be calculated that $\eta_{XY}=\pm 1$;

    (2) $h_{[X]}(t) E_{[Y]}(s) h_{[X]}(t)^{-1} = E_{[Y]}(t^{A_{XY}}s).$

    For any $X,Y \in \operatorname{ind}\mathcal{R},  t,s \in \mathbb{K}^{\times}$,

    (3) $ n_{[X]}(t) n_{[Y]}(s) n_{[X]}(t)^{-1} = n_{[\omega_X(Y)]}(\eta_{XY}t^{-A_{XY}}s) $;

    (4) $ n_{[X]}(t) h_{[Y]}(s) n_{[X]}(t)^{-1} = h_{[\omega_X(Y)]}(s)$;

    (5) $  h_{[X]}(t) h_{[Y]}(s) h_{[X]}(t)^{-1} = h_{[Y]}(s)$;

    (6) $ h_{[X]}(t) n_{[Y]}(s) h_{[X]}(t)^{-1} = n_{[Y]}(t^{A_{XY}}s).$
\end{lemma}

The Steinberg theorem shows that we can define $\mathbf{G}$ as an abstract group with generators and relations.
Denote the commutator $xyx^{-1}y^{-1}$ by $(x,y)$ in a group.

\begin{theorem}
    \cite[Thm 4.40]{notes1}
    For $ M\in \operatorname{ind}\mathcal{R}$ and $ t\in \mathbb{K}$, we introduce a symbol $x_{[M]}(t)$.
    Let $\overline{\mathbf{G}}$ be the abstract group generated by $x_{[M]}(t)$ and subject to the following relations:

    (1) $x_{[M]}(t_1) x_{[M]}(t_2) = x_{[M]}(t_1+t_2)$

    (2) For $M\ncong N,M\ncong TN$, 
    \begin{align*}
        (x_{[M]}(t),x_{[N]}(s)) = \prod_{\substack{L_{M,N,i,j}\in \operatorname{ind}\mathcal{R} \\ i,j>0}} x_{[L_{M,N,i,j}]}(C_{M,N,i,j}t^is^j),
    \end{align*}
    where constants $C_{M,N,i,j}$ are products of some $\gamma$'s and some rational number, see \cite[Prop 4.2]{notes1}.

    (3) $\overline{h_{[M]}}(t_1)\overline{h_{[M]}}(t_2) = \overline{h_{[M]}}(t_1t_2)$, $t_1t_2\neq 0$.

    where $\overline{h_{[M]}}(t) = \overline{n_{[M]}}(t) \overline{n_{[M]}}(-1)$ and $\overline{n_{[M]}}(t) = x_{[M]}(t)x_{[TM]}(t^{-1}) x_{[M]}(t)$.

    (4) $\overline{n_{[M]}}(t) x_{[M]}(s) \overline{n_{[M]}}(t)^{-1} = x_{[TM]}(t^{-2}s)$, $t\neq 0$.

    Let $\overline{\mathbf{Z}} $ be the centre of $\overline{\mathbf{G}}$. Then $\overline{\mathbf{G}}/\overline{\mathbf{Z}}\cong \mathbf{G}$ via the group homomorphism $\overline{\mathbf{G}}\rightarrow  \mathbf{G}$, mapping $x_{[M]}(t)$ to $E_{[M]}(t)$, $\overline{h_{[M]}}(t)$ to $h_{[M]}(t)$, and $\overline{n_{[M]}}(t)$ to $n_{[M]}(t)$.
\end{theorem}

We can write down $\overline{\mathbf{Z}} $ explicitly.
We arbitrarily fix a complete section of $\mathcal{R}$, i.e. a connected full subquiver of the AR-quiver of $\mathcal{R}$, which contains one vertex for each $\tau$-orbit.
Hence we get a hereditary subcategory $\mathcal{B}$ of $\mathcal{R}$. 
Let $\{S_1,\cdots,S_m\}$ be a set of representatives of isomorphism classes of simple objects in $\mathcal{B}$.
Then $\overline{h} = \prod\limits_{i=1}^m \overline{h_{[S_i]}}(t_i) \in \overline{\mathbf{Z}}$ if and only if 
       $ (t_1,\cdots,t_m) \in (\mathbb{K}^{\times})^m$ satisfies $\prod\limits_{i=1}^m t_i^{a_{ij}} = 1$, for all $ j=1,\cdots,m$,
    where $(a_{ij})$ is the Cartan matrix.

For any field $\mathbb{K}$, we have $\mathbf{G}$ is simple as an abstract group, except for cases $A_1(2),A_1(3),B_2(2),G_2(2)$ (the notation $A_1(2)$ means $\mathcal{R}$ is of type $A_1$ and the field $\mathbb{K}=\mathbb{F}_2$, others are similar).
When $\mathbb{K}$ is an algebraically closed field, $\mathbf{G}$ is naturally a linear algebraic group, and we have the following results.

\begin{corollary}
    \cite[Cor 4.42,4.44]{notes1}
    Assume $\mathbb{K}$ is an algebraically closed field, then $\mathbf{G}$ is a semisimple linear algebraic group, and the Lie algebra of $\mathbf{G}$ is isomorphic to $\mathfrak{g}_{\mathbb{K}}$.
\end{corollary}

\subsection{Compact Lie Groups}

In this subsection, we present some classical results about compact Lie groups. We assume $G$ is a compact Lie group and all representations are complex and finite-dimensional, except when otherwise specified.

A Haar measure is a nonzero regular Borel measure that is invariant under left and right translantions. We admit that $\mathbf{G}$ has a Haar measure, and normalize it such that $\mathbf{G}$ has volume 1. We write the integrals simply by $\int_{\mathbf{G}}\,f(x)dx $.

For a representation $\pi:\mathbf{G}\rightarrow \operatorname{GL}(V)$, we say an inner product $(,)$ on $V$ is $\mathbf{G}$-equivariant, if $(\pi(x)v,\pi(x)w)=(v,w)$ for any $v,w\in V$ and any $x\in \mathbf{G}$. Any finite-dimensional complex representation $(\pi,V)$ of the compact Lie group $\mathbf{G}$ has a $\mathbf{G}$-equivariant inner product.

\begin{definition}
    For a representation $(\pi,V)$ of $\mathbf{G}$, a function on $\mathbf{G}$ of the form $f(x)=L(\pi(x)v)$, for some $v\in V,L\in V^*$ and any $x\in \mathbf{G}$, is called a matrix coefficient.
\end{definition}
    
If we admit a $\mathbf{G}$-equivariant inner product $(,)$ on $V$, we can also write a matrix coefficient in the form $f(x)=(\pi(x)v,w)$, for some $v,w\in V$. 
It is easy to see that the matrix coefficients of $\mathbf{G}$ are continuous functions on $\mathbf{G}$. The pointwise sum or product of matrix coefficients is still a matrix coefficient.

If $f$ is a function on $\mathbf{G}$ and $x,y\in \mathbf{G}$, we define the left and right translantions of $\mathbf{G}$ on the space of functions on $\mathbf{G}$ by
\begin{align*}
    &(l(x)f)(y)=f(x^{-1}y),\\
    &(r(x)f)(y)=f(yx).
\end{align*}
These translantions give a useful criteria for determining whether a function is a matrix coefficient.

\begin{lemma}\label{lefttrans}
    Let $f$ be a function on $\mathbf{G}$. Then the following are equivalent.

    (1) The functions $\{ l(x)f |x\in \mathbf{G} \}$ span a finite-dimensional vector space.

    (2) The functions $\{ r(x)f |x\in \mathbf{G} \}$ span a finite-dimensional vector space.

    (3) The function $f$ is a matrix coefficient of a finite-dimensional representation.
\end{lemma}

\begin{proof}
    If $f$ is a matrix coefficient of some finite-dimensional representation $V$, so are $l(x)f$ and $r(x)f$. We can choose a basis of $V$ and show that all its matrix coefficients span a space of dimension no more than $(\operatorname{dim}V)^2$. Thus (3) implies (1) and (2).

    If the functions $\{ r(x)f |x\in \mathbf{G} \}$ span a finite-dimensional vector space $V$, then $(r,V)$ is a finite-dimensional representation of $\mathbf{G}$. Define $L:V\rightarrow \mathbb{C}$ by $L(h)=h(1)$ for any function $h$ in $V$. Then $f(x)=L(r(x)f)$. Thus $f$ is a matrix coefficient of $(r,V)$ and (2) implies (3). 
    Dually, we can show (1) implies (3). 
\end{proof}

If a function $f$ on $\mathbf{G}$ satisfies 
\begin{align*}
    \int_{\mathbf{G}} |f(x)|^2 dx < \infty,
\end{align*}
then we say $f$ is square-integrable (with respect to the Haar measure).
The square-integrable functions on $\mathbf{G}$ form a Hilbert space $L^2(\mathbf{G})$, and its inner product is given by 
\begin{align*}
    \langle f,g \rangle=\int_{\mathbf{G}} f(x)\overline{g(x)} dx.
\end{align*}
Obviously, all matrix coefficients are in $L^2(\mathbf{G})$.
The following theorem computes the inner product between matrix coefficients.

\begin{theorem} \label{Schur}
    (Schur Orthogonality)

    Suppose that $(\pi_1,V_1),(\pi_2,V_2)$ are two irreducible representations of $\mathbf{G}$, with $\mathbf{G}$-equivariant inner products both denoted by $(,)$. Let $ x\in \mathbf{G}$.

    (1) If $(\pi_1,V_1),(\pi_2,V_2)$ are not isomorphic, for $v_1,w_1\in V_1, v_2,w_2\in V_2$, 
    \begin{align*}
        \int_{\mathbf{G}} (\pi_1(x)w_1,v_1)\overline{(\pi_2(x)w_2,v_2)} dx = 0.
    \end{align*}

    (2) For $(\pi_1,V_1)$ and $v_1,w_1,v_2,w_2 \in V_1$, 
    \begin{align*}
        \int_{\mathbf{G}} (\pi_1(x)w_1,v_1)\overline{(\pi_1(x)w_2,v_2)} dx = \frac{(w_1,w_2)(v_2,v_1)}{\operatorname{dim}V_1}.
    \end{align*}
\end{theorem}

\begin{proof}
    See, for instance, \cite[Corollary 4.10]{Knapp}.
\end{proof}

To show the matrix coefficients form an orthogonal basis of $L^2(\mathbf{G})$, we also need some results about compact operators. We present the propositions without proof, and the reader may refer to Chapter 3,4 of \cite{Bump} for details.

Let $\mathcal{H}$ is a Hilbert space with inner product $\langle,\rangle$.
A bounded operator $T:\mathcal{H}\rightarrow \mathcal{H}$ is self-adjoint, if $\langle Tf,g \rangle = \langle f,Tg \rangle$ holds for all $f,g\in \mathcal{H}$. 
A bounded operator $T:\mathcal{H}\rightarrow \mathcal{H}$ is compact, if for any bounded sequence $\{x_1,x_2,\cdots \}$ in $\mathcal{H}$, the sequence $\{Tx_1,Tx_2, \cdots \}$ has a convergent subsequence.

\begin{theorem}\label{Spectral}
    (Spectral theorem for compact operators)

    Let $T$ be a compact self-adjoint operator on a Hilbert space $\mathcal{H}$. Let $\mathcal{N}$ be the nullspace of $T$. Then $\mathcal{N}^{\bot}$ has an orthonormal basis $\phi_i$ of eigenvectors of $T$, with $T\phi_i=\lambda_i \phi_i$, $i=1,2,\cdots$. The $\lambda$-eigenspace is finite-dimensional for each nonzero eigenvalue $\lambda$. If $\mathcal{N}^{\bot}$ is not finite-dimensional, the eigenvalues $\lambda_i \rightarrow 0$ as $i\rightarrow \infty$.
\end{theorem}

Let $C(\mathbf{G})$ be the space of continuous functions on $\mathbf{G}$. It has a multiplication given by convolution:
\begin{align*}
    (f*g)(x)=\int_{\mathbf{G}} f(y^{-1}x)g(y)dy=\int_{\mathbf{G}} f(y)g(xy^{-1})dy.
\end{align*}

\begin{proposition}
    For $\phi \in C(\mathbf{G})$, let $T_{\phi}$ be the left convolution with $\phi$. Then $T_{\phi}$ is a compact operator on $L^2(\mathbf{G})$. If $\phi(x^{-1})=\overline{\phi(x)}$ for all $x\in \mathbf{G}$, then $T_{\phi}$ is self-adjoint. 
\end{proposition}

Now we recall the definition of compact Lie algebra.
Let $\mathfrak{g}$ be a real Lie algebra, and $\operatorname{Aut}\mathfrak{g}$ be the group of real linear automorphisms of $\mathfrak{g}$.
Define $\operatorname{Int}\mathfrak{g}$ to be the analytic subgroup of $\operatorname{Aut}\mathfrak{g}$ with Lie algebra $\operatorname{ad}\mathfrak{g}$.
We say that $\mathfrak{g}$ is compact if the group $\operatorname{Int}\mathfrak{g}$ is compact.

In particular, if $\mathfrak{g}$ is the Lie algebra of a Lie group $\mathbf{G}$, and let $\operatorname{Ad}(\mathbf{G})$ be the image of the smooth homomorphism $\operatorname{Ad}:\mathbf{G}\rightarrow \operatorname{Aut}\mathfrak{g}$. Then $\operatorname{Ad}(\mathbf{G})$ has Lie algebra $\operatorname{ad}\mathfrak{g}$, and $\operatorname{Int}\mathfrak{g}$ is the identity component of $\operatorname{Ad}(\mathbf{G})$.

For any semisimple compact Lie group $\mathbf{G}$, let $\mathfrak{g}_0$ be its Lie algebra, and $\mathfrak{g}$ be its complexification. 
Fix a maximal torus $\mathbf{T}$, and let $\mathfrak{t}_0$ be the Lie algebra of $\mathbf{T}$, and $\mathfrak{t}$ be its complexification. 
There is root space decomposition 
\begin{align*}
    \mathfrak{g}=\mathfrak{t}\oplus \bigoplus_{\alpha\in \Phi(\mathfrak{g},\mathfrak{t}) }\mathfrak{g}_{\alpha}, 
\end{align*}
where $\mathfrak{g}_{\alpha}=\{X\in \mathfrak{g}\mid [H,X]=\alpha(H)X, \forall H\in \mathfrak{t}\}$.

Let $\lambda\in \mathfrak{t}^*$. 
If for any $H\in \mathfrak{t}_0$ satisfying $\operatorname{exp}H=1$, we have $\lambda(H)\in 2\pi i\mathbb{Z}$, or equivalently, if there exists a multiplicative character $\xi_{\lambda}$ of $\mathbf{T}$ with $\xi_{\lambda}(\operatorname{exp}H)=e^{\lambda(H)}$ for any $H\in \mathfrak{t}_0$, we say $\lambda$ is analytically integral.
If $\lambda $ satisfies that $\frac{2\langle \lambda,\alpha\rangle}{\langle \alpha,\alpha \rangle}\in \mathbb{Z}$ for all $\alpha\in \Phi$, we say $\lambda$ is algebraically integral.

If $\lambda$ is analytically integral, it must be algebraically integral. Conversely, we have

\begin{proposition}
    \cite[Prop 4.67]{Knapp}
    If $\mathbf{G}$ is a compact connected Lie group and $\tilde{\mathbf{G}}$ is a finite covering group, then the index of the group of analytically integral forms for $\mathbf{G}$ in the group of analytically integral forms for $\tilde{\mathbf{G}}$ equals the order of the kernel of the covering homomorphism $\tilde{\mathbf{G}}\rightarrow \mathbf{G}$.
\end{proposition}

When $\mathbf{G}$ is simply-connected, this two kinds of forms are equal.

\begin{proposition}
    \cite[Thm 5.107]{Knapp}
    Let $\mathbf{G}$ be a simply-connected compact semisimple Lie group, $\mathbf{T}$ be a maximal torus and $\mathfrak{t}$ be the complexified Lie algebra of $\mathbf{T}$.
    Then every algebraically integral member of $\mathfrak{t}^*$ is analytically integral.
\end{proposition}

Combining these two propositions, we have 

\begin{corollary}\label{analyint}
    \cite[Cor 5.108]{Knapp}
    The order of the fundamental group $\pi_1(\mathbf{G})$ of compact semisimple Lie group $\mathbf{G}$ is equal to the index of the group of analytically integral forms for $\mathbf{G}$ in the group of algebraically integral forms for $\mathbf{G}$.
\end{corollary}

\subsection{Lusztig's $\dot{\mathbf{U}}$ and Reductive Group}
In this subsection, we review Lusztig's construction of $\dot{\mathbf{U}}$ and some of the properties of its canonical basis $\dot{\mathbf{B}}$ in \cite{L1}. Then we overview the construction of the coordinate ring $\mathbf{O}_A$ of a reductive group $\mathbf{G}_A$ in \cite{zform}.

For a Cartan datum $(I,\cdot)$ of finite type, we fix a root datum $(X,Y,I,\langle,\rangle)$ of type $(I,\cdot)$. This consists of two finitely generated free abelian groups X,Y, a perfect bilinear pairing $\langle,\rangle:Y\times X\rightarrow \mathbb{Z}$, and the set $I$ with imbeddings $I\rightarrow X(i\mapsto i')$ and $I\rightarrow Y(i\mapsto i)$. Let $X^+=\{ \lambda \in X | \langle i, \lambda \rangle \in \mathbb{N}, \forall i \in I \}$.

Let $v$ be an indeterminate. For any $i\in I$, set $v_i=v^{\frac{i\cdot i}{2}}$.

Let $\mathbf{U}$ be the quantized enveloping algebra associated to this root datum, i.e. it is the associative $\mathbb{Q}(v)$-algebra with generators $E_i,F_i,K_{\mu}$, for $i\in I,\mu \in Y$, and satisfies relations:
\begin{align*}
    &K_0=1, \quad K_{\mu}K_{\mu'}=K_{\mu+\mu'}, \forall \mu,\mu'\in Y. \\
    &K_{\mu}E_i=v^{<\mu,i'>}E_iK_{\mu}, \forall i\in I,\mu\in Y.\\
    &K_{\mu}F_i=v^{-<\mu,i'>}F_iK_{\mu}, \forall i\in I,\mu\in Y.\\
    &E_iF_j-F_jE_i=\delta_{ij}\frac{\tilde{K}_i - \tilde{K}_{-i}}{v_i-v_i^{-1}}.\\
    &\sum_{p+q=1-\frac{2i\cdot j}{i\cdot i}} (-1)^q E_i^{(p)}E_jE_i^{(q)} = 0, \forall i \neq j.\\
    &\sum_{p+q=1-\frac{2i\cdot j}{i\cdot i}} (-1)^q F_i^{(p)}F_jF_i^{(q)} = 0, \forall i \neq j.
\end{align*}
Where $\tilde{K}_{\pm i}= K_{\pm (\frac{i\cdot i}{2})i}$, $E_i^{(p)}=\frac{E_i^p}{[p]_i^!}$, $F_i^{(p)}=\frac{F_i^p}{[p]_i^!}$, and $[p]_i^!= \prod_{s=1}^p \frac{v_i^s-v_i^{-s}}{v_i-v_i^{-1}}$.

Let $\mathbf{'f}$ be the free associative $\mathbb{Q}(v)$-algebra with 1 and with generators $\theta_i,i\in I$. For any $\nu=\sum_i \nu_i i\in \mathbb{N}[I]$, let $\mathbf{'f}_{\nu}$ be the $\mathbb{Q}(v)$-subspace spanned by $\theta_{i_1}\cdots \theta_{i_r}$ such that the number of each $i$ in the sequence $i_1,\cdots,i_r$ is $\nu_i$. There is a symmetric bilinear inner product $(,)$ on $\mathbf{'f}$. Let $\mathbf{f}$ be the quotient algebra of $\mathbf{'f}$ by the radical of $(,)$, and $\mathbf{f}_{\nu}$ be the image of $\mathbf{'f}_{\nu}$. The form $(,)$ on $\mathbf{'f}$ induces a symmetric bilinear form on $\mathbf{f}$, still denoted by $(,)$. There're algebra homomorphisms $\mathbf{f}\rightarrow \mathbf{U}$ ($x\mapsto x^+$) such that $E_i=\theta_i^+$, and $\mathbf{f}\rightarrow \mathbf{U}$ ($x\mapsto x^-$) such that $F_i=\theta_i^-$.  
Let $\mathbf{B}$ be the canonical basis of $\mathbf{f}$. 

Then we define the modified quantized enveloping algebra $\dot{\mathbf{U}}$ as follows.

If $\lambda',\lambda'' \in X$, set
\begin{align*}
    {}_{\lambda'} \mathbf{U}_{\lambda''} = \mathbf{U}/(\sum_{\mu\in Y}(K_{\mu}-v^{<\mu,\lambda'>})\mathbf{U} + \sum_{\mu\in Y}\mathbf{U} (K_{\mu}-v^{<\mu,\lambda''>})).
\end{align*}
Let $\dot{\mathbf{U}}=\oplus_{\lambda',\lambda''\in X}(_{\lambda'} \mathbf{U}_{\lambda''}) $.
There is a natural associative $\mathbb{Q}(v)$-algebra structure on $\dot{\mathbf{U}}$ inherited from that of $\mathbf{U}$.
Let $\pi_{\lambda',\lambda''}:\mathbf{U}\rightarrow {}_{\lambda'}\mathbf{U}_{\lambda''}$ be the canonical projection. Define elements $1_{\lambda}=\pi_{\lambda,\lambda}(1)$.

The comultiplication of $\mathbf{U}$ induces a collection of linear maps of $\dot{\mathbf{U}}$. For any $\lambda_1',\lambda_1'',\lambda_2',\lambda_2''\in X$, define linear map 
\begin{align*}
    \Delta_{\lambda_1',\lambda_1'',\lambda_2',\lambda_2''} : {}_{\lambda_1'+\lambda_2'} \mathbf{U}_{\lambda_1''+\lambda_2''} \rightarrow ({}_{\lambda_1'} \mathbf{U}_{\lambda_1''} ) \otimes ({}_{\lambda_2'} \mathbf{U}_{\lambda_2''} )
\end{align*}
such that $\Delta_{\lambda_1',\lambda_1'',\lambda_2',\lambda_2''} (\pi_{\lambda_1'+\lambda_2',\lambda_1''+\lambda_2''}(x))=(\pi_{\lambda_1',\lambda_1''}\otimes \pi_{\lambda_2',\lambda_2''})(\Delta(x))$ for all $x\in \mathbf{U}$.
This collection of linear maps is called the comultiplication of $\dot{\mathbf{U}}$, and we simply denote it by $\Delta$.

For any $\lambda',\lambda''\in X$, the antipode $S:\mathbf{U}\rightarrow \mathbf{U}$ induces a linear isomorphism ${}_{\lambda'} \mathbf{U}_{\lambda''} \rightarrow {}_{-\lambda''} \mathbf{U}_{-\lambda'}$. By taking direct sums, we obtain a linear isomorphism $S:\dot{\mathbf{U}} \rightarrow \dot{\mathbf{U}}$, satisfying $S(1_{\lambda})=1_{-\lambda}$ for all $\lambda\in X$, and $S(uxx'u')=S(u')S(x')S(x)S(u)$ for all $u,u'\in \mathbf{U}$ and $x,x'\in \dot{\mathbf{U}}$.

For any $\lambda,\lambda'\in X$, the $\mathbb{Q}$-algebra homomorphism $\bar{ }:\mathbf{U}\rightarrow \mathbf{U}$ induces a $\mathbb{Q}$-linear map $\bar{ }:{}_{\lambda} \mathbf{U}_{\lambda'} \rightarrow {}_{\lambda} \mathbf{U}_{\lambda'}$. By taking direct sums, we have a $\mathbb{Q}$-linear map $\bar{ }:\dot{\mathbf{U}}\rightarrow \dot{\mathbf{U}}$.

Let $\dot{\mathbf{B}}$ be the canonical basis of $\dot{\mathbf{U}}$ (see Chapter 25 of \cite{L1}). Let $\mathcal{A}=\mathbb{Z}[v,v^{-1}]$. 
For any $a,b,c\in \dot{\mathbf{B}}$, define $m_{a,b}^c \in \mathcal{A}$ by $ab=\sum_{c\in \dot{\mathbf{B}}} m_{a,b}^c c$. For any $\lambda_1',\lambda_1'',\lambda_2',\lambda_2''\in X$ and any $c\in \dot{\mathbf{B}} \cap ({}_{\lambda_1'+\lambda_2'} \mathbf{U}_{\lambda_1''+\lambda_2''} ),a,b\in \dot{\mathbf{B}}$, define $\hat{m}_c^{a,b}\in \mathcal{A}$ by $\Delta(c)=\sum_{a,b\in \dot{\mathbf{B}}} \hat{m}_c^{a,b} a\otimes b$. The elements $m_{a,b}^c,\hat{m}_c^{a,b}$ are called the structure constants of $\dot{\mathbf{U}}$.

Now we come to Lusztig's definition of the inner product on $\dot{\mathbf{U}}$.

Let $\rho_1:\mathbf{U}\rightarrow \mathbf{U}$ be the algebra isomorphism given by 
\begin{align*}
    \rho_1(E_i)=-v_iF_i,\quad \rho_1(F_i) = -v_i^{-1}E_i, \quad \rho_1(K_{\mu}) = K_{-\mu}.
\end{align*}

Let $\rho:\mathbf{U}\rightarrow \mathbf{U}^{op}$ be the algebra isomorphism given by the composition $\rho_1S$.

\begin{theorem} \label{udotpairing}
    \cite[Thm 26.1.2]{L1}
    There exists a unique $\mathbb{Q}(v)$-bilinear pairing $(,):\dot{\mathbf{U}}\times \dot{\mathbf{U}}\rightarrow \mathbb{Q}(v)$ such that:

    (1) $(1_{\lambda_1} x 1_{\lambda_2}, 1_{\lambda_1'}x' 1_{\lambda_2'} )=0$ for all $x,x'\in \dot{\mathbf{U}}$, unless $\lambda_1=\lambda_1',\lambda_2=\lambda_2'$;
    
    (2) $(ux,y)=(x,\rho(u)y)$ for all $x,y\in \dot{\mathbf{U}}$ and $u\in \mathbf{U}$;

    (3) $(x_1^- 1_{\lambda},x_2^{-}1_{\lambda})=(x_1,x_2)$ for all $x_1,x_2\in \mathbf{f}$ and all $\lambda$;

    (4) this bilinear pairing is symmetric.
\end{theorem}

This inner product on $\dot{\mathbf{U}}$ is a limit of inner products of tensor products of modules in the following sense.

Let $\lambda\in X^+$, and $\Lambda_{\lambda}=\mathbf{f}/(\sum_i \mathbf{f}\theta_i^{<i,\lambda>+1})$ be the corresponding highest weight $\mathbf{U}$-module with highest weight vector $\eta_{\lambda}$.
There is a unique algebra automorphism $\omega:\mathbf{U}\rightarrow \mathbf{U}$ that interchanges $E_i$ and $F_i$, and maps $K_{\mu}$ to $K_{-\mu}$.
Let ${}^{\omega} \Lambda_{\lambda}$ be the $\mathbf{U}$-module with the same underlying vector space as $\Lambda_{\lambda}$ and such that any $u\in \mathbf{U}$ acts on it by $\omega(u)$.
The vector $\eta_{\lambda}$, regarded as a vector of ${}^{\omega} \Lambda_{\lambda}$, is denoted by $\xi_{-\lambda}$.

\begin{proposition}\label{modpairing}
    \cite[Prop 19.1.2]{L1}
    For any $\lambda\in X^+$, let $\Lambda=\Lambda_{\lambda}, \eta=\eta_{\lambda}.$
    There is a unique bilinear form $(,):\Lambda \times \Lambda \rightarrow \mathbb{Q}(v)$ such that 

    (1) $(\eta,\eta)=1$,

    (2) $(ux,y)=(x,\rho(u)y)$ for all $x,y\in \Lambda$ and $u\in \mathbf{U}$.

    Moreover, this bilinear form is symmetric.

    For any $\nu\in \mathbb{N}[I]$, let $(\Lambda)_{\nu}$ be the image of $\mathbf{f}_{\nu}$. If $x\in (\Lambda)_{\nu},y\in (\Lambda)_{\mu}$ and $\nu\neq \mu$, then $(x,y)=0$.
\end{proposition}

Let $\zeta \in X$ and $\lambda,\lambda'\in X^+$ such that $\lambda'-\lambda=\zeta$. Consider the bilinear pairing $(,)_{\lambda,\lambda'}$ on ${}^\omega \Lambda_{\lambda}\otimes \Lambda_{\lambda'}$ define by $(x\otimes x',y\otimes y')_{\lambda,\lambda'}=(x,y)_{\lambda}(x',y')_{\lambda'}$, where $(,)_{\lambda'}$ on $\Lambda_{\lambda'}$ is defined as above, and $(,)_{\lambda}$ on ${}^{\omega}\Lambda_{\lambda}$ is the analogous pairing. 

\begin{proposition}
    \cite[Prop 26.2.3]{L1}
    Let $\zeta \in X$ and $\lambda,\lambda'\in X^+$ such that $\lambda'-\lambda=\zeta$.
    Let $x,y\in \dot{\mathbf{U}}1_{\zeta}$. When $\langle i,\lambda \rangle$ tends to $\infty$ for all $i$, the inner product $(x(\xi_{-\lambda}\otimes \eta_{\lambda'} ), y(\xi_{-\lambda}\otimes \eta_{\lambda'}))_{\lambda,\lambda'}\in \mathbb{Q}(v)$ converges to $(x,y)$ in $\mathbb{Q}((v^{-1}))$.
\end{proposition}

Let $\mathbf{f}_{\mathcal{A}}$ be the $\mathcal{A}$-subalgebra of $\mathbf{f}$ generated by the elements $\theta_i^{(s)}=\frac{\theta_i^s}{[s]_{i}^!}$ with $i\in I,s\in \mathbb{N}$.
The $\mathcal{A}$-submodule of $\dot{\mathbf{U}}$ spanned by the elements $x^+1_{\lambda}x'^{-}$ with $x,x'\in \mathbf{f}_{\mathcal{A}}$ coincides with the $\mathcal{A}$-submodule of $\dot{\mathbf{U}}$ spanned by the elements $x^-1_{\lambda}x'^{+}$ with $x,x'\in \mathbf{f}_{\mathcal{A}}$. We denote it by $\dot{\mathbf{U}}_{\mathcal{A}}$.

With the inner product in hand, we have the following characterization of $\pm \dot{\mathbf{B}} $.
Let $\mathbb{A}=\mathbb{Q}[[v^{-1}]]\cap \mathbb{Q}(v)$.

\begin{theorem}
    \cite[Thm 26.3.1]{L1}
    Let $\beta \in \dot{\mathbf{U}}$. Then $\beta \in \pm \dot{\mathbf{B}} $ if and only if $\beta$ satisfies the following three conditions: $\beta \in \dot{\mathbf{U}}_{\mathcal{A}}$, $\bar{\beta}=\beta$ and $(\beta,\beta)\in 1+v^{-1}\mathbb{A}$.
\end{theorem}
By 27.2.1 and 29.1.1 of \cite{L1}, we have a well-defined map $\dot{\mathbf{B}}\rightarrow X^+$. We write $\dot{\mathbf{B}}[\lambda]$ for the fibre of this map at $\lambda$ and obtain a partition $\dot{\mathbf{B}}=\sqcup_{\lambda\in X^+} \dot{\mathbf{B}}[\lambda] $.
For any $\lambda\in X^+$, we denote by $\dot{\mathbf{U}}[\geq \lambda]$ (resp. $\dot{\mathbf{U}}[> \lambda]$) the $\mathbb{Q}(v)$-subspace of $\dot{\mathbf{U}}$ spanned by $\sqcup_{\lambda' \geq \lambda}\dot{\mathbf{B}}[\lambda']$ (resp. $\sqcup_{\lambda' > \lambda}\dot{\mathbf{B}}[\lambda']$).

\begin{lemma}\label{PWlemma}
    \cite[Lemma 29.1.3]{L1}
    The following for an element $u\in \dot{\mathbf{U}}$ are equivalent.

    (1) $u\in \dot{\mathbf{U}}[\geq \lambda_1]$;

    (2) for any $\lambda,\lambda'\in X^+$, we have $u(\xi_{-\lambda}\otimes \eta_{\lambda'})\in ({}^{\omega}\Lambda_{\lambda}\otimes \Lambda_{\lambda'})[\geq \lambda_1]$;

    (3) If $\lambda\in X^+$ and $u $ acts on $\Lambda_{\lambda}$ by a nonzero map, then $\lambda \geq \lambda_1$.
\end{lemma}

We have a similar version for $\dot{\mathbf{U}}[>\lambda_1]$.

Lusztig obtained the following so-called refined Peter-Weyl theorem.

\begin{theorem}
    \cite[Thm 29.3.3]{L1}
    
    (1) The subspaces $\dot{\mathbf{U}}[\geq \lambda]$ and $\dot{\mathbf{U}}[>\lambda]$ are two-sided ideals. The quotient algebra $\dot{\mathbf{U}}[\geq \lambda]/\dot{\mathbf{U}}[>\lambda]$ is isomorphic to $ \operatorname{End}(\Lambda_{\lambda})$, and the image of $\dot{\mathbf{B}}[\lambda]$ under the natural projection form a basis of it (we denote the image still by $\dot{\mathbf{B}}[\lambda]$). 

    (2) There is a unique direct sum decomposition of $\dot{\mathbf{U}}[\geq \lambda]/\dot{\mathbf{U}}[>\lambda]$ into a direct sum of simple left (resp. right) $\dot{\mathbf{U}}$-modules, such that each summand is generated by its intersection with the basis $\dot{\mathbf{B}}[\lambda]$.

    (3) Any summand in the left module decomposition and the right module decomposition have an intersection equal to a line consisting of all multiples of some element in $\dot{\mathbf{B}}[\lambda]$.
\end{theorem}

%%%%%%%%%%%%%%%%%%% 以上是Lusztig书的部分

Fix a commutative ring $A$ with 1 and with ring homomorphism $\mathcal{A}\rightarrow A$ preserving 1.
Then for any algebras, modules, basis or maps over $\mathcal{A}$, we can change it to be over $A$, and change the subscript to $A$.
For example, we can see that the structure constants of $\dot{\mathbf{U}}_{A}$ are in $A$.

Let $\hat{\mathbf{U}}_{A}$ be the $A$-module of all formal linear combinations $\sum_{a\in \dot{\mathbf{B}}} n_a a$ with $n_a\in A$. Define the multiplication on $\hat{\mathbf{U}}_A$ by
\begin{align*}
    (\sum_{a\in \dot{\mathbf{B}}} n_a a)(\sum_{b\in \dot{\mathbf{B}}} n'_b b)=\sum_{c\in \dot{\mathbf{B}}}  r_c c,
\end{align*}
where $r_c=\sum_{a,b\in \dot{\mathbf{B}}} m_{a,b}^c n_a n'_b$.
This defines an associative $A$-algebra structure on $\hat{\mathbf{U}}_{A}$ with unit $1=\sum_{\lambda\in X}1_{\lambda}$.

For any $a\in \dot{\mathbf{B}}$, we denote $S(a)=s_a \underline{a}$, where $a\mapsto \underline{a}$ is an involution of $\dot{\mathbf{B}}$ and $s_a$ is up to sign a power of $v$.
We define an $A$-linear map $S:\hat{\mathbf{U}}_{A} \rightarrow \hat{\mathbf{U}}_{A}$ by $\sum_{a\in \dot{\mathbf{B}}} n_a a \mapsto \sum_{a\in \dot{\mathbf{B}}} s_a n_a \underline{a}$.

For any $a\in \dot{\mathbf{B}}$, we define a linear form $a^*:\dot{\mathbf{U}}_A \rightarrow A$ by $b\mapsto \delta_{a,b}$, for all $b\in \dot{\mathbf{B}}$. 
Let $\mathbf{O}_A$ be the $A$-submodule of $\dot{\mathbf{U}}_A^{\diamond}$ spanned by $\{ a^* | a\in \dot{\mathbf{B}} \}$. 
The Hopf algebra structure is defined as follows.
The multiplication on $\mathbf{O}_A$ is given by $a^*b^*=\sum_{c\in \dot{\mathbf{B}}} \hat{m}_{c}^{a,b}c^*$, for all $a,b\in \dot{\mathbf{B}}$.
The unit element for this algebra structure is $1_0^*$.
The comultiplication $\delta:\mathbf{O}_A\rightarrow \mathbf{O}_A \otimes \mathbf{O}_A$ is given by $c^* \mapsto \sum_{a,b\in \dot{\mathbf{B}}} m_{a,b}^c a^* \otimes b^*$.
The counit $\mathbf{O}_A\rightarrow A$ is given by $a^* \mapsto 1$ if $a=1_{\lambda}$ for some $\lambda\in X$, and $a^* \mapsto 0$ if $a\in \dot{\mathbf{B}}$ is not of the form $1_{\lambda}$.
The antipode $S:\mathbf{O}_A\rightarrow \mathbf{O}_A$ is given by $S(a^*)=s_a \underline{a}^*$.

From now on we assume $v=1$ in $A$.

Define the group scheme $\mathbf{G}_A$ as the set of $A$-algebra homomorphisms $\mathbf{O}_A \rightarrow A$ preserving 1. Since is a Hopf algebra, we have a natural group structure on $\mathbf{G}_A$.
On the other hand, we can identify $\mathbf{G}_A$ with a subset of $\hat{\mathbf{U}}_A$ via $\phi \mapsto \sum_{a\in \dot{\mathbf{B}}} \phi(a^*)a$.
Precisely, we have 
\begin{align*}
    \mathbf{G}_A = \{ \sum_{a\in \dot{\mathbf{B}}} n_a a \in \hat{\mathbf{U}}_A | n_{1_0}=1, \sum_{c\in \dot{\mathbf{B}}} \hat{m}_c^{a,b} n_c = n_a n_b ,\text{ for all } a,b\in \dot{\mathbf{B}}  \}.
\end{align*}
Note that the inverse element of $x\in \mathbf{G}_A$ is $S(x)$, where $S$ is already defined for $\hat{\mathbf{U}}_A$.

Let $\hat{\mathbf{U}}_A^{>0}$ (resp. $\hat{\mathbf{U}}_A^{<0}$) be the $A$-submodule of $\hat{\mathbf{U}}_A$ consisting of all elements of the form $\sum_{b\in \mathbf{B},\lambda \in X} n_b (1_{\lambda}b^+)$ (resp. $\sum_{b\in \mathbf{B},\lambda \in X} n_b (b^- 1_{\lambda})$) with $n_b\in A$.
Define $\mathbf{G}_A^{>0}=\mathbf{G}_A \cap \hat{\mathbf{U}}_A^{>0}$ (resp. $\mathbf{G}_A^{<0}=\mathbf{G}_A \cap \hat{\mathbf{U}}_A^{<0}$). They are subgroups of $\mathbf{G}_A$.

Let $\hat{\mathbf{U}}_A^0$ be the set of elements of $\hat{\mathbf{U}}_A$ of the form $\sum_{\lambda \in X} n_{\lambda}1_{\lambda}$ with $n_{\lambda}\in A$, and $\mathbf{T}_A=\hat{\mathbf{U}}_A^0 \cap \mathbf{G}_A$. 
Let $A^{\circ}$ be the group of invertible elements of $A$.
Then $\mathbf{T}_A$ is an abelian Subgroup of $\mathbf{G}_A$, which can be identified with $\operatorname{Hom}(X,A^{\circ}) = A^{\circ}\otimes_{\mathbb{Z}}Y$, via $d\otimes y \mapsto \sum_{\lambda \in X} d^{<y,\lambda>} 1_{\lambda}$, for $d\in A^{\circ},y\in Y$.

For any $i\in I,h\in A$, set 
\begin{align*}
    x_i(h)=\sum_{c\in \mathbb{N},\lambda \in X} h^c \theta_i^{(c)+}1_{\lambda},\\
    y_i(h)=\sum_{c\in \mathbb{N},\lambda \in X} h^c \theta_i^{(c)-}1_{\lambda}.
\end{align*}
Then we have $x_i(h+h')=x_i(h)x_i(h'),y_i(h+h')=y_i(h)y_i(h')$, and $x_i(h)\in \mathbf{G}_A^{>0},y_i(h)\in \mathbf{G}_A^{<0}$.

For $i\in I$, set
\begin{align*}
    s'_{i}=\sum_{l,m\in \mathbb{Z},\lambda\in X,<i,\lambda>=l+m} (-1)^m v_i^{m} \theta_i^{(l)-} 1_{\lambda} \theta_i^{(m)+};\\
    s''_{i}=\sum_{l,m\in \mathbb{Z},\lambda\in X,<i,\lambda>=l+m} (-1)^l v_i^{l} \theta_i^{(l)-} 1_{\lambda} \theta_i^{(m)+}.
\end{align*} 
We have $s''_{i}=x_i(1)y_i(-1)x_i(1), s'_{i}=y_i(1)x_i(-1)y_i(1)$.

Let $\mathbf{W}$ be the subgroup of $\operatorname{Aut}(Y)$ generated by the involutions $s_i: y\mapsto y-\langle y,i' \rangle i, i\in I$, and $w_0$ be the longest element in $\mathbf{W}$.
For any $w\in \mathbf{W}$, there are unique elements $w',w''\in \hat{\mathbf{U}}_A$ with $w'=s'_{i_1}s'_{i_2}\cdots s'_{i_r},w''=s''_{i_1}s''_{i_2}\cdots s''_{i_r}$ for any reduced sequence $i_1,i_2,\cdots ,i_r$ such that $w=s_{i_1}s_{i_2}\cdots s_{i_r}$. We have $w',w''\in \mathbf{G}_A$.

For $i\in I, u\in A^{\circ}$, let 
\begin{align*}
    t_i(u)=\sum_{\lambda \in X} u^{<i,\lambda>} 1_{\lambda} \in \mathbf{T}_A.
\end{align*}
Then we have $s''_i=s'_it_i(-1)$ in $\mathbf{G}_A$, and $t_i(u)=x_i(u-1)y_i(1)x_i(u^{-1}-1)y_i(-u)$.

Let $n=l(w_0)$ and fix a sequence $i_1,\cdots,i_n$ in $I$ such that $s_{i_1}\cdots s_{i_n}=w_0$. For any $\mathbf{h}=(h_1,\cdots,h_n)\in A^n$, set 
\begin{align*}
    x_{\mathbf{h}}=x_{i_1}(h_1) (s''_{i_1}x_{i_2}(h_2)s''_{i_1}{}^{-1})\cdots (s''_{i_1}s''_{i_2}\cdots s''_{i_{n-1}}x_{i_n}(h_n)s''_{i_{n-1}}{}^{-1}\cdots s''_{i_2}{}^{-1}s''_{i_1}{}^{-1}).
\end{align*}
Then we have 
\begin{lemma}
    \cite[4.8]{zform}
    The map $A^n\rightarrow \mathbf{G}_A^{>0}, \mathbf{h}\mapsto x_{\mathbf{h}}$ is a bijection.
\end{lemma}

For any $\lambda\in X$, define a homomorphism $\chi_{\lambda}:\mathbf{T}_A\rightarrow A^{\circ}, \sum_{\lambda\in X} n_{\lambda} 1_{\lambda} \mapsto n_{\lambda} $. 
Then we have for any $t\in \mathbf{T}_A, h\in A,i\in I$,
\begin{align*}
    tx_i(h)t^{-1}=x_i(\chi_{i'}(t)h).
\end{align*}

More generally, for each $k=1,\cdots,n$, define a homomorphism $f_k:A\rightarrow \mathbf{G}_A^{>0}$ by $h\mapsto s''_{i_1}s''_{i_2}\cdots s''_{i_{k-1}}x_{i_k}(h)s''_{i_{k-1}}{}^{-1}\cdots s''_{i_2}{}^{-1}s''_{i_1}{}^{-1}$.
Set $\lambda_k=s_{i_1}s_{i_2}\cdots s_{i_{k-1}}(i'_k)\in X$, then for $t\in \mathbf{T}_A,h\in A$, we have 
\begin{align*}
    tf_k(h)t^{-1}=f_k(\chi_{\lambda_k}(t)h).
\end{align*}

When $A$ is an algebraically closed field, we have $\mathbf{G}_A$ is a reductive group with coordinate ring $\mathbf{O}_A$.

\begin{theorem}
    \cite[Thm 4.11]{zform}
    Assume $A$ is an algebraically closed field. Then $\mathbf{G}_A$ is a connected reductive group over $A$, with coordinate ring $\mathbf{O}_A$ and with associated root datum the same as $(X,Y,I,\langle,\rangle)$. 
    Moreover, subgroups $\mathbf{G}_A^{>0}\mathbf{T}_A$ and $\mathbf{G}_A^{<0}\mathbf{T}_A$ are Borel subgroups of $\mathbf{G}_A$, and $\mathbf{T}_A$ is a maximal torus.
\end{theorem}

Let $\mathcal{C}_A$ be the category of unital $\dot{\mathbf{U}}_A$-modules, i.e. the object $M$ is finitely generated as an $A$-module, satisfying that for any $z\in M, 1_{\lambda}z=0$ for all but finitely many $\lambda \in X$, and $\sum_{\lambda \in X} 1_{\lambda}z = z$.

When $A=\mathbb{Q}$, any object in $\mathcal{C}_{\mathbb{Q}}$ is semisimple, and the simple objects are exactly the objects $\Lambda_{\lambda,\mathbb{Q}},\lambda \in X^+$.

For any $M\in \mathcal{C}_{\mathbb{Q}}$, define $C_M$ be the $\mathbb{Q}$-subspace of $\dot{\mathbf{U}}_{\mathbb{Q}}^{\diamond}$ spanned by the functions $u\mapsto z'(uz)$ for some $z\in M,z'\in M^{\diamond}$.

\begin{lemma}\label{matrixcoeff}
    \cite[Lemma 5.5]{zform}
    Let $f\in \dot{\mathbf{U}}_{\mathbb{Q}}^{\diamond}$. The following are equivalent.

    (1) $f\in \mathbf{O}_{\mathbb{Q}}$.

    (2) $f\in \sum_{M\in \mathcal{C}_{\mathbb{Q}}} C_M$.

    (3) $f\in \sum_{\lambda\in X^+} C_{\Lambda_{\lambda,\mathbb{Q}}}$.
\end{lemma}

\begin{proof}
    If $f\in \mathbf{O}_{\mathbb{Q}}$, we may assume $f=a^*$ for some $a\in \dot{\mathbf{B}}$. Then there exists some $\lambda,\lambda'\in X^+$ such that $a(\xi_{\lambda}\otimes \eta_{\lambda'})$ is in the canonical basis of ${}^{\omega} \Lambda_{\lambda,\mathbb{Q}}\otimes \Lambda_{\lambda',\mathbb{Q}}$. Denote this module by $M$, and $z=\xi_{\lambda}\otimes \eta_{\lambda'}$. Define $z'\in M^{\diamond}$ by $z'(bz)=\delta_{a,b}$ for any $b\in \dot{\mathbf{B}}$. Then we have $a^*(u)=z'(uz)$ for all $u\in \dot{\mathbf{U}}_{\mathbb{Q}}$. Thus $f=a^*\in C_M$ and (2) holds.

    If $f\in \sum_{M\in \mathcal{C}_{\mathbb{Q}}} C_M$, then (3) holds since $C_{M\oplus M'}=C_M +C_{M'}$ and any object in $\mathcal{C}_{\mathbb{Q}}$ is a direct sum of objects of the form $\Lambda_{\lambda,\mathbb{Q}},\lambda\in X^+$.

    If $f\in \sum_{\lambda\in X^+} C_{\Lambda_{\lambda,\mathbb{Q}}}$, we may assume $f(u)=z'(uz)$ for some $z\in \Lambda_{\lambda,\mathbb{Q}}$ and $z'\in \Lambda_{\lambda,\mathbb{Q}}^{\diamond}$. By Lemma \ref{PWlemma}, we know that if $b\in \dot{\mathbf{B}}-\bigcup_{\lambda'\leq \lambda} \dot{\mathbf{B}}[\lambda']$, then $b$ is zero on $\Lambda_{\lambda,\mathbb{Q}}$ and thus $f(b)=0$. We can deduce that 
    \begin{align*}
        f=\sum_{b\in \bigcup_{\lambda'\leq \lambda}\dot{\mathbf{B}}[\lambda']} f(b) b^* \in \mathbf{O}_{\mathbb{Q}},
    \end{align*}
     thus (1) holds.
\end{proof}

By replacing the $\mathbb{Q}$ by $\mathbb{Q}(v)$ in the proof, the previous Lemma indicates that $\mathbf{O}_{\mathbb{Q}(v)}$ is indeed the quantum coordinate algebra.

Finally, we have the following isomorphism.

\begin{proposition}
    \cite[5.7]{zform}
    We have a $\mathbb{Q}$-vector space isomorphism 
    \begin{align*}
        \mathbf{O}_{\mathbb{Q}} \cong \oplus_{\lambda\in X^+} \Lambda_{\lambda,\mathbb{Q}} \otimes \Lambda_{\lambda,\mathbb{Q}}^{\diamond}.
    \end{align*}
\end{proposition}

\section{Compact Real Form and Maximal Compact Subgroup from Root Category}
In this section, we will use the root category $\mathcal{R}$ of finite type to define a real Lie algebra, and show that this real Lie algebra is a compact real form of $\mathfrak{g}_{\mathbb{C}}$. Using this compact real form, we obtain a maximal compact subgroup of $\mathbf{G}=\mathbf{G}(\mathcal{R},\mathbb{C})$. We adopt the notations in previous sections.

\subsection{Compact Real Form}
\begin{definition}
Let $\mathfrak{u}$ be the quotient of a free abelian group generated by $\alpha_{[X]},\beta_{[X]},\xi_{[X]}$, for all $X\in \operatorname{ind}\mathcal{R}$, and subject to relations:

(1) if there exists a triangle $X\rightarrow Y \rightarrow Z \rightarrow TX$ in $\mathcal{R}$, then $d(Y)\alpha_{[Y]} = d(X)\alpha_{[X]}+d(Z)\alpha_{[Z]}$;

(2) for each $X$, $\beta_{[TX]}=\beta_{[X]}$;

(3) for each $X$, $\xi_{[TX]}=-\xi_{[X]}$.
\end{definition}

By abuse of notations, we write $\alpha_X,\beta_X,\xi_X$ instead of $\alpha_{[X]},\beta_{[X]},\xi_{[X]}$.

\begin{definition}
    We define a bilinear operation $[-,-]$ on $\mathfrak{u}$ as follows.

    (1) For any $a\in \mathfrak{u}$, we have $[a,a]=0$.

    For any $X,Y\in \operatorname{ind}\mathcal{R}$, 

    (2) $[\alpha_X,\alpha_Y]=0$.

    (3) $[\alpha_X,\beta_Y]=-A_{XY}\xi_Y$.

    (4) $[\alpha_X,\xi_Y]=A_{XY}\beta_Y$.

    (5) $[\beta_X,\beta_Y]=\sum_{L\in \operatorname{ind}\mathcal{R}}\gamma_{XY}^L \beta_L+\sum_{M\in \operatorname{ind}\mathcal{R}}\gamma_{X,TY}^M \beta_M$.

    (6) $[\xi_X,\xi_Y]=-\sum_{L\in \operatorname{ind}\mathcal{R}} \gamma_{XY}^L\beta_L +\sum_{M\in \operatorname{ind}\mathcal{R}} \gamma_{X,TY}^M \beta_M$.

    (7)  \begin{equation*}
    [\beta_X,\xi_Y]=
    \begin{cases}
        -2\alpha_X  \qquad &\text{if} \quad X\cong Y, \\
        \sum_{L\in \operatorname{ind}\mathcal{R}} \gamma_{XY}^L\xi_L -\sum_{M\in \operatorname{ind}\mathcal{R}} \gamma_{X,TY}^M \xi_M  &\text{if} \quad X\ncong Y,X\ncong TY.
    \end{cases}
    \end{equation*}

    Then extend $[-,-]$ bilinearly to the whole $\mathfrak{u}$.
\end{definition}

\begin{theorem}
    We have $\mathfrak{u}\otimes_{\mathbb{Z}} \mathbb{R}$ together with $[-,-]$ is a real Lie algebra.
\end{theorem}

\begin{proof}
  
    It suffice to show the Jacobi identity for generators. For $X,Y,Z\in \operatorname{ind}\mathcal{R}$, it is enough to consider the following cases.

    \textbf{Case 1} 
\begin{align*}
       [\alpha_X,[\alpha_Y,\beta_Z]]+[\alpha_Y,[\beta_Z,\alpha_X]]+[\beta_Z,[\alpha_X,\alpha_Y]]=-A_{YZ}A_{XZ}\beta_Z+A_{XZ}A_{YZ}\beta_Z = 0.
\end{align*}
  
    Similarly, we can obtain the identity replacing $\beta_Z$ by $\xi_Z$.
    
    \textbf{Case 2} Consider
    $ [\alpha_X,[\beta_Y,\beta_Z]]+[\beta_Y,[\beta_Z,\alpha_X]]+[\beta_Z,[\alpha_X,\beta_Y]]$.

If $Y\cong Z$ or $Y \cong TZ$, then it is obviously zero.
Otherwise, we have 
\begin{align*}
    &[\alpha_X,[\beta_Y,\beta_Z]]+[\beta_Y,[\beta_Z,\alpha_X]]+[\beta_Z,[\alpha_X,\beta_Y]]\\
    =& [\alpha_X,\sum_L \gamma_{YZ}^L\beta_L+\sum_M \gamma_{Y,TZ}^M \beta_M]+A_{XZ}[\beta_Y,\xi_Z]-A_{XY}[\beta_Z,\xi_Y]\\
    =& \sum_L \gamma_{YZ}^L (-A_{XL}+A_{XZ}+A_{XY})\xi_L +\sum_M \gamma_{Y,TZ}^M (-A_{XM}-A_{XZ}+A_{XY}) \xi_M =0, 
\end{align*} 
since $-A_{XL}+A_{XZ}+A_{XY}=0=-A_{XM}-A_{XZ}+A_{XY}$.

Similarly, we have 
\begin{align*}
    [\alpha_X,[\beta_Y,\xi_Z]]+[\beta_Y,[\xi_Z,\alpha_X]]+[\xi_Z,[\alpha_X,\beta_Y]]=0.\\
    [\alpha_X,[\xi_Y,\xi_Z]]+[\xi_Y,[\xi_Z,\alpha_X]]+[\xi_Z,[\alpha_X,\xi_Y]]=0.
\end{align*}

\textbf{Case 3} Consider
$[\beta_X,[\beta_Y,\beta_Z]]+[\beta_Y,[\beta_Z,\beta_X]]+[\beta_Z,[\beta_X,\beta_Y]]$.

Since the module obtained by the extension of multiple modules is independent of the order in which the extension are performed, we may use $L_{1,1,1}$ to denote the extension of $L_{X,Y,1,1}$ and $Z$, and others are similar. Use these simplified notations, we have
\begin{align*}
    &[\beta_X,[\beta_Y,\beta_Z]]+[\beta_Y,[\beta_Z,\beta_X]]+[\beta_Z,[\beta_X,\beta_Y]]\\
    =&(\gamma_{YZ}^{L_{0,1,1}}\gamma_{X,L_{0,1,1}}^{L_{1,1,1}}+\gamma_{ZX}^{L_{1,0,1}}\gamma_{Y,L_{1,0,1}}^{L_{1,1,1}}+\gamma_{XY}^{L_{1,1,0}}\gamma_{Z,L_{1,1,0}}^{L_{1,1,1}})\beta_{L_{1,1,1}}\\
    +& (\gamma_{YZ}^{L_{0,1,1}}\gamma_{X,L_{0,-1,-1}}^{L_{1,-1,-1}}+\gamma_{Z,TX}^{L_{-1,0,1}}\gamma_{Y,L_{-1,0,1}}^{L_{-1,1,1}}+\gamma_{X,TY}^{L_{1,-1,0}}\gamma_{Z,L_{-1,1,0}}^{L_{-1,1,1}})\beta_{L_{1,-1,-1}}\\
    +& (\gamma_{Y,TZ}^{L_{0,1,-1}}\gamma_{X,L_{0,1,-1}}^{L_{1,1,-1}}+\gamma_{Z,TX}^{L_{-1,0,1}}\gamma_{Y,L_{1,0,-1}}^{L_{1,1,-1}}+\gamma_{XY}^{L_{1,1,0}}\gamma_{Z,L_{-1,-1,0}}^{L_{-1,-1,1}})\beta_{L_{1,1,-1}}\\
    +&(\gamma_{Y,TZ}^{L_{0,1,-1}}\gamma_{X,L_{0,-1,1}}^{L_{1,-1,1}}+\gamma_{ZX}^{L_{1,0,1}}\gamma_{Y,L_{-1,0,-1}}^{L_{-1,1,-1}}+\gamma_{X,TY}^{L_{1,-1,0}}\gamma_{Z,L_{1,-1,0}}^{L_{1,-1,1}})\beta_{L_{1,-1,1}}.
\end{align*}
It suffice to show $\gamma_{YZ}^{L_{0,1,1}}\gamma_{X,L_{0,1,1}}^{L_{1,1,1}}+\gamma_{ZX}^{L_{1,0,1}}\gamma_{Y,L_{1,0,1}}^{L_{1,1,1}}+\gamma_{XY}^{L_{1,1,0}}\gamma_{Z,L_{1,1,0}}^{L_{1,1,1}}=0$, and the results for other three coefficients can obtained by replacing $X,Y,Z$ by $TX,TY,TZ$, if necessary.
The desired identity can be obtained from the octahedral axiom of the root category, and one can see the proof of \cite[Prop 4.1]{1997ROOT} for details. In fact, it corresponds to the Jacobi identity $[u_X,[u_Y,u_Z]]+[u_Y,[u_Z,u_X]]+[u_Z,[u_X,u_Y]]=0$ of the Lie algebra $\mathfrak{g}_{\mathbb{Z}}$.

\textbf{Case 4} Consider
$[\beta_X,[\beta_Y,\xi_Z]]+[\beta_Y,[\xi_Z,\beta_X]]+[\xi_Z,[\beta_X,\beta_Y]]$.

If $X\ncong Z,X\ncong TZ,Y\ncong Z,Y\ncong TZ$, it is similar to Case 3. If $X\cong Y\cong Z$, it is obvious.
If $X\ncong Z,X\ncong TZ,Y\cong Z$, we use $L_{i,j}$ to denote $L_{X,Y,i,j}$, and we have 
\begin{align*}
    &[\beta_X,[\beta_Y,\xi_Z]]+[\beta_Y,[\xi_Z,\beta_X]]+[\xi_Z,[\beta_X,\beta_Y]]\\
    =& (-2A_{YX}-2\gamma_{XY}^{L_{1,1}}\gamma_{Y,L_{-1,-1}}^{TX}+2\gamma_{X,TY}^{L_{1,-1}}\gamma_{Y,L_{1,-1}}^X)\xi_X=0,
\end{align*}
where $A_{YX}+\gamma_{XY}^{L_{1,1}}\gamma_{Y,L_{-1,-1}}^{TX}-\gamma_{X,TY}^{L_{1,-1}}\gamma_{Y,L_{1,-1}}^X=0$ is also obtained in the proof of \cite[Prop 4.1]{1997ROOT}, corresponding to the Jacobi identity $[[u_Y,u_{TY}],u_X]+[[u_{TY},u_X],u_Y]+[[u_X,u_Y],u_{TY}]=0$ of the Lie algebra $\mathfrak{g}_{\mathbb{Z}}$.

Similarly, we have 
\begin{align*}
    [\beta_X,[\xi_Y,\xi_Z]]+[\xi_Y,[\xi_Z,\beta_X]]+[\xi_Z,[\beta_X,\xi_Y]]=0.
\end{align*}

\textbf{Case 5} Consider
\begin{align*}
      &[\xi_X,[\xi_Y,\xi_Z]]+[\xi_Y,[\xi_Z,\xi_X]]+[\xi_Z,[\xi_X,\xi_Y]]\\
      =&[\xi_X,-\gamma_{YZ}^{L_{0,1,1}}\beta_{L_{0,1,1}}+\gamma_{Y,TZ}^{L_{0,1,-1}}\beta_{L_{0,1,-1}}]+[\xi_Y,-\gamma_{ZX}^{L_{1,0,1}}\beta_{L_{1,0,1}}+\gamma_{Z,TX}^{L_{-1,0,1}}\beta_{L_{-1,0,1}}]\\
      &+[\xi_Z,-\gamma_{XY}^{L_{1,1,0}}\beta_{L_{1,1,0}}+\gamma_{X,TY}^{L_{1,-1,0}}\beta_{L_{1,-1,0}}].
\end{align*}
  
If the modules appearing in the above formula are pairwise non-isomorphic, then the case is similar to Case 3. Otherwise, for example, if $X\cong L_{0,1,1}$, then other isomorphic relations are determined, and we have 
\begin{align*}
    &[\xi_X,[\xi_Y,\xi_Z]]+[\xi_Y,[\xi_Z,\xi_X]]+[\xi_Z,[\xi_X,\xi_Y]]\\
    =&[\xi_X,-\gamma_{YZ}^{X}\beta_{X}+\gamma_{Y,TZ}^{L_{0,1,-1}}\beta_{L_{0,1,-1}}]+[\xi_Y,-\gamma_{ZX}^{L_{1,0,1}}\beta_{L_{1,0,1}}+\gamma_{Z,TX}^{TY}\beta_{Y}]\\
      &+[\xi_Z,-\gamma_{XY}^{L_{1,1,0}}\beta_{L_{1,1,0}}+\gamma_{X,TY}^{Z}\beta_{Z}]\\
      =& -2 \gamma_{YZ}^X \alpha_X - \gamma_{Y,TZ}^{L_{0,1,-1}}[\beta_{L_{0,1,-1}},\xi_X] + \gamma_{ZX}^{L_{1,0,1}} [\beta_{L_{1,0,1}},\xi_Y] \\
      &+2\gamma_{Z,TX}^{TY}\alpha_Y +\gamma_{XY}^{L_{1,1,0}}[\beta_{L_{1,1,0}},\xi_Z] +2\gamma_{X,TY}^Z \alpha_Z\\
      =& -2 \gamma_{YZ}^X \alpha_X +2\gamma_{Z,TX}^{TY}\alpha_Y +2\gamma_{X,TY}^Z \alpha_Z.
\end{align*}
The last equality follows from that there is no extension between $X$ and $L_{0,1,-1}$, $Y$ and $L_{1,0,1}$, or $Z$ and $L_{1,1,0}$.

Since $d(X)\alpha_X = d(Y) \alpha_Y + d(Z)\alpha_Z$, it suffice to show 
\begin{align*}
    \gamma_{TZ,X}^Y d(X)=\gamma_{YZ}^X d(Y).\\
    \gamma_{X,TY}^Z d(X) = \gamma_{YZ}^X d(Z).
\end{align*}

The second formula can be obtained from the first one by exchanging $Y$ and $Z$. 
Since there exists triangle $TZ\rightarrow Y \rightarrow X \rightarrow Z$ if and only if there exists triangle $Y\rightarrow X \rightarrow Z \rightarrow TY$, we have $\gamma_{TZ,X}^Y d(X)$ and $\gamma_{YZ}^X d(Y)$ have the same sign.
 Then we can vertify this equality by a case-by-case inspection of the possible rank two root system generated by $Y,Z$.
\end{proof}

Then we demonstrate the relation between the real Lie algebra $\mathfrak{u}\otimes_{\mathbb{Z}} \mathbb{R}$ and the complex Lie algebra $\mathfrak{g}_{\mathbb{C}}$ defined by Peng and Xiao.

The shift functor $T$ of the root category $\mathcal{R}$ naturally induces an $\mathbb{R}$-linear involution $\theta$ of $\mathfrak{g}_{\mathbb{C}}$. Precisely,
\begin{align*}
    \theta(H_X) = H_{TX}, \quad \theta(u_{X})=u_{TX} ,
\end{align*}
for each $X\in \operatorname{ind}\mathcal{R}$, and $\theta$ maps $i$ to $-i$.

Let 
\begin{align*}
    \mathfrak{r}=\{ X\in \mathfrak{g}_{\mathbb{C}}| \theta(X)=X \}.
\end{align*}
Then $\mathfrak{r}$ is a real Lie algebra, and $\mathfrak{r}\otimes_{\mathbb{R}}\mathbb{C} \cong \mathfrak{g}_{\mathbb{C}}$. Thus $\mathfrak{r}$ is a real form of $\mathfrak{g}_{\mathbb{C}}$.

\begin{proposition}
    The real Lie algebra $\mathfrak{u}\otimes_{\mathbb{Z}} \mathbb{R}$ is isomorphic to the real form $\mathfrak{r}$ of $\mathfrak{g}_{\mathbb{C}}$.
\end{proposition}

\begin{proof}
    We define an $\mathbb{R}$-linear map $\varphi: \mathfrak{u}\otimes_{\mathbb{Z}} \mathbb{R} \rightarrow \mathfrak{g}_{\mathbb{C}}$, which maps $\alpha_X $ to $iH_X'$, $\beta_X$ to $u_X+u_{TX}$, and $\xi_X$ to $i(u_X-u_{TX})$, for each $X\in \operatorname{ind}\mathcal{R}$.

    It can be calculated that this linear map preserves the Lie bracket, and the image of $\varphi$ is exactly $\mathfrak{r}$. Thus we obtain a real Lie algebra homomorphism $\varphi: \mathfrak{u}\otimes_{\mathbb{Z}} \mathbb{R} \rightarrow \mathfrak{r}$, which is surjective. 
    
    On the other hand, if we choose a hereditary subcategory $\mathcal{B}$ of $\mathcal{R}$, and let $\Delta$ be the set of simple objects in $\mathcal{B}$, we have $\{iH_X',X\in \Delta\}$ and $\{u_Y+u_{TY},i(u_Y-u_{TY}), Y\in \operatorname{ind}\mathcal{B}\}$ form an $\mathbb{R}$-basis of $\mathfrak{r}$. Then it can be easily shown that $\varphi $ is injective, and thus isomorphism. 
\end{proof}

We use the isomorphism $\mathfrak{u}\otimes_{\mathbb{Z}} \mathbb{R}\cong \mathfrak{r}$ to regard $\mathfrak{u}\otimes_{\mathbb{Z}} \mathbb{R}$ as a real Lie subalgebra of $\mathfrak{g}_{\mathbb{C}}$.

To show the real Lie algebra $\mathfrak{u}\otimes_{\mathbb{Z}} \mathbb{R}$ is compact, we first define an invariant symmetric bilinear form on $\mathfrak{g}_{\mathbb{C}}$. 

\begin{definition}
Let $(,):\mathfrak{g}_{\mathbb{C}} \times \mathfrak{g}_{\mathbb{C}} \rightarrow \mathbb{C}$ be a non-degenerate symmetric bilinear form on $\mathfrak{g}_{\mathbb{C}}$, satisfying the following properties:

(1) For any $x,y,z\in \mathfrak{g}_{\mathbb{C}}$, we have 
\begin{align*}
    ([x,y],z)=(x,[y,z]).
\end{align*}
    
(2) For any $X,Y\in \mathcal{R}$, we have 
\begin{align*}
    (H_X,H_Y)=\sum_{Z\in \operatorname{ind}\mathcal{R}}(H_X|H_Z)(H_Y|H_Z),
\end{align*}
where $(-|-)$ is the symmetric Euler form on $\mathcal{K}$.

(3) For any $X,Y\in \operatorname{ind}\mathcal{R}$, $(H_X,u_Y)=0$, i.e. $(\mathcal{K},\mathcal{N})=0$.

(4) For any $X,Y\in \operatorname{ind}\mathcal{R}$, if $X\ncong TY$, then $(u_X,u_Y)=0$. 
If $X\cong TY$, then 
\begin{align*}
    (u_X,u_{TX})= -4+\sum_{Y\in\operatorname{ind}\mathcal{R}}(\sum_{L\in\operatorname{ind}\mathcal{R}} \gamma_{TX,Y}^L\gamma_{X,L}^Y)=-4+\sum_{Y\in \operatorname{ind}\mathcal{R}} \gamma_{TX,Y}^{L_{TX,Y,1,1}}\gamma_{X,L_{TX,Y,1,1}}^Y .
\end{align*}
\end{definition}

The properties uniquely determine the non-degenerate symmetric bilinear form.
This bilinear form is arising from the root category $\mathcal{R}$, and it actually coincides with the Killing form of $\mathfrak{g}_{\mathbb{C}}$. Note that in \cite{PX} Peng and Xiao defined a different nondegenerate symmetric bilinear form from the root category $\mathcal{R}$, which also satisfies some kind of invariance.

We can restrict the bilinear form $(,)$ to $\mathfrak{r}\cong \mathfrak{u}\otimes_{\mathbb{Z}} \mathbb{R}$ and then obtain the following result.

\begin{proposition}
    The bilinear form $(,)$ on $\mathfrak{u}\otimes_{\mathbb{Z}} \mathbb{R}$ is negative-definite, and thus $ \mathfrak{u}\otimes_{\mathbb{Z}} \mathbb{R}$ is a compact Lie algebra.
\end{proposition}

\begin{proof}
Firstly, consider the restriction of $(,)$ to the $\mathbb{R}$-span of $\alpha_X,X\in \operatorname{ind}\mathcal{R}$. It is obviously negative-definite since $\alpha_X \mapsto iH_X'$.

By the isomorphism $\varphi$ and the property (3) of $(,)$, we can easily calculate that $(\alpha_X,\beta_Y)=(\alpha_X,\xi_Y)=0$ for any $X,Y\in \operatorname{ind}\mathcal{R}$.

Use property (4) of $(,)$ and for $X\ncong Y,X\ncong TY$, we have $(\beta_X,\beta_Y)=(\xi_X,\xi_Y)=(\beta_X,\xi_Y)=0$. For any $X$, we have 
\begin{align*}
    &(\beta_X,\beta_X)=(u_X+u_{TX},u_X+u_{TX})=2(u_X,u_{TX}).\\
    &(\xi_X,\xi_X)=-(u_X-u_{TX},u_X-u_{TX})=2(u_X,u_{TX}).
\end{align*}

For $M,N,L\in \operatorname{ind}\mathcal{R}$, if $\gamma_{TM,N}^L\gamma_{M,L}^N \neq 0$, then $\gamma_{TM,N}^L >0 $ if and only if there exists triangle $TM\rightarrow L \rightarrow N \rightarrow T^2M \cong M$ in the root category $\mathcal{R}$, if and only if there exists triangle $L\rightarrow N \rightarrow M \rightarrow TL$ in $\mathcal{R}$, if and only if $\gamma_{M,L}^N<0$. Thus we have $\gamma_{TM,N}^L\gamma_{M,L}^N \leq 0$.

Return to our case, 
$(u_X,u_{TX})= -4+\sum_{Y\in \operatorname{ind}\mathcal{R}} \gamma_{TX,Y}^{L_{TX,Y,1,1}}\gamma_{X,L_{TX,Y,1,1}}^Y <0$ holds for any $X$, thus $(\beta_X,\beta_X)=(\xi_X,\xi_X)<0$.

Hence the bilinear form $(,)$ on $\mathfrak{u}\otimes_{\mathbb{Z}} \mathbb{R}$ is negative-definite, and this implies that $\mathfrak{u}\otimes_{\mathbb{Z}} \mathbb{R}$ is semisimple.
Then $\operatorname{Int}(\mathfrak{u}\otimes_{\mathbb{Z}} \mathbb{R})$ is the identity component of $\operatorname{Aut}(\mathfrak{u}\otimes_{\mathbb{Z}} \mathbb{R})$, and thus a closed subgroup of $\operatorname{GL}(\mathfrak{u}\otimes_{\mathbb{Z}} \mathbb{R})$. 
Moreover, since the negative of $(,)$ forms an inner product on $\mathfrak{u}\otimes_{\mathbb{Z}} \mathbb{R}$, and by definition every element of $\operatorname{ad}(\mathfrak{u}\otimes_{\mathbb{Z}} \mathbb{R})$ acts by skew-symmetric transformation with respect to this inner product, we have $\operatorname{Int}(\mathfrak{u}\otimes_{\mathbb{Z}} \mathbb{R})$ acts by orthogonal transformations. Thus $\operatorname{Int}(\mathfrak{u}\otimes_{\mathbb{Z}} \mathbb{R})$ is a closed subgroup of the orthogonal group, and is therefore compact. By definition, the real Lie algebra $\mathfrak{u}\otimes_{\mathbb{Z}} \mathbb{R}$ is compact.
\end{proof}

To summarize, the real Lie algebra $\mathfrak{u}\otimes_{\mathbb{Z}} \mathbb{R} \cong \mathfrak{r}$ is a compact real form of the complex Lie algebra $\mathfrak{g}_{\mathbb{C}}$.

\subsection{Maximal Compact Subgroup}
Now we want to construct a maximal compact subgroup of $\mathbf{G}(\mathcal{R},\mathbb{C})$.
We have already shown that $\operatorname{Int}(\mathfrak{r})$ is compact and connected, and has Lie algebra $\mathfrak{r}\cong \operatorname{ad}\mathfrak{r}$. Thus $\operatorname{Int}(\mathfrak{r})$ is a desired maximal compact subgroup.

On the other hand, we consider the group generated by images of the exponential map.
We first calculate the action of $\operatorname{exp}(t\operatorname{ad}\alpha_X),\operatorname{exp}(t\operatorname{ad}\beta_X),\operatorname{exp}(t\operatorname{ad}\xi_X)$ on generators of $\mathfrak{r}$, for all $t\in \mathbb{R}$ and $X\in \operatorname{ind}\mathcal{R}$.
Since we have 
\begin{align*}
    &\operatorname{cos}x=\frac{e^{ix}+e^{-ix}}{2}=1-\frac{x^2}{2!}+\frac{x^4}{4!}-\frac{x^6}{6!}+\cdots,\\
    &\operatorname{sin}x=\frac{e^{ix}-e^{-ix}}{2i}=x-\frac{x^3}{3!}+\frac{x^5}{5!}-\frac{x^7}{7!}+\cdots,
\end{align*}
for any $t\in \mathbb{R}$ and $X,Y\in \operatorname{ind}\mathcal{R}$, it can be easily calculated that 
\begin{align*}
    &\operatorname{exp}(t\operatorname{ad}\alpha_X)(\alpha_Y)=\alpha_Y,\\
    &\operatorname{exp}(t\operatorname{ad}\alpha_X)(\beta_Y)=\operatorname{cos}(tA_{XY})\beta_Y - \operatorname{sin}(tA_{XY})\xi_Y,\\
    &\operatorname{exp}(t\operatorname{ad}\alpha_X)(\xi_Y)=\operatorname{sin}(tA_{XY})\beta_Y+\operatorname{cos}(tA_{XY})\xi_Y,\\
    &\operatorname{exp}(t\operatorname{ad}\beta_X)(\alpha_Y)=\alpha_Y+ \frac{\operatorname{cos}(2t)-1}{2} A_{YX}\alpha_X  +\frac{\operatorname{sin}(2t)}{2}A_{YX}\xi_X,\\
    &\operatorname{exp}(t\operatorname{ad}\xi_X)(\alpha_Y)=\alpha_Y+ \frac{\operatorname{cos}(2t)-1}{2} A_{YX}\alpha_X  -\frac{\operatorname{sin}(2t)}{2}A_{YX}\beta_X.
\end{align*}

The action of $\operatorname{exp}(t\operatorname{ad}\beta_X)$ or $\operatorname{exp}(t\operatorname{ad}\xi_X)$ on $\beta_Y,\xi_Y$ are more complicated.
To calculate these, we present the operators $\operatorname{exp}(t\operatorname{ad}(u_X+u_{TX}))$ and $\operatorname{exp}(t\operatorname{ad}i(u_X-u_{TX}))$ as elements in the Chevalley group $\mathbf{G}(\mathcal{R},\mathbb{R})$.

\begin{lemma}\label{expbeta}
    For any $t\in \mathbb{R}, X\in \operatorname{ind}\mathcal{R}$, we have 
    \begin{align*}
       & \operatorname{exp}(t\operatorname{ad}(u_X+u_{TX}))=E_{[X]}(\operatorname{tan}t)h_{[X]}((\operatorname{cos}t)^{-1})E_{[TX]}(\operatorname{tan}t),\\
        & \operatorname{exp}(t\operatorname{ad}i(u_X-u_{TX}))=E_{[X]}(i\operatorname{tan}t)h_{[X]}((\operatorname{cos}t)^{-1})E_{[TX]}(-i\operatorname{tan}t).
    \end{align*}
\end{lemma}

\begin{proof}
    For any $t\in \mathbb{R}$, we have
\begin{align*}
   \operatorname{exp}(t\begin{pmatrix}
    0&1\\
    -1&0
   \end{pmatrix}) =\begin{pmatrix}
    \operatorname{cos} t & \operatorname{sin}t\\
    -\operatorname{sin}t & \operatorname{cos} t 
   \end{pmatrix}
   = \begin{pmatrix}
    1 & \operatorname{tan}t\\
    0 & 1
   \end{pmatrix}
   \begin{pmatrix}
    (\operatorname{cos}t)^{-1} & 0\\
    0& \operatorname{cos}t
   \end{pmatrix}
   \begin{pmatrix}
    1 & 0\\
    -\operatorname{tan}t & 1 
   \end{pmatrix}.
\end{align*}

Then through the Lie algebra homomorphism $\mathfrak{sl}_2(\mathbb{R})\rightarrow \mathfrak{g}_{\mathbb{R}}$ given by 
\begin{align*}
    \begin{pmatrix}
        0&1\\
        0&0
    \end{pmatrix}
    \mapsto u_X, \quad 
    \begin{pmatrix}
        0 & 0\\
        -1 & 0
    \end{pmatrix}
    \mapsto u_{TX}, \quad
    \begin{pmatrix}
        -1 & 0\\
        0 & 1
    \end{pmatrix}
    \mapsto H_X',
\end{align*}
we have 
\begin{align*}
    \operatorname{exp}(t(u_X+u_{TX}))=\operatorname{exp}((\operatorname{tan}t)u_X)\operatorname{exp}(\operatorname{ln}(\operatorname{cos}t)H_X') \operatorname{exp}((\operatorname{tan}t)u_{TX}).
\end{align*}
Applying $\operatorname{Ad}$ to both sides, we have 
\begin{align*}
    \operatorname{exp}(t\operatorname{ad}(u_X+u_{TX}))=E_{[X]}(\operatorname{tan}t)h_{[X]}((\operatorname{cos}t)^{-1})E_{[TX]}(\operatorname{tan}t).
\end{align*}

Similarly, since 
\begin{align*}
     \operatorname{exp}(it\begin{pmatrix}
    0&1\\
    1&0
   \end{pmatrix}) =\begin{pmatrix}
    \operatorname{cos} t & i\operatorname{sin}t\\
    i\operatorname{sin}t & \operatorname{cos} t 
   \end{pmatrix}
   = \begin{pmatrix}
    1 & i\operatorname{tan}t\\
    0 & 1
   \end{pmatrix}
   \begin{pmatrix}
    (\operatorname{cos}t)^{-1} & 0\\
    0& \operatorname{cos}t
   \end{pmatrix}
   \begin{pmatrix}
    1 & 0\\
    i\operatorname{tan}t & 1 
   \end{pmatrix},
\end{align*}
we obtain that
\begin{align*}
    \operatorname{exp}(t\operatorname{ad}i(u_X-u_{TX}))=E_{[X]}(i\operatorname{tan}t)h_{[X]}((\operatorname{cos}t)^{-1})E_{[TX]}(-i\operatorname{tan}t).
\end{align*}
\end{proof}

Now we prove some properties satisfied by the Hall numbers $\gamma_{XY}^Z$.

\begin{lemma}\label{gamma1}
    For $X,Y \in \operatorname{ind}\mathcal{R}$, denote $L_{j,k}=L_{X,Y,j,k}$.
    Assume that $p_{XY}=0$, then for any $0\leq k\leq j\leq q_{XY}$, we have
    \begin{align*}
    \gamma_{X,L_{k,1}}^{L_{k+1,1}}\cdots \gamma_{X,L_{j-1,1}}^{L_{j,1}}\gamma_{X,TL_{j,1}}^{TL_{j-1,1}} \cdots \gamma_{X,TL_{k+1,1}}^{TL_{k,1}} = (-1)^{j-k}\frac{(q_{XY}-k)!j!}{(q_{XY}-j)!k!}.
    \end{align*}
\end{lemma}

\begin{proof}
    We know that the possible values of $\gamma_{MN}^L\gamma_{M,TL}^{TN}$ can only be -1 (type $A_2$), -2 (type $B_2$), -3 and -4 (type $G_2$). Then this equality can be checked case by case.
\end{proof}

For any $X,Y\in \operatorname{ind}\mathcal{R}$, denote $p=p_{XY},q=q_{XY},L_{j,k}=L_{X,Y,j,k}$, $L_{-j,k}=L_{TX,Y,j,k}$ and $L_{-j,-k}=L_{TX,TY,j,k}$.
Constants $C_{X,Y,j,k}$ are products of some $\gamma$'s and some rational number, see \cite[Prop 4.2]{notes1}.
In particular, when $k=1$, we have $C_{X,Y,j,1}=\frac{1}{j!}\gamma_{X,Y}^{L_{1,1}}\gamma_{X,L_{1,1}}^{L_{2,1}}\cdots \gamma_{X,L_{j-1,1}}^{L_{j,1}}$.
For integer $-p \leq k \leq q$ and $t\in\mathbb{R}$, define 
\begin{align*}
    &D_{X,Y,k}(t)=\sum_{j=\operatorname{max}\{0,-k\} }^p C_{TX,Y,j,1} C_{X,L_{-j,1},j+k,1} (\operatorname{sin}t)^{2j} (\operatorname{tan}t)^k (\operatorname{cos}t)^{q-p},\\
    &D_{X,Y,k}'(t)=\sum_{j=\operatorname{max}\{0,k\}}^q C_{X,Y,j,1}C_{X,L_{-j,-1},j-k,1}(\operatorname{sin}t)^{2j}(\operatorname{tan}t)^{-k}(\operatorname{cos}t)^{p-q}.
\end{align*}

\begin{lemma}\label{gamma2}
    For any $X,Y\in \operatorname{ind}\mathcal{R}$, any integer $-p_{XY} \leq k \leq q_{XY}$, and with the notations as above,  
    \begin{align*}
        D_{X,Y,k}(t)=D_{X,Y,k}'(t)
    \end{align*}
    holds for all $t\in \mathbb{R}$.
\end{lemma}

\begin{proof}
Denote $p=p_{XY},q=q_{XY},L_{j,k}=L_{X,Y,j,k}$, $L_{-j,k}=L_{TX,Y,j,k}$ and $L_{-j,-k}=L_{TX,TY,j,k}$.

If $p=0$, then $0\leq k\leq q$, and it is equivalent to show that 
\begin{align*}
    (\operatorname{cos}t)^{2(q-k)}=\sum_{j=k}^q \frac{k!}{j!(j-k)!} \gamma_{X,L_{k,1}}^{L_{k+1,1}}\cdots \gamma_{X,L_{j-1,1}}^{L_{j,1}}\gamma_{X,TL_{j,1}}^{TL_{j-1,1}}\cdots\gamma_{X,TL_{k+1,1}}^{TL_{k,1}} (\operatorname{sin}t)^{2(j-k)}.
\end{align*}
Note that the left hand side 
\begin{align*}
    (\operatorname{cos}t)^{2(q-k)} &= (1-\operatorname{sin}^2t)^{q-k}=\sum_{l=0}^{q-k}\frac{(q-k)!}{l!(q-k-l)!} (-\operatorname{sin}^2t)^l\\
    & = \sum_{j=k}^q \frac{(q-k)!}{(j-k)!(q-j)!}(-\operatorname{sin}^2 t)^{j-k}.
\end{align*}
Then the equality follows from Lemma \ref{gamma1}.

If $q=0$, the result follows from the previous case, by replacing $Y$ by $TY$.

If $p=q=1$, then it is equivalent to prove 
\begin{align*}
    &\sum_{j=\operatorname{max} \{0,-k\} }^1  C_{TX,Y,j,1} C_{X,L_{-j,1},j+k,1} (\operatorname{sin}t)^{2j} (\operatorname{tan}t)^{2k}\\
    =& \sum_{j=\operatorname{max} \{0,k\} }^1  C_{X,Y,j,1}C_{X,L_{-j,-1},j-k,1}(\operatorname{sin}t)^{2j}.
\end{align*}
Then we check for $k=1,0,-1$.
If $k=1$, the left hand side is equal to 
\begin{align*}
    (1+\frac{\operatorname{sin}^2t}{2}\gamma_{TX,Y}^{L_{-1,1}} \gamma_{X,L_{-1,1}}^{Y})\gamma_{XY}^{L_{1,1}} \operatorname{tan}^2t = (1-\operatorname{sin}^2t)(\operatorname{tan}t )^2\gamma_{XY}^{L_{1,1}}=(\operatorname{sin}t)^2 \gamma_{XY}^{L_{1,1}},
\end{align*}
and it is equal to the right hand side.
When $k=-1$, both sides are equal to $(\operatorname{cos}t)^2 \gamma_{TX,Y}^{L_{-1,1}}$, and for $k=0$, both sides are equal to $1-2\operatorname{sin}^2 t$, thus proved this case.

If $p=2,q=1$, we have 
\begin{align*}
    &\gamma_{XY}^{L_{1,1}}=\pm 3, \quad \gamma_{TX,Y}^{L_{-1,1}}=\pm 2, \quad \gamma_{TX,L_{-1,1}}^{L_{-2,1}}=\pm 3,\\
    & \gamma_{X,TL_{1,1}}^{TY}=-\frac{1}{3}\gamma_{XY}^{L_{1,1}}, \quad \gamma_{X,L_{-1,1}}^Y = -\gamma_{TX,Y}^{L_{-1,1}}, \quad \gamma_{X,L_{-2,1}}^{L_{-1,1}}=-\frac{1}{3}\gamma_{TX,L_{-1,1}}^{L_{-2,1}}.
\end{align*}
Then the equality can be easily checked for $-2 \leq k \leq 1$.

If $p=1,q=2$, again we can replace $Y$ by $TY$ and obtain the desired equality from that of the case $p=2,q=1$.

We have exhausted all the possible values for $p$ and $q$, and thus the Lemma follows.
\end{proof}

Now we can calculate the action of $\operatorname{exp}(t\operatorname{ad}\beta_X)$ or $\operatorname{exp}(t\operatorname{ad}\xi_X)$ on $\beta_Y,\xi_Y$.

\begin{proposition}
    For $X,Y\in \operatorname{ind}\mathcal{R}$, denote $p=p_{XY},q=q_{XY},L_{jk}=L_{X,Y,j,k}$. For any $t\in \mathbb{R}$, we have 
    \begin{align*}
        &\operatorname{exp}(t\operatorname{ad}\beta_X)(\beta_Y)=\sum_{k=-p}^q D_{X,Y,k}(t) \beta_{L_{k,1}},\\
        &\operatorname{exp}(t\operatorname{ad}\beta_X)(\xi_Y)=\sum_{k=-p}^q D_{X,Y,k}(t) \xi_{L_{k,1}},\\
        &\operatorname{exp}(t\operatorname{ad}\xi_X)(\beta_Y)=\sum_{\substack{-p\leq k\leq q,\\ k \text{ odd}} } (-1)^{\frac{k-1}{2}} D_{X,Y,k}(t) \xi_{L_{k,1}} + \sum_{\substack{-p\leq k\leq q,\\ k \text{ even}} } (-1)^{\frac{k}{2}} D_{X,Y,k}(t) \beta_{L_{k,1}},\\
        &\operatorname{exp}(t\operatorname{ad}\xi_X)(\xi_Y)=\sum_{\substack{-p\leq k\leq q,\\ k \text{ odd}} } (-1)^{\frac{k+1}{2}} D_{X,Y,k}(t) \beta_{L_{k,1}} + \sum_{\substack{-p\leq k\leq q,\\ k \text{ even}} } (-1)^{\frac{k}{2}} D_{X,Y,k}(t) \xi_{L_{k,1}}.
    \end{align*}
\end{proposition}

\begin{proof}
    With the notations as above, by Lemma \ref{expbeta}, we have 
    \begin{align*}
        \operatorname{exp}&(t\operatorname{ad}(u_X+u_{TX}))u_Y=E_{[X]}(\operatorname{tan}t)h_{[X]}((\operatorname{cos}t)^{-1})E_{[TX]}(\operatorname{tan}t)u_Y\\
        &=\sum_{l=0}^p \sum_{j=0}^{q+l} C_{TX,Y,l,1}C_{X,L_{-l,1},j,1} (\operatorname{tan}t)^{l+j}(\operatorname{cos}t)^{q-p+2l}u_{L_{j-l},1}\\
        &= \sum_{k=-p}^q D_{X,Y,k}(t) u_{L_{k,1}},
    \end{align*}
    where the last equality is given by letting $k=j-l$.
    Similarly, we have 
    \begin{align*}
         \operatorname{exp}&(t\operatorname{ad}(u_X+u_{TX}))u_{TY}=E_{[X]}(\operatorname{tan}t)h_{[X]}((\operatorname{cos}t)^{-1})E_{[TX]}(\operatorname{tan}t)u_{TY}\\
         &=\sum_{l=0}^q \sum_{j=0}^{p+l}C_{X,Y,l,1}C_{X,L_{-l,-1},j,1}(\operatorname{tan}t)^{l+j}(\operatorname{cos}t)^{p-q+2l}u_{L_{j-l,-1}}\\
         &=\sum_{k=-p}^q D_{X,Y,k}'(t)u_{L_{-k,-1}},
    \end{align*}
    where the last equality is given by letting $k=l-j$.
    
    By Lemma \ref{gamma2}, we have 
    \begin{align*}
         \operatorname{exp}&(t\operatorname{ad}(u_X+u_{TX}))(u_Y+u_{TY})=\sum_{k=-p}^q D_{X,Y,k}(t) (u_{L_{k,1}}+u_{TL_{k,1}}).
    \end{align*}
    Thus by the isomorphism between $\mathfrak{u}\otimes_{\mathbb{Z}}\mathbb{R}$ and $\mathfrak{r}$, we have 
    \begin{align*}
        \operatorname{exp}(t\operatorname{ad}\beta_X)(\beta_Y)=\sum_{k=-p}^q D_{X,Y,k}(t) \beta_{L_{k,1}}.
    \end{align*}
    To calculate $\operatorname{exp}(t\operatorname{ad}\beta_X)(\xi_Y)$ we only need to replace $u_Y+u_{TY}$ by $i(u_Y-u_{TY})$.

    Similarly, we can obtain that 
    \begin{align*}
        &\operatorname{exp}(t\operatorname{ad}i(u_X-u_{TX}))u_Y
        =\sum_{k=-p}^q D_{X,Y,k}(t) i^k u_{L_{k,1}},\\
        &\operatorname{exp}(t\operatorname{ad}i(u_X-u_{TX}))u_{TY}
        =\sum_{k=-p}^q D_{X,Y,k}(t) i^{-k} u_{TL_{k,1}}.
    \end{align*}
    Thus 
    \begin{align*}
        &\operatorname{exp}(t\operatorname{ad}i(u_X-u_{TX}))(u_Y+u_{TY}) = \sum_{k=-p}^q D_{X,Y,k}(t) i^k (u_{L_{k,1}}+(-1)^k u_{TL_{k,1}})\\
        =&\sum_{\substack{-p\leq k\leq q,\\ k \text{ odd}} } (-1)^{\frac{k-1}{2}} D_{X,Y,k}(t) i(u_{L_{k,1}}-u_{TL_{k,1}}) + \sum_{\substack{-p\leq k\leq q,\\ k \text{ even}} } (-1)^{\frac{k}{2}} D_{X,Y,k}(t) (u_{L_{k,1}}+u_{TL_{k,1}}), 
    \end{align*}
    and we obtain the desired result for $\operatorname{exp}(t\operatorname{ad}\xi_X)(\beta_Y)$. 
    
    The calculation for $\operatorname{exp}(t\operatorname{ad}\xi_X)(\xi_Y)$ is similar. Thus the proof is finished.
\end{proof}

Let $\mathbf{R}$ be the subgroup of $\operatorname{Aut}(\mathfrak{r})$ generated by $\operatorname{exp}(t\operatorname{ad}\alpha_X),\operatorname{exp}(t\operatorname{ad}\beta_X)$ and $\operatorname{exp}(t\operatorname{ad}\xi_X)$, with $t\in \mathbb{R}$ and $X\in \operatorname{ind}\mathcal{R}$, and let $\mathbf{K}$ be the identity component of $\mathbf{R}$.
By the calculations above, we have the following corollary.

\begin{proposition}
    The groups $\mathbf{R}$ and $\mathbf{K}$ are closed subgroups of $\operatorname{GL}_n(\mathbb{R})$, where $n=\operatorname{dim}\mathfrak{r}$.
\end{proposition}
 
Since $\mathbf{K}$ is a connected, closed subgroup of $\operatorname{GL}_n(\mathbb{R})$, it is a closed subgroup of $\operatorname{Int}(\mathfrak{r})$.
Then $\mathbf{K}$ is compact because $\operatorname{Int}(\mathfrak{r})$ is compact, and we have shown that this kind of compact Lie group can be realized within the field $\mathbb{R}$, beyond as a group of complex matrices.

\begin{proposition}
    The two connected, compact Lie groups $\mathbf{K}$ and $\operatorname{Int}(\mathfrak{r})$ coincide.
\end{proposition}

\begin{proof}
    Since $\alpha_X,\beta_X,\xi_X$ are contained in the Lie algebra of $\mathbf{K}$, for all $X\in \operatorname{ind}\mathcal{R}$, the Lie algebra of $\mathbf{K}$ is $\mathfrak{r}\cong \operatorname{ad}\mathfrak{r}$.
    Thus $\mathbf{K}$ and $\operatorname{Int}(\mathfrak{r})$ are two connected, compact Lie groups with the same Lie algebra, and $\mathbf{K}\subseteq \operatorname{Int}(\mathfrak{r})$.
    They have the same universal covering group $\widetilde{\mathbf{K}}$, with $\widetilde{\mathbf{K}}/\mathbf{K}$ and $\widetilde{\mathbf{K}}/\operatorname{Int}(\mathfrak{r})$ both finite.
    Hence $\operatorname{Int}(\mathfrak{r})/\mathbf{K}$ is finite.
    On the other hand, $\operatorname{Int}(\mathfrak{r})/\mathbf{K}$ is connected, since $\operatorname{Int}(\mathfrak{r})$ is connected. 
    So it must be $\{1\}$, and $\mathbf{K}=\operatorname{Int}(\mathfrak{r})$. 
\end{proof}

\begin{remark}
   It is well-known that the exponential map for connected compact Lie group is surjective, see for example \cite[Cor 4.48]{Knapp}.
Thus $\operatorname{exp}(\operatorname{ad}\mathfrak{r})=\operatorname{Int}(\mathfrak{r})=\mathbf{K}$, 
and 
\begin{align*}
    \langle \operatorname{exp}(\operatorname{ad}\mathfrak{r}) \rangle = \langle \mathbf{K} \rangle \subseteq \mathbf{R} \subseteq \langle \operatorname{exp}(\operatorname{ad}\mathfrak{r}) \rangle,
\end{align*}
where $\langle S \rangle$ means the group generated by elements in $S$. So $\mathbf{R}=\langle \operatorname{exp}(\operatorname{ad}\mathfrak{r}) \rangle$.
We shall not use this fact later.
\end{remark}

\section{Representations of Compact Lie Groups}

In this section, we would like to consider the representations of our compact Lie group $\mathbf{K}$, arising from the root category $\mathcal{R}$.
We first obtain a maximal compact subgroup $(\mathbf{G}_{\mathbb{C}}^{\rho_1'})^{\circ}$ of Lusztig's reductive group $\mathbf{G}_{\mathbb{C}}$. 
Then we show that $(\mathbf{G}_{\mathbb{C}}^{\rho_1'})^{\circ}$ is simply-connected, and is the universal covering group of $\mathbf{K}$.
Moreover, the highest weight modules $\Lambda_{\lambda},\lambda\in X^+$ form a complete set of representatives of isomorphism classes of finite dimensional irreducible representations of $(\mathbf{G}_{\mathbb{C}}^{\rho_1'})^{\circ}$, based on which we obtain the Peter-Weyl theorem and Plancherel theorem for $(\mathbf{G}_{\mathbb{C}}^{\rho_1'})^{\circ}$.
Finally, use the relation between $(\mathbf{G}_{\mathbb{C}}^{\rho_1'})^{\circ}$ and $\mathbf{K}$, we also obtain the theorems for our $\mathbf{K}$.

We maintain the previous setting for the root category $\mathcal{R}$, and let the root datum $(X,Y,I,\langle, \rangle)$ corresponds to the root system $\Phi(\mathcal{R})$ of $\mathcal{R}$. 
In particular, we take $X$ to be the weight lattice of the Lie algebra $\mathfrak{g}_{\mathbb{C}}$ of the root category $\mathcal{R}$.

\subsection{Construction of $\rho_1' $ and Maximal Compact Subgroup of $\mathbf{G}_{\mathbb{C}}$}

Lusztig has defined a reductive algebraic linear group $\mathbf{G}_A$ for any algebraically closed field $A$ in \cite{zform}. 
We will take $A=\mathbb{C}$, and construct a maximal compact subgroup $(\mathbf{G}_{\mathbb{C}}^{\rho_1'})^{\circ}$ of $\mathbf{G}_{\mathbb{C}}$.

Recall that $\rho_1:\mathbf{U}\rightarrow \mathbf{U}$ is the $\mathbb{Q}(v)$-algebra isomorphism given by 
\begin{align*}
    \rho_1(E_i)=-v_iF_i,\quad \rho_1(F_i) = -v_i^{-1}E_i, \quad \rho_1(K_{\mu}) = K_{-\mu}.
\end{align*}
For each $\lambda',\lambda''\in X$, we have linear isomorphism $\rho_1:{}_{\lambda'}\mathbf{U}_{\lambda''}\rightarrow {}_{-\lambda'}\mathbf{U}_{-\lambda''}$. Take direct sums, we obtain an algebra automorphism $\rho_1:\dot{\mathbf{U}}\rightarrow \dot{\mathbf{U}}$, which satisfies $\rho_1(1_{\lambda})=1_{-\lambda}$, and $\rho_1(uxx'u')=\rho_1(u)\rho_1(x)\rho_1(x')\rho_1(u')$ for any $u,u'\in \mathbf{U},x,x'\in \dot{\mathbf{U}}$.

\begin{lemma}
    When $v=1$, we have $(\rho_1(x),\rho_1(y))=(x,y)$ for any $x,y\in \dot{\mathbf{U}}$.
\end{lemma}

\begin{proof}
    We only need to show that $x,y\mapsto (\rho_1(x),\rho_1(y))$ satisfies (1)-(3) in Thm \ref{udotpairing}. For (1), we have 
    \begin{align*}
        (\rho_1(1_{\lambda_1}x1_{\lambda_2}),\rho_1(1_{\lambda_1'}x'1_{\lambda_2'}))=(1_{-\lambda_1}\rho_1(x)1_{-\lambda_2},1_{-\lambda_1'}\rho_1(x')1_{-\lambda_2'})=0,
    \end{align*}
    unless $-\lambda_1=-\lambda_1',-\lambda_2=-\lambda_2'$.

    Since when $v=1$, we have $\rho\rho_1=\rho_1\rho$, and (2) follows. 

    As for (3), since $\rho_1(\theta_i^+)=-\theta_i^-,\rho_1(\theta_i^-)=-\theta_i^+$, we have for any $x\in \mathbf{f}_{\nu}$, $\rho_1(x^-)=(-1)^{|\nu|}x^+,\rho_1(x^+)=(-1)^{|\nu|}x^-$.
    We may assume that $x_1,x_2\in \mathbf{f}_{\nu}$ for some $\nu$, then 
    \begin{align*}
        (\rho_1(x_1^-1_{\lambda}),\rho_1(x_2^-1_{\lambda})) = (\rho_1(x_1^-)1_{-\lambda},\rho_1(x_2^-)1_{-\lambda})=(x_1^+ 1_{-\lambda},x_2^+ 1_{-\lambda}) = (x_1,x_2).
    \end{align*}
    
    By the uniqueness of the pairing, we have $(\rho_1(x),\rho_1(y))=(x,y)$ for any $x,y\in \dot{\mathbf{U}}$.
\end{proof}

\begin{proposition}
    When $v=1$, for $\beta\in \dot{\mathbf{B}}$, we have $\rho_1(\beta)\in \pm \dot{\mathbf{B}}$.
\end{proposition}

\begin{proof}
    It can be easily checked that $\rho_1$ commute with $\bar{\quad}$ when $v=1$. Moreover, since $\rho_1$ preserves $\dot{\mathbf{U}}_{\mathcal{A}}$ and $(,)$, we have the desired result.
\end{proof}

Thus when $v=1$, we can define $\rho_1:\hat{\mathbf{U}}\rightarrow \hat{\mathbf{U}}$ by 
\begin{align*}
    \rho_1(\sum_{a\in\dot{\mathbf{B}}}n_a a)=\sum_{a\in \dot{\mathbf{B}}} n_a \rho_1(a).
\end{align*}
For each $a\in \dot{\mathbf{B}}$, we denote $\rho_1(a)=p_a \tilde{a}$, where $a\mapsto \tilde{a}$ is an involution on $\dot{\mathbf{B}}$ and $p_a=\pm 1$.

Since $\rho_1(1_{\lambda})=1_{-\lambda}$, we have $p_{1_{\lambda}}=1$.
By calculating the action of $(\rho_1\otimes \rho_1) \Delta$ and $\Delta \rho_1$ on $E_i,F_i,K_{\mu}$, we obtain that 
\begin{align*}
    p_c \hat{m}_{\tilde{c}}^{\tilde{b},\tilde{a}} = p_a p_b \hat{m}_c^{a,b}
\end{align*}
for any $a,b,c\in \dot{\mathbf{B}}$.

\begin{lemma}
    Let $v=1$ in $A$, then $\rho_1:\hat{\mathbf{U}}\rightarrow \hat{\mathbf{U}}$ restricts to a group isomorphism $\rho_1:\mathbf{G}_A\rightarrow \mathbf{G}_A$.
\end{lemma}

\begin{proof}
    For any $\sum_{a\in\dot{\mathbf{B}}}n_a a\in \mathbf{G}_A$, we need to show that $\sum_{a\in\dot{\mathbf{B}}}n_a p_a \tilde{a}\in \mathbf{G}_A$. Denote $n_{\tilde{a}}'=n_ap_a$.

    Since $n_{1_0}=1,p_{1_0}=1$ and $\tilde{1}_0=1_0$, we have $n_{1_0}'=n_{1_0}p_{1_0}=1$.

    For any $a,b\in \dot{\mathbf{B}}$, since $\sum_{a\in\dot{\mathbf{B}}}n_a a\in \mathbf{G}_A$, the coefficients should satisfy that $\sum_c \hat{m}_c^{a,b}n_c=n_a n_b$.
    Then 
    \begin{align*}
         \sum_{c\in \dot{\mathbf{B}}} \hat{m}_{\tilde{c}}^{\tilde{a},\tilde{b}} n_{\tilde{c}}'= \sum_{c\in \dot{\mathbf{B}}} \hat{m}_{\tilde{c}}^{\tilde{a},\tilde{b}} n_c p_c = \sum_{c\in \dot{\mathbf{B}}} n_c p_a p_b \hat{m}_c^{b,a} = n_a n_b p_a p_b=n_{\tilde{a}}'n_{\tilde{b}}'.
    \end{align*}
    Thus $\sum_{a\in\dot{\mathbf{B}}}n_a' a\in \mathbf{G}_A$, i.e. $\rho_1(\mathbf{G}_A)\subseteq \mathbf{G}_A$. 
    So $\rho_1$ can be restricted to $\mathbf{G}_A\rightarrow \mathbf{G}_A$, which is obviously a group isomorphism.
\end{proof}

Now we consider the complexification.

When $v=1$, $\rho_1$ becomes a $\mathbb{Q}$-algebra isomorphism of $\mathbf{U}_{\mathbb{Q}}$ to itself.
By applying $-\otimes_{\mathbb{Q}}\mathbb{R}$, we have $\mathbb{R}$-algebra isomorphism $\rho_1:\mathbf{U}_{\mathbb{R}}\rightarrow \mathbf{U}_{\mathbb{R}}$, which maps $E_i$ to $-F_i$, $F_i$ to $-E_i$ and $K_{\mu}$ to $K_{-\mu}$.

Let $\mathbf{U}_{\mathbb{C}}=\mathbf{U}_{\mathbb{R}}\otimes_{\mathbb{R}}\mathbb{C}$. 
Define $\rho_1':\mathbf{U}_{\mathbb{C}} \rightarrow \mathbf{U}_{\mathbb{C}}$ by 
\begin{align*}
    \rho_1'(u\otimes z)=\rho_1(u)\otimes \bar{z},
\end{align*}
for any $u\in \mathbf{U}_{\mathbb{R}},z\in \mathbb{C}$, where $\bar{z}$ is the conjugate complex number of $z$.
Then for any $\lambda\in \mathbb{C},v\in \mathbf{U}_{\mathbb{C}}$, we have $\rho_1'(\lambda v)=\bar{\lambda}\rho_1'(v)$. 
For any $a,b\in \mathbf{U}_{\mathbb{C}}$, $\rho_1'(ab)=\rho_1'(a)\rho_1'(b)$.

Let $\dot{\mathbf{U}}_{\mathbb{C}}=\dot{\mathbf{U}}_{\mathbb{R}}\otimes_{\mathbb{R}}\mathbb{C}$. 
Then $\dot{\mathbf{B}}\otimes 1$ form a $\mathbb{C}$-basis of $\dot{\mathbf{U}}_{\mathbb{C}}$.
Similarly, we can define $\rho_1':\dot{\mathbf{U}}_{\mathbb{C}} \rightarrow \dot{\mathbf{U}}_{\mathbb{C}}$ by taking direct sums.
For any $\beta\in \dot{\mathbf{B}}$, we have
\begin{align*}
    \rho_1'(\beta \otimes 1)=\rho_1(\beta)\otimes 1 \in \pm \dot{\mathbf{B}}\otimes 1.
\end{align*}

Let $\hat{\mathbf{U}}_{\mathbb{C}}$ be the set of all formal linear combinations $\sum_{a\in \dot{\mathbf{B}}} n_a a\otimes 1$, with $n_a\in \mathbb{C}$.
Then define $\rho_1':\hat{\mathbf{U}}_{\mathbb{C}} \rightarrow\hat{\mathbf{U}}_{\mathbb{C}} $ by
\begin{align*}
    \rho_1'(\sum_{a\in \dot{\mathbf{B}}} n_a a\otimes 1) = \sum_{a\in \dot{\mathbf{B}}} \bar{n}_a \rho_1'(a\otimes 1)= \sum_{a\in \dot{\mathbf{B}}} \bar{n}_a p_a  \tilde{a}\otimes 1
\end{align*}

\begin{lemma}
    We can restrict $\rho_1':\hat{\mathbf{U}}_{\mathbb{C}}\rightarrow \hat{\mathbf{U}}_{\mathbb{C}}$ to a group isomorphism $\rho_1':\mathbf{G}_{\mathbb{C}}\rightarrow \mathbf{G}_{\mathbb{C}}$.
\end{lemma}

\begin{proof}
    For any $\sum_{a\in \dot{\mathbf{B}}} n_a a\otimes 1\in \mathbf{G}_{\mathbb{C}}$, we need to show 
    \begin{align*}
        \rho_1'(\sum_{a\in \dot{\mathbf{B}}} n_a a\otimes 1) = \sum_{a\in \dot{\mathbf{B}}} \bar{n}_a p_a  \tilde{a}\otimes 1\in \mathbf{G}_{\mathbb{C}}.
    \end{align*}
    Denote $n_{\tilde{a}}''=\bar{n}_a p_a$.
    Then $n_{1_0}''=\bar{n}_{1_0}p_{1_0}=1$.

    Since $\hat{m}_{c}^{a,b}\in \mathbb{R}$ for any $a,b,c\in \dot{\mathbf{B}}$, take the conjugate of both sides of the equality
    \begin{align*}
     \sum_{c\in \dot{\mathbf{B}}} \hat{m}_{\tilde{c}}^{\tilde{a},\tilde{b}} n_c p_c = n_a n_b p_a p_b,
    \end{align*}
    we have 
    \begin{align*}
       \sum_{c\in \dot{\mathbf{B}}} \hat{m}_{\tilde{c}}^{\tilde{a},\tilde{b}} \bar{n}_c p_c = \bar{n}_a \bar{n}_b p_a p_b,
    \end{align*}
i.e. $\sum_{c\in \dot{\mathbf{B}}} \hat{m}_{c}^{a,b} n_c'' = n_a'' n_b''$, which shows $\rho_1'(\mathbf{G}_{\mathbb{C}})\subseteq \mathbf{G}_{\mathbb{C}}$. Thus $\rho_1'$ is restricted to a group isomorphism $\rho_1':\mathbf{G}_{\mathbb{C}}\rightarrow \mathbf{G}_{\mathbb{C}}$.
\end{proof}

In particular, we consider the action of $\rho_1'$ on some speical elements of $\mathbf{G}_{\mathbb{C}}$.
For each $i\in I,h\in \mathbb{C}$,
\begin{align*}
    &\rho_1'(x_i(h))=\sum_{c\in \mathbb{N},\lambda \in X} \bar{h}^c (-1)^c \theta_i^{(c)-}1_{-\lambda}=y_i(-\bar{h}),\\
    &\rho_1'(y_i(h))=\sum_{c\in \mathbb{N},\lambda \in X} \bar{h}^c (-1)^c \theta_i^{(c)+}1_{-\lambda}=x_i(-\bar{h}).
\end{align*}
For any $\sum_{\lambda\in X}n_{\lambda}1_{\lambda}\in \mathbf{T}$, we have $\rho_1'(\sum_{\lambda\in X}n_{\lambda}1_{\lambda})=\sum_{\lambda\in X}\bar{n}_{-\lambda}1_{\lambda}\in \mathbf{T}$.

Now we want to construct a maximal compact subgroup of $\mathbf{G}_{\mathbb{C}}$ using $\rho_1'$.

Let $Q$ be the $\mathbb{Z}$-span of roots in $X$.
Since the root datum associated to $\mathbf{G}_{\mathbb{C}}$ is the same as the initial one $(X,Y,I,\langle,\rangle)$, the radical $R(\mathbf{G}_{\mathbb{C}})$ of the reductive group $\mathbf{G}_{\mathbb{C}}$ is the subgroup of $\mathbf{T}_{\mathbb{C}}$ generated by the groups $\operatorname{Im}y$, with $y\in Y$ satisfying that $\langle y, Q \rangle =\{0\}$.
Since $X$ is taken to be the weight lattice of $\mathfrak{g}_{\mathbb{C}}$ and the pairing $\langle,\rangle$ is perfect, we have $R(\mathbf{G}_{\mathbb{C}})=\{1\}$. 
Thus under our chosen root datum, $\mathbf{G}_{\mathbb{C}}$ is a semisimple linear algebraic group with finite center.
The Lie algebra $\operatorname{Lie}(\mathbf{G}_{\mathbb{C}})$ of $\mathbf{G}_{\mathbb{C}}$ is isomorphic to the Lie algebra $\mathfrak{g}_{\mathbb{C}}$. 

The group isomorphism $\rho_1':\mathbf{G}_{\mathbb{C}}\rightarrow \mathbf{G}_{\mathbb{C}}$ naturally induces a Lie algebra isomorphism $d\rho_1':\operatorname{Lie}(\mathbf{G}_{\mathbb{C}})\rightarrow \operatorname{Lie}(\mathbf{G}_{\mathbb{C}})$, which is an involution on $\operatorname{Lie}(\mathbf{G}_{\mathbb{C}})$.
Moreover, it coincides with the $\theta$ we defined in the previous section, under the isomorphism $\operatorname{Lie}(\mathbf{G}_{\mathbb{C}}) \cong \mathfrak{g}_{\mathbb{C}}$.
Let $\operatorname{Lie}(\mathbf{G}_{\mathbb{C}})^{d\rho_1'}$ be the set of $d\rho_1'$-fixed points of $\operatorname{Lie}(\mathbf{G}_{\mathbb{C}})$.
Then it is isomorphic to $\mathfrak{r}$, and is thus a compact real form of $\operatorname{Lie}(\mathbf{G}_{\mathbb{C}})$. 

Let $\mathbf{G}_{\mathbb{C}}^{\rho_1'}$ be the set of $\rho_1'$-fixed points of $\mathbf{G}_{\mathbb{C}}$. 
It is a closed subgroup of $\mathbf{G}_{\mathbb{C}}$ with Lie algebra $\operatorname{Lie}(\mathbf{G}_{\mathbb{C}})^{d\rho_1'}$.
However, it may not be connected.
So we consider the identity component $(\mathbf{G}_{\mathbb{C}}^{\rho_1'})^{\circ}$ of $\mathbf{G}_{\mathbb{C}}^{\rho_1'}$.

\begin{lemma}
    The center $\mathbf{Z}(\mathbf{G}_{\mathbb{C}})$ of $\mathbf{G}_{\mathbb{C}}$ is contained in $(\mathbf{G}_{\mathbb{C}}^{\rho_1'})^{\circ}$.
\end{lemma}

\begin{proof}
    Since $\mathbf{Z}(\mathbf{G}_{\mathbb{C}})\subseteq \mathbf{T}_{\mathbb{C}}$, any element in $\mathbf{Z}(\mathbf{G}_{\mathbb{C}})$ is of the form $\sum_{\lambda\in X}n_{\lambda}1_{\lambda}$, and 
    \begin{align*}
        \rho_1'(\sum_{\lambda\in X}n_{\lambda}1_{\lambda})=\sum_{\lambda\in X}\bar{n}^{-1}_{\lambda}1_{\lambda}=\sum_{\lambda\in X}\frac{n_{\lambda}}{|n_{\lambda}|^2}1_{\lambda}.
    \end{align*}
    Since the element is in $\mathbf{Z}(\mathbf{G}_{\mathbb{C}})$, for each root $\alpha$, we have $\chi_{\alpha}(\sum_{\lambda\in X}n_{\lambda}1_{\lambda})= n_{\alpha}=1$. Then $n_{\mu}=1$ for all $\mu\in Q$.
    If $\mu \in X$ is not in $Q$, then there exists an integer $c$ such that $c\mu \in Q$. Then $n_{\mu}^c=n_{c\mu}=1$, and $\frac{n_{\mu}}{|n_{\mu}|^2}=n_{\mu}$. 
    So we can see that $\rho_1'(\sum_{\lambda\in X}n_{\lambda}1_{\lambda})=\sum_{\lambda\in X}n_{\lambda}1_{\lambda}$, i.e. $\mathbf{Z}(\mathbf{G}_{\mathbb{C}})$ is contained in $\mathbf{G}_{\mathbb{C}}^{\rho_1'}$.

    Moreover, we show that there is a path in $\mathbf{G}_{\mathbb{C}}^{\rho_1'}$ from the identity to any element of $\mathbf{Z}(\mathbf{G}_{\mathbb{C}})$.
    We can write an element $z$ of $\mathbf{Z}(\mathbf{G}_{\mathbb{C}})$ as 
    \begin{align*}
        z=\sum_{\lambda\in Q} 1_{\lambda}+\sum_{\mu\in X-Q}e^{i\theta_{\mu}} 1_{\mu},
    \end{align*}
    where $\theta_{\mu}\in [0,2\pi)$.
    Define a continuous function $\gamma:\mathbb{R}\rightarrow \mathbf{G}_{\mathbb{C}}$ by 
    \begin{align*}
        \gamma(t)=\sum_{\lambda\in Q} 1_{\lambda}+\sum_{\mu\in X-Q}e^{it\theta_{\mu}} 1_{\mu}.
    \end{align*}
    Then $\gamma(0)=1,\gamma(1)=z$, and
    \begin{align*}
        \rho_1'(\gamma(t))=\sum_{\lambda\in Q} 1_{\lambda}+\sum_{\mu\in X-Q}\frac{e^{it\theta_{\mu}}}{1} 1_{\mu}=\gamma(t),
    \end{align*}
i.e. $\gamma(t)\in \mathbf{G}_{\mathbb{C}}^{\rho_1'}$ for all $t\in \mathbb{R}$. 
Thus $\gamma$ is a path in $\mathbf{G}_{\mathbb{C}}^{\rho_1'}$ connecting $1$ and $z$, which shows $z\in (\mathbf{G}_{\mathbb{C}}^{\rho_1'})^{\circ}$. Hence $\mathbf{Z}(\mathbf{G}_{\mathbb{C}})$ is contained in $(\mathbf{G}_{\mathbb{C}}^{\rho_1'})^{\circ}$. 
\end{proof}

\begin{proposition}
    The subgroup $(\mathbf{G}_{\mathbb{C}}^{\rho_1'})^{\circ}$ is a maximal compact subgroup of $\mathbf{G}_{\mathbb{C}}$.
\end{proposition}

\begin{proof}
    For each $x\in \mathbf{G}_{\mathbb{C}}$, the map $\operatorname{Int}(x):y\mapsto xyx^{-1}$ is an automorphism of $\mathbf{G}_{\mathbb{C}}$, and $\operatorname{Ad}(x)$ is its differential at the identity.
    Note that since $\mathbf{G}_{\mathbb{C}}$ is connected, the kernel of the homomorphism $\operatorname{Ad}:\mathbf{G}_{\mathbb{C}}\rightarrow \operatorname{Aut}(\operatorname{Lie}(\mathbf{G}_{\mathbb{C}}))$ is exactly the center of $\mathbf{G}_{\mathbb{C}}$, which is finite.

    For $x\in \mathbf{G}_{\mathbb{C}}^{\rho_1'}$, we have $\rho_1' \operatorname{Int}(x)=\operatorname{Int}(x)\rho_1'$. 
    Taking the differential at the identity, we have $d\rho_1'\operatorname{Ad}(x)=\operatorname{Ad}(x)d\rho_1'$.
    Thus for each $Y\in \operatorname{Lie}(\mathbf{G}_{\mathbb{C}})^{d\rho_1'}$, we have
    \begin{align*}
        d\rho_1'(\operatorname{Ad}(x)Y)=\operatorname{Ad}(x)d\rho_1'(Y)=\operatorname{Ad}(x)Y.
    \end{align*}
    So $\operatorname{Ad}(x)\in \operatorname{Aut}(\operatorname{Lie}(\mathbf{G}_{\mathbb{C}})^{d\rho_1'})$. 
    Since $\operatorname{Lie}(\mathbf{G}_{\mathbb{C}})^{d\rho_1'}$ is a compact real Lie algebra, $\operatorname{Ad}(x)$ preserves its Killing form, which is negative-definite. Thus $\operatorname{Ad}(x)$ is contained in the orthogonal group on $\operatorname{Lie}(\mathbf{G}_{\mathbb{C}})^{d\rho_1'}$.

    Since $(\mathbf{G}_{\mathbb{C}}^{\rho_1'})^{\circ}$ is connected with Lie algebra $\operatorname{Lie}(\mathbf{G}_{\mathbb{C}})^{d\rho_1'}$, we have $\operatorname{Ad}((\mathbf{G}_{\mathbb{C}}^{\rho_1'})^{\circ}) = \operatorname{Int}(\operatorname{Lie}(\mathbf{G}_{\mathbb{C}})^{d\rho_1'})$, which is compact.
    Then since $\operatorname{Ad}:(\mathbf{G}_{\mathbb{C}}^{\rho_1'})^{\circ} \rightarrow \operatorname{Ad}((\mathbf{G}_{\mathbb{C}}^{\rho_1'})^{\circ})$ is a finite covering homomorphism, we have $(\mathbf{G}_{\mathbb{C}}^{\rho_1'})^{\circ}$ is compact.

    If $M$ is a compact subgroup of $\mathbf{G}_{\mathbb{C}}$ and containing $(\mathbf{G}_{\mathbb{C}}^{\rho_1'})^{\circ}$, then applying the homomorphism $\operatorname{Ad}$, we have $\operatorname{Ad}((\mathbf{G}_{\mathbb{C}}^{\rho_1'})^{\circ})\subseteq \operatorname{Ad}(M)$. 
    Since $\operatorname{Ad}((\mathbf{G}_{\mathbb{C}}^{\rho_1'})^{\circ}) = \operatorname{Int}(\operatorname{Lie}(\mathbf{G}_{\mathbb{C}})^{d\rho_1'})\cong \operatorname{Int}(\mathfrak{r})$ is a maximal compact subgroup of $\operatorname{Ad}(\mathbf{G}_{\mathbb{C}})$, and $\operatorname{Ad}(M)$ is compact, we have $\operatorname{Ad}((\mathbf{G}_{\mathbb{C}}^{\rho_1'})^{\circ})= \operatorname{Ad}(M)$, and hence $M\subseteq (\mathbf{G}_{\mathbb{C}}^{\rho_1'})^{\circ} \cdot \mathbf{Z}(\mathbf{G}_{\mathbb{C}}) = (\mathbf{G}_{\mathbb{C}}^{\rho_1'})^{\circ} $. So $M=(\mathbf{G}_{\mathbb{C}}^{\rho_1'})^{\circ}$, and thus $(\mathbf{G}_{\mathbb{C}}^{\rho_1'})^{\circ}$ is a maximal compact subgroup of $\mathbf{G}_{\mathbb{C}}$.
\end{proof}

\subsection{Hilbert Space $\hat{\mathbf{O}}_{\mathbb{C}}$}
In this subsection, we construct a Hilbert Space $\hat{\mathbf{O}}_{\mathbb{C}}$ using finite dimensional representations of the maximal compact subgroup $(\mathbf{G}_{\mathbb{C}}^{\rho_1'})^{\circ}$. 

Define $S:\mathbf{U}_{\mathbb{C}} \rightarrow \mathbf{U}_{\mathbb{C}}$ to be $S\otimes 1$, and homomorphism $\rho':\mathbf{U}_{\mathbb{C}} \rightarrow \mathbf{U}_{\mathbb{C}}^{op}$ to be $\rho_1'S$. 
It satisfies that $ \rho'(u\otimes z)=\rho(u)\otimes \bar{z}$. We still have $(\rho')^2=1$.

For each $\lambda\in X^+$, we have highest weight module $\Lambda_{\lambda}$. 
Take $v=1$ in $\mathbb{Q}(v)$, and let $\Lambda_{\lambda,\mathbb{C}}=\Lambda_{\lambda}\otimes \mathbb{C}$.
Define $(,)_{\lambda}:\Lambda_{\lambda,\mathbb{C}} \times \Lambda_{\lambda,\mathbb{C}} \rightarrow \mathbb{C}$ such that 
\begin{align*}
    (x\otimes z_1,y\otimes z_2)_{\lambda}=z_1\bar{z}_2(x,y),
\end{align*}
for all $x,y\in \Lambda_{\lambda}, z_1,z_2\in \mathbb{C}$,  where $(,)$ is the unique bilinear form defined in Prop \ref{modpairing}.

\begin{lemma}
    The form $(,)_{\lambda}$ satisfies the following properties:

    (1) $(\eta_{\lambda},\eta_{\lambda})_{\lambda}=1$, 

    (2) For any $u\in \mathbf{U}_{\mathbb{C}}$, $x,y\in \Lambda_{\lambda,\mathbb{C}}$, we have $(ux,y)_{\lambda}=(x,\rho'(u)y)_{\lambda}$.

    (3) For any $x,y,z\in \Lambda_{\lambda,\mathbb{C}}$ and $c\in \mathbb{C}$, we have $(x,y)_{\lambda}=\overline{(y,x)}_{\lambda}$, $(cx+y,z)_{\lambda}=c(x,z)_{\lambda}+(y,z)_{\lambda}$. 
\end{lemma}

\begin{proof}
    The properties follow directly from the definition.
\end{proof}

Then we want to show that $(,)_{\lambda}$ is positive-definite, and thus is an inner product on $\Lambda_{\lambda,\mathbb{C}}$.
We simply denote $\Lambda_{\lambda}$ by $\Lambda$, $\Lambda_{\lambda,\mathbb{C}}$ by $\Lambda_{\mathbb{C}}$, and $\eta_{\lambda}$ by $\eta$.

Let $L(\Lambda)$ be the $\mathbb{A}$-submodule of $\Lambda$ generated by its signed basis, see \cite[18.1.8]{L1}.
It has direct sum decomposition $L(\Lambda)=\sum_{\nu}L(\Lambda)_{\nu}$ and $L(\Lambda)_{\nu}\subseteq (\Lambda)_{\nu}$ is contained in the $(\lambda-\nu)$-weight space, where $(\Lambda)_{\nu}$ is the image of $\mathbf{f}_{\nu}$ under the canonical map $\mathbf{f}\rightarrow \Lambda$.

\begin{lemma}\label{+}
    Let $v=1$. For any $\nu\in \mathbb{N}[I]$ and any $x\in L(\Lambda)_{\nu}$, we have $(x,x)\geq 0$.
    Moreover, if $x\in L(\Lambda)_{\nu}$ satisfies $(x,x)=0$, then $x=0$.
\end{lemma}
 
\begin{proof}
    We prove this Lemma by induction on $\operatorname{tr}\nu =N$.

    If $N=0$, then $(\eta,\eta)=1>0$.

    If $N>0$, assume the results hold for all $\nu'$ with $\operatorname{tr}\nu'<N$.
    For any $j\in I$ such that $\nu_j>0$, by \cite[18.2.2]{L1}, we may assume that $x=F_j^{(s)}y$ for some $y\in L(\Lambda)_{\nu-sj}$, satisfying that $E_j y=0$, $s\geq 0$ and $s+\langle j,\lambda-\nu\rangle \geq 0$. Then 
    \begin{align*}
        (x,x)&=(F_j^{(s)}y,F_j^{(s)}y) \\
            &=v_j^{s^2-s\langle j,\lambda-\nu+sj'\rangle} 
            \begin{bmatrix}
                \langle j,\lambda-\nu+sj'\rangle \\
                s
            \end{bmatrix}_j (y,y)\\
            &=\begin{bmatrix}
                2s+\langle j,\lambda-\nu \rangle\\
                s
            \end{bmatrix}(y,y).
    \end{align*}
    We have $\begin{bmatrix}
                2s+\langle j,\lambda-\nu \rangle\\
                s
            \end{bmatrix}>0$, since $s\geq 0$ and $s+\langle j,\lambda-\nu\rangle \geq 0$. 
    By induction hypothesis, $(y,y)\geq 0$.
    Thus $(x,x)\geq 0.$

    If $x\in L(\Lambda)_{\nu}$ satisfies $(x,x)=0$ and if $x\neq 0$, then there exists some $j\in I$ such that $\nu_i>0$. 
    Again, we assume $x=F_j^{(s)}y$, and obtain that 
    $\begin{bmatrix}
                2s+\langle j,\lambda-\nu \rangle\\
                s
            \end{bmatrix}(y,y)=0$. Since $\begin{bmatrix}
                2s+\langle j,\lambda-\nu \rangle\\
                s
            \end{bmatrix}>0$, we have $(y,y)=0$.
    By induction hypothesis, $y=0$, which contradicts to the assumption that $F_j^{(s)}y=x\neq 0$.
    Thus $x=0$ and the lemma follows. 
\end{proof}

\begin{proposition}
    Let $v=1$. Then $(,)_{\lambda}$ is positive-definite on $\Lambda_{\mathbb{C}}$, and hence is an inner product.
\end{proposition}

\begin{proof}
    If $x\in L(\Lambda)_{\nu}\subseteq (\Lambda)_{\nu}, y\in L(\Lambda)_{\mu}\subseteq (\Lambda)_{\mu}$ and $\nu\neq\mu$, then $(x,y)=0$.

    For any $x\in L(\Lambda)=\sum_{\nu}L(\Lambda)_{\nu}$, we write $x=\sum_{\nu}x_{\nu}$ and $x_{\nu}\in L(\Lambda)_{\nu}$ for each $\nu\in \mathbb{N}[I]$.
    By Lemma \ref{+}, $(x,x)=(\sum_{\nu}x_{\nu},\sum_{\nu}x_{\nu})=\sum_{\nu}(x_{\nu},x_{\nu})\geq 0$.
    If $x\in L(\Lambda)$ satisfies that $(x,x)=0$, then $\sum_{\nu}(x_{\nu},x_{\nu})=0$. Since every summand is larger or equal to 0, we have $(x_{\nu},x_{\nu})=0$ for any $\nu$.
    Thus $x_{\nu}=0$ and $x=0$.
    So $(,)$ is positive-definite on $L(\Lambda)$.

    By definition, $L(\Lambda)$ is the $\mathbb{A}$-submodule of $\Lambda$ generated by its signed basis, so $(,)$ is also positive-definite on $\Lambda$. 
    Thus $(,)_{\lambda}$ is positive-definite on $\Lambda_{\mathbb{C}}$, and hence is an inner product on $\Lambda_{\mathbb{C}}$.
\end{proof}

The following Lemma shows that this inner product makes $\Lambda_{\mathbb{C}}$ a unitary representation for our group $(\mathbf{G}_{\mathbb{C}}^{\rho_1'})^{\circ}$.

\begin{lemma}\label{Ginvar}
    The inner product $(,)_{\lambda}$ is $\mathbf{G}_{\mathbb{C}}^{\rho_1'}$-invariant.
\end{lemma}

\begin{proof}
    For any $u\in \mathbf{G}_{\mathbb{C}}$ and $x,y\in \Lambda_{\mathbb{C}}$, we have 
    \begin{align*}
        (ux,y)_{\lambda}=(x,\rho'(u)y)_{\lambda}=(x,\rho_1'S(u)y)_{\lambda}=(x,\rho_1'(u^{-1})y)_{\lambda}.
    \end{align*}
    If $u$ is in $\mathbf{G}_{\mathbb{C}}^{\rho_1'}$, then $\rho_1'(u^{-1})=u^{-1}$, and thus $(ux,y)_{\lambda}=(x,u^{-1}y)_{\lambda}$, i.e. $(,)_{\lambda}$ is $\mathbf{G}_{\mathbb{C}}^{\rho_1'}$-invariant.
\end{proof}

For any $M\in \mathcal{C}_{\mathbb{C}}$, we define $C_M$ be the $\mathbb{C}$-subspace of $\dot{\mathbf{U}}_{\mathbb{C}}^{\diamond}$ spanned by the functions $u\mapsto z'(uz)$ for some $z\in M,z'\in M^*$.
For any $L\in \Lambda^*$, there exists a unique $x\in \Lambda$ such that $L(-)=(-,x)_{\lambda}$. We identify $L$ with $x$.
Under this identification, we denote function $f_{z,z'}(u)= z'(uz)$ in $C_{\Lambda}$ by $u\mapsto (uz,z')_{\lambda}$.

From now on, we simply denote $\Lambda_{\lambda,\mathbb{C}}$ by $\Lambda_{\lambda}$.
Let $\sum_{\lambda\in X^+}C_{\Lambda_{\lambda}}$ be the set of formal linear sums $\sum_{\lambda\in X^+} f_{\lambda}$, each $f_{\lambda}\in C_{\Lambda_{\lambda}}$. 
We define a pairing $(,)$ on $\sum_{\lambda\in X^+}C_{\Lambda_{\lambda}}$ as follows.

For $f_{\lambda}\in C_{\Lambda_{\lambda}},f_{\mu}\in C_{\Lambda_{\mu}}$ and $\lambda\neq \mu$, let 
\begin{align*}
    (f_{\lambda},f_{\mu})=0.
\end{align*}

For $f_{z_1,z_1'},f_{z_2,z_2'}\in C_{\Lambda_{\lambda}}$, $z_1,z_1',z_2,z_2'\in \Lambda_{\lambda}$, let
\begin{align*}
    (f_{z_1,z_1'},f_{z_2,z_2'})=\frac{(z_1,z_2)_{\lambda}(z_2',z_1')_{\lambda}}{\operatorname{dim}\Lambda_{\lambda}}.
\end{align*}
Denote $\|f_{z_1,z_1'}\|=\sqrt{(f_{z_1,z_1'},f_{z_1,z_1'})}$, which is by definition well-defined.

Note that this pairing may not converge on the whole space $\sum_{\lambda\in X^+}C_{\Lambda_{\lambda}}$. Thus we consider the following subset instead. Let
\begin{align*}
    \hat{\mathbf{O}}_{\mathbb{C}}=\{f=\sum_{\lambda\in X^+} f_{\lambda}\in \sum_{\lambda\in X^+}C_{\Lambda_{\lambda}} \mid (f,f)=\sum_{\lambda\in X^+}(f_{\lambda},f_{\lambda})<\infty \}.
\end{align*}
By Lemma \ref{matrixcoeff} (extend the coefficients to $\mathbb{C}$), it is clear that $\mathbf{O}_{\mathbb{C}}$ is contained in $\hat{\mathbf{O}}_{\mathbb{C}}$.

\begin{proposition}
    The pairing $(,)$ is an inner product on $\hat{\mathbf{O}}_{\mathbb{C}}$.
\end{proposition}

\begin{proof}
    Firstly, we show that the pairing is well-defined on $\hat{\mathbf{O}}_{\mathbb{C}}$.

    Since $(,)_{\lambda}$ is an inner product on $\Lambda_{\lambda}$, we denote $\|x\|=\sqrt{(x,x)_{\lambda}}$ for each $x\in \Lambda_{\lambda}$, and have Schwarz inequality, i.e. $|(x,y)_{\lambda}|\leq \|x\|\cdot \|y\|$, for any $x,y\in \Lambda_{\lambda}$.
    Then for any $f_{z_1,z_1'},f_{z_2,z_2'}\in C_{\Lambda_{\lambda}}$, we have 
    \begin{align*}
        |(f_{z_1,z_1'},f_{z_2,z_2'})|=\left|\frac{(z_1,z_2)_{\lambda}(z_2',z_1')_{\lambda}}{\operatorname{dim}\Lambda_{\lambda}}\right| \leq \frac{\|z_1\|\cdot\|z_2\|\cdot\|z_1'\|\cdot\|z_2'\|}{\operatorname{dim}\Lambda_{\lambda}}=\|f_{z_1,z_1'}\|\cdot\|f_{z_2,z_2'}\|.
    \end{align*}
    So for any $f=\sum_{\lambda}f_{\lambda},g=\sum_{\mu}g_{\mu}\in \hat{\mathbf{O}}_{\mathbb{C}}$, 
    \begin{align*}
        |(f,g)|=&\left| \sum_{\lambda}(f_{\lambda},g_{\lambda})  \right| \leq \sum_{\lambda} \left| (f_{\lambda},g_{\lambda}) \right| \\
        &\leq \sum_{\lambda} \| f_{\lambda} \|\cdot \| g_{\lambda} \| \leq (\sum_{\lambda}\|f_{\lambda}\|^2)^{\frac{1}{2}}(\sum_{\lambda}\|g_{\lambda}\|^2)^{\frac{1}{2}} <\infty.
    \end{align*}
Thus $(,)$ on $\hat{\mathbf{O}}_{\mathbb{C}}$ is well-defined.

Then by definition we can easily deduce that for any $f,g,h\in \hat{\mathbf{O}}_{\mathbb{C}}$ and $c\in \mathbb{C}$, we have $(cf,g)=c(f,g)$, $(f+g,h)=(f,h)+(g,h)$ and $(f,g)=\overline{(g,f)}$. 
For any $f=\sum_{\lambda}f_{\lambda}\in \hat{\mathbf{O}}_{\mathbb{C}}$, it is clear that $(f,f)\geq 0$.
If $(f,f)=0$, then for any $\lambda\in X^+$, we have $(f_\lambda,f_{\lambda})=0$. 
In particular, if $0=(f_{z,z'},f_{z,z'})=\frac{(z,z)_{\lambda}(z',z')_{\lambda}}{\operatorname{dim}\Lambda_{\lambda}}$, then $(z,z)_{\lambda}=0$ or $(z',z')_{\lambda}=0$. So $z=0$ or $z'=0$. Both cases leads to $f_{z,z'}=0$.
Since $f_{z,z'}$, $z,z'\in \Lambda_{\lambda}$ span the whole $C_{\Lambda_{\lambda}}$, we have $f_{\lambda}=0$.
Thus $f=\sum_{\lambda}f_{\lambda}=0$. This shows $(,)$ is positive-definite, and thus is an inner product on $\hat{\mathbf{O}}_{\mathbb{C}}$.
\end{proof}

\begin{proposition}
    The inner product space $\hat{\mathbf{O}}_{\mathbb{C}}$ is a Hilbert space.
\end{proposition}

\begin{proof}
    For any $x\in \hat{\mathbf{O}}_{\mathbb{C}}$, define the norm $\|x\| = \sqrt{(x,x)}$. Then $\hat{\mathbf{O}}_{\mathbb{C}}$ is a normed vector space with respect to $\|\cdot \|$.
    We need to show it is complete, i.e. every Cauchy sequence in $\hat{\mathbf{O}}_{\mathbb{C}}$ converges to a point in $\hat{\mathbf{O}}_{\mathbb{C}}$.

    Firstly, we consider a Cauchy sequence $\{ f_n \}_{n\geq 1}$ in $\hat{\mathbf{O}}_{\mathbb{C}}$, where each $f_n\in C_{\Lambda_{\lambda}}$, for an arbitrarily fixed  $\lambda\in X^+$.
    That is to say, for any $\varepsilon >0$, there exists $N>0$, such that for any $m,n>N$, we have $\| f_m-f_n \|<\varepsilon $.

    Note that Lemma \ref{matrixcoeff} can be extended to be over field $\mathbb{C}$. 
    Then by the proof of the Lemma, we know that any $f\in C_{\Lambda_{\lambda}}$ can be written as $f=\sum_{b\in \bigcup_{\lambda\geq \lambda'}\dot{\mathbf{B}}[\lambda']} f(b)b^*$.
    Every $b\in \bigcup_{\lambda\geq \lambda'}\dot{\mathbf{B}}[\lambda']$ is a linear operator on the finite dimensional vector space $\Lambda_{\lambda}$, thus is bounded. For any $x\in \Lambda_{\lambda}$, we have $\|bx\|\leq M_b \|x\|$ for some $M_b$.
    Since $\bigcup_{\lambda\geq \lambda'}\dot{\mathbf{B}}[\lambda']$ is a finite set, we let $M=\operatorname{max}_{b\in \bigcup_{\lambda\geq \lambda'}\dot{\mathbf{B}}[\lambda']} M_b$. Then $\|bx\|\leq M \|x\|$ holds for any $b\in \bigcup_{\lambda\geq \lambda'}\dot{\mathbf{B}}[\lambda']$ and $x\in \Lambda_{\lambda}$.

    We pick a subsequence of $\{f_n\}_{n\geq 1}$. For each $k\in \mathbb{N}$, there exists $n_k\in \mathbb{N}$, such that 
    \begin{align*}
        \| f_{n_k}-f_{n_{k+1}} \| < \frac{1}{2^k} \cdot \frac{1}{M(\operatorname{dim}\Lambda_{\lambda})^{\frac{1}{2}}}.
    \end{align*}

    For any $f\in C_{\Lambda_{\lambda}}$, we may assume $f=f_{z,z'}$ for some $z,z'\in \Lambda_{\lambda}$. If $\|f\|<\varepsilon $, i.e. $\|z\|\cdot\|z'\|<\varepsilon (\operatorname{dim}\Lambda_{\lambda})^{\frac{1}{2}}$, then 
    \begin{align*}
        |f(b)|=|(bz,z')_{\lambda}|\leq \|bz\|\cdot \|z'\| \leq M \cdot \|z\|\cdot\|z'\|<M\varepsilon (\operatorname{dim}\Lambda_{\lambda})^{\frac{1}{2}}
    \end{align*}
    Replace $f$ by $f_{n_k}-f_{n_{k+1}}$, we have 
    \begin{align*}
        |f_{n_k}(b)-f_{n_{k+1}}(b)|< M(\operatorname{dim}\Lambda_{\lambda})^{\frac{1}{2}}\cdot \frac{1}{2^k} \frac{1}{M(\operatorname{dim}\Lambda_{\lambda})^{\frac{1}{2}}}=\frac{1}{2^k}, 
    \end{align*}
    for any $b\in \bigcup_{\lambda\geq \lambda'}\dot{\mathbf{B}}[\lambda']$. 
    The sequence $\{f_{n_k}(b)\}_{k\geq 1}$ is a sequence in $\mathbb{C}$, and thus converges. We denote its limit by $h(b)$.

    So we obtained a complex number $h(b)$ for each $b\in \bigcup_{\lambda\geq \lambda'}\dot{\mathbf{B}}[\lambda']$. Let
    \begin{align*}
        h=\sum_{b\in \bigcup_{\lambda\geq \lambda'}\dot{\mathbf{B}}[\lambda']} h(b)b^*.
    \end{align*}
    Then $h$ is an element in $\mathbf{O}_{\mathbb{C}}$, and
    \begin{align*}
        \|f_{n_k}-h\|=\|\sum_{b\in \bigcup_{\lambda\geq \lambda'}\dot{\mathbf{B}}[\lambda']} (f_{n_k}(b)-h(b)) b^* \| \leq \sum_{b\in \bigcup_{\lambda\geq \lambda'}\dot{\mathbf{B}}[\lambda']} |f_{n_k}(b)-h(b)|\cdot \|b^*\|\rightarrow 0
    \end{align*}
    when $k\rightarrow \infty$.
    Thus $f_{n_k}$ converges to $h$ when $k\rightarrow \infty$.

    For any $\varepsilon>0$, there exists $N,k>0$, such that when  $n,n_k>N$, we have
    \begin{align*}
        \|f_{n_k}-f_n\|<\varepsilon.
    \end{align*}
    For each fixed $n>N$, let $k\rightarrow \infty$, we have $\|h-f_n\|<\varepsilon$. Thus the sequence $\{f_n\}_{n\geq 1}$ converges to $h\in \mathbf{O}_{\mathbb{C}}\subseteq \hat{\mathbf{O}}_{\mathbb{C}}$.

    Then we consider the general case. 
    Suppose $\{f_n\}_{n\geq 1}$ is a Cauchy sequence in $\hat{\mathbf{O}}_{\mathbb{C}}$, and we can write each $f_n=\sum_{\lambda}f_n^{\lambda}$ with $f_n^{\lambda}\in C_{\Lambda_{\lambda}}$.
    Since $\{f_n\}_{n\geq 1}$ is a Cauchy sequence, for any $\varepsilon>0$, there exists $N>0$, such that for any $n,m>N$, we have $\|f_n-f_m\|<\varepsilon$, i.e. 
    \begin{align*}
        \|f_n-f_m\|=\| \sum_{\lambda}(f_n^{\lambda}-f_m^{\lambda}) \| =\sum_{\lambda} \|f_n^{\lambda}-f_m^{\lambda}\| <\varepsilon.
    \end{align*}
    So for any $\lambda\in X^+$, we have $\|f_n^{\lambda}-f_m^{\lambda}\| <\varepsilon$.
    Then $\{f_n^{\lambda}\}_{n\geq 1}$ form a Cauchy sequence with each $f_n^{\lambda}\in C_{\Lambda_{\lambda}}$, and thus converges to some element $h^{\lambda}\in \mathbf{O}_{\mathbb{C}}$. 

    Let $h=\sum_{\lambda}h^{\lambda}$. 
    Note that $h^{\lambda}$ may not contained in $C_{\Lambda_{\lambda}}$.
    However, since $\lambda\in X^+ $ is countable, we still have 
    \begin{align*}
        \|f_n-h\|=\|\sum_{\lambda} (f_n^{\lambda}-h^{\lambda})\|\leq \sum_{\lambda}\|f_n^{\lambda}-h^{\lambda}\| \rightarrow 0,
    \end{align*}
    when $n\rightarrow \infty$.
    Thus $\{f_n\}_{n\geq 1} $ converges to $h$.
    On the other hand, since 
    \begin{align*}
        \|h\|=\|h-f_n+f_n\|\leq \|h-f_n\|+\|f_n\|<\infty,
    \end{align*}
    we have $h\in \hat{\mathbf{O}}_{\mathbb{C}}$.
    So we have shown that every Cauchy sequence in $ \hat{\mathbf{O}}_{\mathbb{C}}$ converges to some element in $ \hat{\mathbf{O}}_{\mathbb{C}}$, which means $\hat{\mathbf{O}}_{\mathbb{C}}$ is complete. Hence $ \hat{\mathbf{O}}_{\mathbb{C}}$ is a Hilbert space.
\end{proof}

\subsection{Peter-Weyl Theorem}

In this subsection, we combine the maximal compact subgroups and the Hilbert space together, and obtain the Peter-Weyl theorem for compact groups $(\mathbf{G}_{\mathbb{C}}^{\rho_1'})^{\circ}$ and $\mathbf{K}$.

Let $\mathbf{T}'$ be the set of $\rho_1'$-fixed points in $\mathbf{T}_{\mathbb{C}}$. Since $\rho_1'(\sum_{\lambda\in X}n_{\lambda}1_{\lambda})=\sum_{\lambda\in X} \frac{n_{\lambda}}{|n_{\lambda}|^2}1_{\lambda}$, we have 
\begin{align*}
    \mathbf{T}'=\{ \sum_{\lambda\in X}n_{\lambda}1_{\lambda}\in \mathbf{G}_{\mathbb{C}} \mid |n_{\lambda}|=1, \forall \lambda\in X \}.
\end{align*}
Since $\mathbf{T}_{\mathbb{C}}$ is connected and $\mathbf{T}'=\mathbf{T}_{\mathbb{C}}\bigcap \mathbf{G}_{\mathbb{C}}^{\rho_1'}$, $\mathbf{T}'$ is contained in $(\mathbf{G}_{\mathbb{C}}^{\rho_1'})^{\circ}$.

Let $\mathfrak{t}$ be the Lie algebra of $\mathbf{T}_{\mathbb{C}}$, and $\mathfrak{t}_0$ be the Lie algebra of $\mathbf{T}'$. 
Note that $\mathfrak{t}_0=\mathfrak{t}^{d\rho_1'}$ is the $d\rho_1'$-fixed point set of $\mathfrak{t}$.
Recall that for each $\lambda\in X$, we have homomorphism $\chi_{\lambda}:\mathbf{T}_{\mathbb{C}}\rightarrow \mathbb{C}^{\times}$ which maps $\sum_{\mu\in X}n_{\mu}1_{\mu}$ to $n_{\lambda}$.
This homomorphism induces Lie algebra homomorphism $d\chi_{\lambda}:\mathfrak{t}\rightarrow\mathbb{C}$. Thus $d\chi_{\lambda}\in \mathfrak{t}^*$, and we regard $\lambda$ as an element in $\mathfrak{t}^*$ in this way.

\begin{proposition}\label{sc}
    The compact group $(\mathbf{G}_{\mathbb{C}}^{\rho_1'})^{\circ}$ is simply-connected.
\end{proposition}

\begin{proof}
By definition, the image of $\mathbf{T}'$ under the homomorphism $\chi_{\lambda}$ is in the unit circle $S^1$ of $\mathbb{C}$. Thus $\chi_{\lambda}|_{\mathbf{T}'}:\mathbf{T}'\rightarrow S^1$ is a multiplicative character of $\mathbf{T}'$. For any $H\in \mathfrak{t}_0$, we have 
\begin{align*}
    \chi_{\lambda}|_{\mathbf{T}'}(\operatorname{exp}H)=e^{d(\chi_{\lambda}|_{\mathbf{T}'})(H)} = e^{(d\chi_{\lambda})(H)} = e^{\lambda(H)}.
\end{align*}
Thus any $\lambda\in X$ is analytically integral.

Since $X$ is the weight lattice of the Lie algebra $\mathfrak{g}_{\mathbb{C}}$, by definition the group of algebraically integral forms for $(\mathbf{G}_{\mathbb{C}}^{\rho_1'})^{\circ}$ is exactly $X$.
Thus the group of analytically integral forms for $(\mathbf{G}_{\mathbb{C}}^{\rho_1'})^{\circ}$ coincides with the group of algebraically integral forms for $(\mathbf{G}_{\mathbb{C}}^{\rho_1'})^{\circ}$, and by Corollary \ref{analyint}, the fundamental group $\pi_1((\mathbf{G}_{\mathbb{C}}^{\rho_1'})^{\circ})$ is trivial.
Thus the compact group $(\mathbf{G}_{\mathbb{C}}^{\rho_1'})^{\circ}$ is simply-connected.
\end{proof}

\begin{corollary}
    The compact group $(\mathbf{G}_{\mathbb{C}}^{\rho_1'})^{\circ}$ is a universal covering group of $\mathbf{K}$.
\end{corollary}

\begin{proof}
    This follows from Prop \ref{sc} and the fact that the compact groups $(\mathbf{G}_{\mathbb{C}}^{\rho_1'})^{\circ}$ and $\mathbf{K}$ have the same Lie algebra.
\end{proof}

For each $\lambda\in X^+$, the $\mathbf{U}_{\mathbb{C}}$-module $\Lambda_{\lambda,\mathbb{C}}$ is a $\dot{\mathbf{U}}_{\mathbb{C}}$-module. 
Then it has a left $\hat{\mathbf{U}}_{\mathbb{C}}$-module structure. 
Since $(\mathbf{G}_{\mathbb{C}}^{\rho_1'})^{\circ}\subseteq \mathbf{G}_{\mathbb{C}}\subseteq \hat{\mathbf{U}}_{\mathbb{C}}$, it naturally becomes a left $\mathbf{G}_{\mathbb{C}}$-module and a left $(\mathbf{G}_{\mathbb{C}}^{\rho_1'})^{\circ}$-module. 
By taking differentials, it has a $\operatorname{Lie}(\mathbf{G}_{\mathbb{C}})$-module structure and a $\operatorname{Lie}(\mathbf{G}_{\mathbb{C}})^{d\rho_1'}$-module structure. 
We will show that this $\operatorname{Lie}(\mathbf{G}_{\mathbb{C}})$-module structure on $\Lambda_{\lambda}$ coincides with the restriction of the initial $\mathbf{U}_{\mathbb{C}}$-module structure on $\Lambda_{\lambda}$, and then obtain the following result.

\begin{proposition}
    The set $\{\Lambda_{\lambda}\mid \lambda\in X^+\}$ form a complete set of representatives of isomorphism classes of finite-dimensional irreducible representations of $(\mathbf{G}_{\mathbb{C}}^{\rho_1'})^{\circ}$. 
    Moreover, each $\Lambda_{\lambda}$ is unitary with respect to the inner product $(,)_{\lambda}$.
\end{proposition}

\begin{proof}
   For any $\lambda\in X^+$, denote the $\mathbf{G}_{\mathbb{C}}$-module structure on $\Lambda_{\lambda}$ obtained as above by $\Phi_{\lambda}:\mathbf{G}_{\mathbb{C}}\rightarrow \operatorname{GL}(\Lambda_{\lambda})$.
   We only need to check for generators of the Lie algebra.

   For any $v\in \Lambda_{\lambda},i\in I$, we have 
   \begin{align*}
    \frac{d}{dt}\Big|_{t=0} \Phi_{\lambda}(x_i(t))v=\frac{d}{dt}\Big|_{t=0}\sum_{c\in\mathbb{N},\lambda\in X}t^c\theta_i^{(c)+}1_{\lambda}v=\theta_i^+ \sum_{\lambda\in X} 1_{\lambda} v=\theta_i^+ v=E_iv,\\
    \frac{d}{dt}\Big|_{t=0} \Phi_{\lambda}(y_i(t))v=\frac{d}{dt}\Big|_{t=0}\sum_{c\in\mathbb{N},\lambda\in X}t^c\theta_i^{(c)-}1_{\lambda}v=\theta_i^- \sum_{\lambda\in X} 1_{\lambda} v=\theta_i^- v=F_iv.
   \end{align*}
   Thus $d\Phi_{\lambda}$ is exactly the restriction of the initial $\mathbf{U}_{\mathbb{C}}$-module structure. This means $(d\Phi_{\lambda},\Lambda_{\lambda})$ is an irreducible representation of $\operatorname{Lie}(\mathbf{G}_{\mathbb{C}})$.
   In this way, $(d\Phi_{\lambda},\Lambda_{\lambda}),\lambda\in X^+$ gives all the finite-dimensional irreducible representations of $\operatorname{Lie}(\mathbf{G}_{\mathbb{C}})$ (up to isomorphism).

   Since $\operatorname{Lie}(\mathbf{G}_{\mathbb{C}})^{d\rho_1'}$ is a real form of $\operatorname{Lie}(\mathbf{G}_{\mathbb{C}})$, the restriction of $d\Phi_{\lambda}$ to $\operatorname{Lie}(\mathbf{G}_{\mathbb{C}})^{d\rho_1'}$ also form an irreducible representation of $\operatorname{Lie}(\mathbf{G}_{\mathbb{C}})^{d\rho_1'}$, and $(d\Phi_{\lambda},\Lambda_{\lambda}),\lambda\in X^+$ gives all the finite-dimensional irreducible representations of  $\operatorname{Lie}(\mathbf{G}_{\mathbb{C}})^{d\rho_1'}$ (up to isomorphism).
  
   Since $(\mathbf{G}_{\mathbb{C}}^{\rho_1'})^{\circ}$ is simply-connected, and $\operatorname{Lie}(\mathbf{G}_{\mathbb{C}})^{d\rho_1'}$ is its Lie algebra, $(d\Phi_{\lambda},\Lambda_{\lambda}),\lambda\in X^+$ can be lifted to all the finite-dimensional irreducible representations of $(\mathbf{G}_{\mathbb{C}}^{\rho_1'})^{\circ}$ (up to isomorphism), which are exactly the restrictions of the $\mathbf{G}_{\mathbb{C}}$-modules $(\Phi_{\lambda},\Lambda_{\lambda})$ to $(\mathbf{G}_{\mathbb{C}}^{\rho_1'})^{\circ}$.

   With these $(\mathbf{G}_{\mathbb{C}}^{\rho_1'})^{\circ}$-module structures on each $\Lambda_{\lambda}$, the second statement follows from Lemma \ref{Ginvar}.
\end{proof}

Then we will prove the Peter-Weyl theorem for $(\mathbf{G}_{\mathbb{C}}^{\rho_1'})^{\circ}$.

Since $(\mathbf{G}_{\mathbb{C}}^{\rho_1'})^{\circ}$ is compact, it has a (normalized) Haar measure .
Recall $L^2((\mathbf{G}_{\mathbb{C}}^{\rho_1'})^{\circ})$ is the space of functions on $(\mathbf{G}_{\mathbb{C}}^{\rho_1'})^{\circ}$, which are square-integrable with respect to the Haar measure.
We denote the inner product on $L^2((\mathbf{G}_{\mathbb{C}}^{\rho_1'})^{\circ})$ by $\langle,\rangle$.

\begin{theorem}
    (Peter-Weyl)
    The two Hilbert spaces
     $L^2((\mathbf{G}_{\mathbb{C}}^{\rho_1'})^{\circ})$ and $\hat{\mathbf{O}}_{\mathbb{C}}$ coincide.
\end{theorem}

\begin{proof}
    By Schur orthogonality (Thm \ref{Schur}) and definition, we have $(f,g)=\langle f,g \rangle$ for all $f,g\in \mathbf{O}_{\mathbb{C}}$.

    Any $f=\sum_{\lambda}f_{\lambda}\in \hat{\mathbf{O}}_{\mathbb{C}}$ is a function on $(\mathbf{G}_{\mathbb{C}}^{\rho_1'})^{\circ}$.
    Since
    \begin{align*}
        \langle f,f \rangle = \langle \sum_{\lambda}f_{\lambda},\sum_{\lambda}f_{\lambda}\rangle = \sum_{\lambda}\langle f_{\lambda},f_{\lambda}\rangle = \sum_{\lambda} (f_{\lambda},f_{\lambda})=(f,f)<\infty,
    \end{align*}
    $f$ is square-integrable.
    Thus $\hat{\mathbf{O}}_{\mathbb{C}}$ is contained in $L^2((\mathbf{G}_{\mathbb{C}}^{\rho_1'})^{\circ})$, and the inner product $(,)$ on $\hat{\mathbf{O}}_{\mathbb{C}}$ is the restriction of $\langle,\rangle$.

    On the other hand, to show $L^2((\mathbf{G}_{\mathbb{C}}^{\rho_1'})^{\circ})$ is contained in $\hat{\mathbf{O}}_{\mathbb{C}}$, we first show that for any $f\in C((\mathbf{G}_{\mathbb{C}}^{\rho_1'})^{\circ})$, there exists a Cauchy sequence in $\hat{\mathbf{O}}_{\mathbb{C}}$ converges to $f$.

    Let $f\in C((\mathbf{G}_{\mathbb{C}}^{\rho_1'})^{\circ})$. 
    Since $(\mathbf{G}_{\mathbb{C}}^{\rho_1'})^{\circ}$ is compact, $f$ is uniformly continuous. 
    Thus for any $\varepsilon>0$, there exists an open neighborhood $V$ of $1\in (\mathbf{G}_{\mathbb{C}}^{\rho_1'})^{\circ}$, such that for any $x\in V$, $|l(x)f-f|_{\infty}<\varepsilon$. Here $l$ is the left translantion. 
    Let $\phi$ be a non-negative function whose support is contained in $V$ and satisfies that $\int_{(\mathbf{G}_{\mathbb{C}}^{\rho_1'})^{\circ}}\phi(x)dx=1$ and $\overline{\phi(x)}=\phi(x^{-1})$ for all $x$.
    Then the operator $\phi * -$ is a compact self-adjoint operator on $L^2((\mathbf{G}_{\mathbb{C}}^{\rho_1'})^{\circ})$, and it can be calculated that for any $x\in (\mathbf{G}_{\mathbb{C}}^{\rho_1'})^{\circ}$, $|(\phi*f)(x)-f(x)|<\varepsilon$.
    Thus $|\phi*f-f|_{\infty}<\varepsilon$.

    By Spectral Theorem \ref{Spectral}, if $a\neq 0$ is an eigenvalue of the operator $\phi * -$, then the corresponding eigenspace $V(a)$ is finite dimensional.
    All the $V(a),a\neq 0$ are mutually orthogonal, and together with $V(0)$, they span $L^2((\mathbf{G}_{\mathbb{C}}^{\rho_1'})^{\circ})$ as Hilbert space.
    For $a\neq 0$, since $V(a)$ is finite dimensional and invariant under all right translantions, by Lemma \ref{lefttrans}, the elements in $V(a)$ are all matrix coefficients, and thus $V(a)\subseteq \hat{\mathbf{O}}_{\mathbb{C}}$. 

    Let $f_a$ be the projection of $f$ on $V(a)$. Then $f=\sum f_a$, and 
    \begin{align*}
        \phi*f=\sum_{a\neq 0}\phi*f_a=\sum_{a\neq 0}af_a \in \hat{\mathbf{O}}_{\mathbb{C}}.
    \end{align*} 

    For any $k\in \mathbb{N}$, there exists a continuous function $\phi_k\in C((\mathbf{G}_{\mathbb{C}}^{\rho_1'})^{\circ})$, such that $\phi_k*f\in \hat{\mathbf{O}}_{\mathbb{C}}$ and $|\phi_k*f-f|_{\infty}<\frac{1}{2^k}$.
    Since $|\cdot|_2 \leq |\cdot|_{\infty}$, the sequence $\{\phi_k*f\}_{k\geq 1}$ is a Cauchy sequence in $\hat{\mathbf{O}}_{\mathbb{C}}$.
    Thus it converges to some $h$ in $\hat{\mathbf{O}}_{\mathbb{C}}\subseteq L^2((\mathbf{G}_{\mathbb{C}}^{\rho_1'})^{\circ})$. 
    On the other hand, $\{\phi_k*f\}_{k\geq 1}$ converges to $f$ in $L^2((\mathbf{G}_{\mathbb{C}}^{\rho_1'})^{\circ})$.
    Hence $f=h\in \hat{\mathbf{O}}_{\mathbb{C}}$, and $C((\mathbf{G}_{\mathbb{C}}^{\rho_1'})^{\circ})$ is contained in $\hat{\mathbf{O}}_{\mathbb{C}}$.
    
    Since $C((\mathbf{G}_{\mathbb{C}}^{\rho_1'})^{\circ})$ is dense in $L^2((\mathbf{G}_{\mathbb{C}}^{\rho_1'})^{\circ})$ and $\hat{\mathbf{O}}_{\mathbb{C}}$ is closed in $L^2((\mathbf{G}_{\mathbb{C}}^{\rho_1'})^{\circ})$, we have $L^2((\mathbf{G}_{\mathbb{C}}^{\rho_1'})^{\circ})=\hat{\mathbf{O}}_{\mathbb{C}}$.
\end{proof}

Then we consider the compact group $\mathbf{K}$.
We arbitrarily fix a hereditary subcategory $\mathcal{B}$ of the root category $\mathcal{R}$, and let $\{S_1,\cdots,S_m\}$ be a complete set of representatives of isomorphism classes of simple objects in $\mathcal{B}$.
The compact real form $\mathfrak{r}$ of $\mathfrak{g}_{\mathbb{C}}$ is spanned by $\{iH_X',u_X+u_{TX},i(u_X-u_{TX})\mid X\in \operatorname{ind}\mathcal{R}\}$, so it has an $\mathbb{R}$-basis $\{iH_{S_j}'\mid j=1,\cdots,m\}\bigcup\{u_X+u_{TX},i(u_X-u_{TX})\mid X\in \operatorname{ind}\mathcal{B}\}$.

Let $\mathbf{T}$ be the maximal torus of $\mathbf{K}$, which is isomorphic to $(S^1)^m$.
Let $\mathfrak{t}_0$ be the Lie algebra of $\mathbf{T}$, which has an $\mathbb{R}$-basis $\{iH_{S_j}'\mid j=1,\cdots,m\}$.
Let $\mathfrak{t}$ be the complexification of $\mathfrak{t}_0$.
For any $Y\in \operatorname{ind}\mathcal{R}$, denote the corresponding root in $\mathfrak{t}^*$ by $\zeta_Y$.
Since for any $X,Y\in \operatorname{ind}\mathcal{R}$, we have 
    \begin{align*}
        [H_X',u_Y]=-A_{XY}u_Y=-(H_X'|H_Y)u_Y.
    \end{align*}
    Thus $\zeta_Y(H)=-(H|H_Y)$ for any $H\in \mathcal{K}$.

\begin{lemma}
    The analytically integral forms for $\mathbf{K}$ are exactly the $\mathbb{Z}$-span of roots.
\end{lemma}

\begin{proof}
    Let $H=\sum_{j=1}^m b_j(iH_{S_j}')  \in \mathfrak{t}_0$, with $b_j\in \mathbb{R}$. 
    We have
    \begin{align*}
        1=\operatorname{exp}(\operatorname{ad}H)=\prod_{j=1}^m \operatorname{exp}(ib_j \operatorname{ad}H_{S_j}') = \prod_{j=1}^m h_{[S_j]}(e^{-ib_j}),
    \end{align*}
    if and only if 
    \begin{align*}
       \prod_{j=1}^{m} e^{-ib_ja_{jk}}=1 
    \end{align*}
    holds for all $k=1,\cdots,m$, if and only if $ \sum_{j=1}^m b_ja_{jk}\in 2\pi\mathbb{Z}$ holds for all $k=1,\cdots,m$.

    Let $\zeta_k$ denote the root corresponds to $S_k$, $k=1,\cdots,m$.
    Since $(a_{ij})$ is the Cartan matrix, we have $\zeta_k(H_{S_j}')=-a_{jk}$.  
    Then the condition $\operatorname{exp}(\operatorname{ad}H)=1$ is equivalent to $\zeta_k(H)\in 2\pi i\mathbb{Z}$ for all $k=1,\cdots,m$.

    Assume that $\lambda\in \mathfrak{t}^*$ is an analytically integral form. 
    We may write $\lambda=\sum_{j=1}^m c_j\zeta_j$ for some $c_j\in \mathbb{C}$.
    For each $l=1,\cdots,m$, choose $H_l\in \mathfrak{t}_0$ such that $\zeta_j(H_l)=2\pi i\delta_{jl}$, for all $j=1\cdots,m$.
    Then $H_l$ satisfies that $\operatorname{exp}(\operatorname{ad}H_l)=1$. 
    By definition, $\lambda(H_l)\in 2\pi i\mathbb{Z}$.
    On the other hand, $\lambda(H_l)=2\pi i c_l$. 
    Thus $c_l\in \mathbb{Z}$ for all $l=1,\cdots,m$ and $\lambda$ is in the $\mathbb{Z}$-span of roots.

    Conversely, by definition all roots are analytically integral. The Lemma is proved.
\end{proof}

\begin{remark}
The algebraically integral forms for $\mathbf{K}$ are not necessarily analytically integral.
For example, consider the root category of type $A_3$.
Let $H=i\pi (H_{S_1}'+H_{S_3}')$, then $\zeta_j(H)$ are in $2\pi i \mathbb{Z}$ for $j=1,2,3$, which means $\operatorname{exp}(\operatorname{ad}(H))=1$.
If $\lambda$ is algebraically integral, then it should satisfy $\lambda(H_{S_j}')\in \mathbb{Z}$ for $j=1,2,3$.
Choose $\lambda$ such that $\lambda(H_{S_1}')=1,\lambda(H_{S_2}')=\lambda(H_{S_3}')=0$.
Then we have $\lambda(H)=i\pi \notin 2\pi i\mathbb{Z} $. 
Thus $\lambda$ is not analytically integral.
\end{remark}
 
Let $\mathbf{Z}$ be the subgroup of the center of $(\mathbf{G}_{\mathbb{C}}^{\rho_1'})^{\circ}$, such that $\mathbf{K}\cong (\mathbf{G}_{\mathbb{C}}^{\rho_1'})^{\circ}/\mathbf{Z}$. 
The space $L^2(\mathbf{K})$ can be identified with the following subset of $L^2((\mathbf{G}_{\mathbb{C}}^{\rho_1'})^{\circ})$,
\begin{align*}
    \{ f\in L^2((\mathbf{G}_{\mathbb{C}}^{\rho_1'})^{\circ}) \mid f(xz)=f(x), \forall x\in (\mathbf{G}_{\mathbb{C}}^{\rho_1'})^{\circ},z\in \mathbf{Z}  \}.
\end{align*}

Let $Q^+$ be the set of dominant analytically integral forms for $\mathbf{K}$. 
Then we obtain the Peter-Weyl theorem for $\mathbf{K}$.

\begin{theorem}\label{peterK}
    (Peter-Weyl)
    The linear span of $C_{\Lambda_{\lambda}},\lambda\in Q^+$ is dense in $L^2(\mathbf{K})$.
\end{theorem}

\begin{proof}
    For any $\lambda\in X^+$ and $0\neq f_{z,z'}\in C_{\Lambda_{\lambda}}$, if $f_{z,z'}\in L^2(\mathbf{K})$, that is,
    \begin{align*}
        (xyz,z')_{\lambda}=f_{z,z'}(xy)=f_{z,z'}(x)=(xz,z')_{\lambda}
    \end{align*}
    for all $x\in (\mathbf{G}_{\mathbb{C}}^{\rho_1'})^{\circ},y\in \mathbf{Z}$.
    Since $y\in \mathbf{Z}$ acts on $\Lambda_{\lambda}$ by multiply $\chi_{\lambda}(y)$, we have $(\chi_{\lambda}(y)-1)(xz,z')_{\lambda}=0$ holds for all $x\in (\mathbf{G}_{\mathbb{C}}^{\rho_1'})^{\circ},y\in \mathbf{Z}$. 
    Thus $\chi_{\lambda}(y)=1$ for all $y\in \mathbf{Z}$ if and only if $C_{\Lambda_{\lambda}}\subseteq L^2(\mathbf{K})$.

    Since $\lambda$ is an analytically integral form for $(\mathbf{G}_{\mathbb{C}}^{\rho_1'})^{\circ}$, by definition, $\chi_{\lambda}(\mathbf{Z})=\{1\}$ if and only if $\lambda$ is an analytically integral form for $\mathbf{K}$.
    Thus $L^2(\mathbf{K})$ contains exactly all the $C_{\Lambda_{\lambda}}$ with $\lambda\in Q^+$. 
    Then the Peter-Weyl theorem for $\mathbf{K}$ follows from the Peter-Weyl theorem for $(\mathbf{G}_{\mathbb{C}}^{\rho_1'})^{\circ}$.
\end{proof}

\begin{corollary}
    The set $\{\Lambda_{\lambda}\mid \lambda \in Q^+\}$ form a complete set of representatives of isomorphism classes of finite-dimensional irreducible representations of $\mathbf{K}$.
\end{corollary}

\begin{proof}
    For $\lambda\in Q^+$, we have $\chi_{\lambda}(\mathbf{Z})=\{1\}$. Thus $\mathbf{Z}$ acts trivially on $\Lambda_{\lambda}$, and the irreducible $(\mathbf{G}_{\mathbb{C}}^{\rho_1'})^{\circ}$-module structure on $\Lambda_{\lambda}$ can be descend to an irreducible $\mathbf{K}$-module structure.

    On the other hand, by Thm \ref{peterK}, we have $\Lambda_{\lambda}, \lambda\in Q^+$ exhaust all the finite-dimensional irreducible representations of $\mathbf{K}$.
\end{proof}

\subsection{Plancherel Theorem}
The Peter-Weyl theorem in the previous subsection actually shows that $L^2((\mathbf{G}_{\mathbb{C}}^{\rho_1'})^{\circ})$ is isomorphic to $\hat{\oplus}_{\lambda\in X^+}C_{\Lambda_{\lambda}}\cong \hat{\oplus}_{\lambda\in X^+}\Lambda_{\lambda}\otimes \Lambda_{\lambda}^*$ as $(\mathbf{G}_{\mathbb{C}}^{\rho_1'})^{\circ}\times (\mathbf{G}_{\mathbb{C}}^{\rho_1'})^{\circ}$-modules.
In this subsection, we want to show that they can also be isomorphic as algebras, where the algebra structure on $L^2((\mathbf{G}_{\mathbb{C}}^{\rho_1'})^{\circ})$ given by convolution.
Similar result holds for $\mathbf{K}$.

Let $f\in L^2((\mathbf{G}_{\mathbb{C}}^{\rho_1'})^{\circ})$ and $(\pi,V_{\pi})$ be a unitary representation of $(\mathbf{G}_{\mathbb{C}}^{\rho_1'})^{\circ}$.
Define the Fourier transformation
\begin{align*}
    \hat{f}(\pi)=\int_{(\mathbf{G}_{\mathbb{C}}^{\rho_1'})^{\circ}} f(x)\pi(x^{-1})dx \in \operatorname{End}(V_{\pi}).
\end{align*}

For any $\lambda\in X^+$ and $v\otimes w\in \Lambda_{\lambda}\otimes \Lambda_{\lambda}^*$, define $T_{v,w}\in \operatorname{End}(\Lambda_{\lambda})$ by $T_{v,w}(u)=w(u)v$.
It is easy to see that $C_{\Lambda_{\lambda}}\cong \Lambda_{\lambda}\otimes \Lambda_{\lambda}^* \cong \operatorname{End}(\Lambda_{\lambda})$ as $(\mathbf{G}_{\mathbb{C}}^{\rho_1'})^{\circ} \times (\mathbf{G}_{\mathbb{C}}^{\rho_1'})^{\circ}$-modules, via $v\otimes w \mapsto f_{v,w}$ and $v\otimes w\mapsto T_{v,w}$.
This induces a $(\mathbf{G}_{\mathbb{C}}^{\rho_1'})^{\circ}\times (\mathbf{G}_{\mathbb{C}}^{\rho_1'})^{\circ}$-module isomorphism $\mathcal{F}:L^2((\mathbf{G}_{\mathbb{C}}^{\rho_1'})^{\circ})\rightarrow \hat{\oplus}_{\lambda\in X^+} \operatorname{End}(\Lambda_{\lambda})$.

\begin{lemma}
    For any $\lambda\in X^+$, denote the representation by $(\pi_{\lambda},\Lambda_{\lambda})$.
    The isomorphism $C_{\Lambda_{\lambda}}\cong \operatorname{End}(\Lambda_{\lambda})$, $f_{v,w}\mapsto T_{v,w}$, maps each $f\in C_{\Lambda_{\lambda}}$ to $(\operatorname{dim}\Lambda_{\lambda})\hat{f}(\pi_{\lambda})$.
\end{lemma}

\begin{proof}
    Take an orthonormal basis $e_1,\cdots,e_n$ of $\Lambda_{\lambda}$ with respect to $(,)_{\lambda}$ (denote it by $(,)$ for simplicity).
    Let $\pi_{ij}(-)=(\pi_{\lambda}(-)e_i,e_j)$.
    Then $\pi_{ij}$ form a basis of $C_{\Lambda_{\lambda}}$, and we have 
       $ \langle \pi_{ij},\pi_{kl} \rangle = \delta_{ik}\delta_{jl}(\operatorname{dim}\Lambda_{\lambda})^{-1}.$
    By this, any $f\in C_{\Lambda_{\lambda}}$ can be written as $f=(\operatorname{dim}\Lambda_{\lambda})\sum_{i,j}\langle f,\pi_{ij}\rangle \pi_{ij}$. 

    On the other hand, the image of $\pi_{ij}$ under the isomorphism is denoted by $T_{ij}$, which maps $v$ to $(v,e_j)e_i$. These $T_{ij}$ also form a basis of $\operatorname{End}(\Lambda_{\lambda})$, and $(T_{kl}(e_i),e_j)=\delta_{il}\delta_{kj}$.
    Since for any $f\in C_{\Lambda_{\lambda}}$,
    \begin{align*}
        (\hat{f}(\pi_{\lambda})e_i,e_j)&=\int_{(\mathbf{G}_{\mathbb{C}}^{\rho_1'})^{\circ}} f(x) (\pi_{\lambda}(x^{-1})e_i,e_j) dx=\int_{(\mathbf{G}_{\mathbb{C}}^{\rho_1'})^{\circ}} f(x) \overline{(\pi_{\lambda}(x)e_j,e_i)} dx \\
        &= \int_{(\mathbf{G}_{\mathbb{C}}^{\rho_1'})^{\circ}} f(x)\overline{\pi_{ji}(x)}dx = \langle f , \pi_{ji} \rangle,
    \end{align*}
    we have $\hat{f}(\pi_{\lambda})=\sum_{i,j} \langle f , \pi_{ij} \rangle T_{ij}$, and the lemma follows.
\end{proof}

Thus $\mathcal{F}$ maps $f=\sum_{\lambda\in X^+}f_{\lambda} $ to $\sum_{\lambda\in X^+} (\operatorname{dim}\Lambda_{\lambda}) \hat{f}(\pi_{\lambda})$.

For any $\lambda\in X^+$, define an inner product on $\operatorname{End}(\Lambda_{\lambda})$ by $(A,B)=(\operatorname{dim}\Lambda_{\lambda})^{-1}\operatorname{tr}(B^*A)$.
Then we can obtain an inner product on $ \hat{\oplus}_{\lambda\in X^+} \operatorname{End}(\Lambda_{\lambda})$.

\begin{proposition}
    (Parseval-Plancherel formula)
    The isomorphism $\mathcal{F}$ preserves the inner products.
    Let $||\cdot || $ be the norm induces by inner product, then  
    \begin{align*}
        ||f||_{L^2((\mathbf{G}_{\mathbb{C}}^{\rho_1'})^{\circ})}^2 = \sum_{\lambda\in X^+}(\operatorname{dim}\Lambda_{\lambda})^2 ||\hat{f}(\pi_{\lambda})||^2.
    \end{align*}
\end{proposition}

\begin{proof}
    We only need to show for any $\lambda\in X^+$ and $f_{v_1,w_1},f_{v_2,w_2}\in C_{\Lambda_{\lambda}}$.
    Let $e_1,\cdots,e_n$ be an orthonormal basis of $\Lambda_{\lambda}$ with respect to $(,)$.
    We have
    \begin{align*}
        &(\mathcal{F}f_{v_1,w_1},\mathcal{F}f_{v_2,w_2})=(T_{v_1,w_1},T_{v_2,w_2})= (\operatorname{dim}\Lambda_{\lambda})^{-1} \sum_i ((e_i,w_1)v_1,(e_i,w_2)v_2) \\
        &= (\operatorname{dim}\Lambda_{\lambda})^{-1} (v_1,v_2)(\sum_i(w_2,e_i)e_i,w_1) = \frac{(v_1,v_2)(w_2,w_1)}{\operatorname{dim}\Lambda_{\lambda}}=\langle f_{v_1,w_1},f_{v_2,w_2} \rangle.
    \end{align*}
    Thus $\mathcal{F}$ preserves the inner products. 
    The Parseval-Plancherel formula is a direct corollary.
\end{proof}

Define the multiplication on $\hat{\oplus}_{\lambda\in X^+}\operatorname{End}(\Lambda_{\lambda})$ by 
\begin{align*}
    (\sum_{\lambda\in X^+}A_{\lambda})(\sum_{\lambda\in X^+} B_{\lambda}) = \sum_{\lambda\in X^+}(\operatorname{dim}\Lambda_{\lambda})^{-1}A_{\lambda}B_{\lambda}.
\end{align*}

\begin{theorem}
    (Plancherel)
    The map $\mathcal{F}:L^2((\mathbf{G}_{\mathbb{C}}^{\rho_1'})^{\circ})\rightarrow \hat{\oplus}_{\lambda\in X^+} \operatorname{End}(\Lambda_{\lambda})$ is an algebra isomorphism.
\end{theorem}

\begin{proof}
    We only need to show $\mathcal{F}$ is an algebra homomorphism.
    It can be easily calculated that for $f,g\in L^2((\mathbf{G}_{\mathbb{C}}^{\rho_1'})^{\circ})$, we 
     $\hat{(f*g)}(\pi_{\lambda})=\hat{f}(\pi_{\lambda})\hat{g}(\pi_{\lambda})$.
    Then
    \begin{align*}
        \mathcal{F}(f*g)&=\sum_{\lambda\in X^+}(\operatorname{dim}\Lambda_{\lambda})\hat{f}(\pi_{\lambda})\hat{g}(\pi_{\lambda})\\
        &=(\sum_{\lambda\in X^+}(\operatorname{dim}\Lambda_{\lambda})\hat{f}(\pi_{\lambda}))(\sum_{\lambda\in X^+}(\operatorname{dim}\Lambda_{\lambda})\hat{g}(\pi_{\lambda}))=\mathcal{F}(f)\mathcal{F}(g).
    \end{align*}
  Thus proved the theorem.
\end{proof}

Lusztig proved that there is an algebra isomorphism $\dot{\mathbf{U}}[\geq \lambda]/\dot{\mathbf{U}}[>\lambda]\rightarrow \operatorname{End}(\Lambda_{\lambda})$ for each $\lambda\in X^+$.
Now we want to show that this algebra isomorphism is also a $\mathbf{G}_{\mathbb{C}}\times \mathbf{G}_{\mathbb{C}}$-module isomorphism.

\begin{lemma}
    The left or right multiplication by $\mathbf{G}_{\mathbb{C}}$ preserves $\dot{\mathbf{U}}[\geq \lambda]/\dot{\mathbf{U}}[>\lambda]$, which makes $\dot{\mathbf{U}}[\geq \lambda]/\dot{\mathbf{U}}[>\lambda]$ a $\mathbf{G}_{\mathbb{C}}\times \mathbf{G}_{\mathbb{C}}$-module.
\end{lemma}

\begin{proof}
    For $\sum_{a\in \dot{\mathbf{B}}}n_a a\in \mathbf{G}_{\mathbb{C}}$ and $b\in \dot{\mathbf{B}}[\lambda]$, 
   since $ \dot{\mathbf{U}}[\geq\lambda]$ is a two-sided ideal, for any $a\in \dot{\mathbf{B}}$, we have $ab\in \dot{\mathbf{U}}[\geq\lambda]$. 
   Moreover, since $\dot{\mathbf{B}}\bigcap \dot{\mathbf{U}}[\geq\lambda]$ form a basis of $\dot{\mathbf{U}}[\geq\lambda]$, if $m_{ab}^c\neq 0$, then $c\in \dot{\mathbf{B}}\bigcap \dot{\mathbf{U}}[\geq\lambda]=\bigcup_{\lambda_1\geq \lambda}\dot{\mathbf{B}}[\lambda_1]$.
   That is to say, for any fixed $c\in \dot{\mathbf{B}}$, if there exists $a$ such that $m_{ab}^c \neq 0$, then we have $c\in \bigcup_{\lambda_1\geq \lambda}\dot{\mathbf{B}}[\lambda_1]$, and 
   \begin{align*}
      (\sum_{a\in \dot{\mathbf{B}}}n_a a)b=\sum_{c\in \bigcup_{\lambda_1\geq \lambda}\dot{\mathbf{B}}[\lambda_1]}(\sum_{a\in \dot{\mathbf{B}}}m_{ab}^c n_a)c.
   \end{align*}
   Modulo both sides by $\dot{\mathbf{U}}[>\lambda]$, then 
   \begin{align*}
    (\sum_{a\in \dot{\mathbf{B}}}n_a a)b\equiv \sum_{c\in \dot{\mathbf{B}}[\lambda]}(\sum_{a\in \dot{\mathbf{B}}}m_{ab}^c n_a)c
   \end{align*}
   in $\dot{\mathbf{U}}[\geq \lambda]/\dot{\mathbf{U}}[>\lambda]$. Since $\dot{\mathbf{B}}[\lambda]$ is a finite set, the left multiplication of $\mathbf{G}_{\mathbb{C}}$ preserves $\dot{\mathbf{U}}[\geq \lambda]/\dot{\mathbf{U}}[>\lambda]$. Similar for right multiplications.
   Thus $\dot{\mathbf{U}}[\geq \lambda]/\dot{\mathbf{U}}[>\lambda]$ becomes a $\mathbf{G}_{\mathbb{C}}\times \mathbf{G}_{\mathbb{C}}$-module by $(g_1,g_2)b=g_1bg_2^{-1}, b\in \dot{\mathbf{U}}[\geq \lambda]/\dot{\mathbf{U}}[>\lambda],g_1,g_2\in \mathbf{G}_{\mathbb{C}}$.

   Recall that the $\mathbf{G}_{\mathbb{C}}\times \mathbf{G}_{\mathbb{C}}$-module structure on $\operatorname{End}(\Lambda_{\lambda})$ is given by $(g_1,g_2)A=g_1Ag_2^{-1}, A\in \operatorname{End}(\Lambda_{\lambda}),g_1,g_2\in \mathbf{G}_{\mathbb{C}}$, the algebra isomorphism $\dot{\mathbf{U}}[\geq \lambda]/\dot{\mathbf{U}}[>\lambda]\rightarrow \operatorname{End}(\Lambda_{\lambda})$ is a $\mathbf{G}_{\mathbb{C}}\times \mathbf{G}_{\mathbb{C}}$-module isomorphism.
\end{proof}

Define the multiplication on $\hat{\oplus}_{\lambda\in X^+}\dot{\mathbf{U}}[\geq \lambda]/\dot{\mathbf{U}}[>\lambda]$ by 
\begin{align*}
    (\sum_{\lambda\in X^+}u_{\lambda})(\sum_{\lambda\in X^+} u'_{\lambda}) = \sum_{\lambda\in X^+}(\operatorname{dim}\Lambda_{\lambda})^{-1}u_{\lambda}u'_{\lambda}.
\end{align*}

\begin{corollary}
    We have a $(\mathbf{G}_{\mathbb{C}}^{\rho_1'})^{\circ}\times (\mathbf{G}_{\mathbb{C}}^{\rho_1'})^{\circ}$-module as well as algebra isomorphism
    \begin{align*}
        L^2((\mathbf{G}_{\mathbb{C}}^{\rho_1'})^{\circ})\cong \hat{\oplus}_{\lambda\in X^+}\dot{\mathbf{U}}[\geq \lambda]/\dot{\mathbf{U}}[>\lambda].
    \end{align*}
\end{corollary}

Replacing $(\mathbf{G}_{\mathbb{C}}^{\rho_1'})^{\circ}$ by $\mathbf{K}$ and $X^+$ by $Q^+$, we obtain all the above results for $\mathbf{K}$.

\subsection*{Acknowledgments}
We would like to thank Yixin Lan for his discussion.

\nocite{GTM235}

%\bibliography{mybibfilelby}

\end{document}